\documentclass[reqno,10pt]{amsart} 

\usepackage{amsthm,amsmath,amsfonts,amssymb,stmaryrd,mathrsfs,enumitem,
hyperref, graphicx}

\newcommand{\znd}{\mathbb{Z}^d_n}

\newtheorem{theorem}{Theorem}
\newtheorem{lemma}{Lemma}[section]
\newtheorem{proposition}[lemma]{Proposition}
\newtheorem{corollary}[lemma]{Corollary}
\newtheorem{definition}[lemma]{Definition}
\newtheorem{assumption}[lemma]{Assumption}
\newtheorem{remark}[lemma]{Remark}

\newcommand{\CC}{\mathbb{C}}

\newcommand{\RR}{\mathbb{R}}
\newcommand{\NN}{\mathbb{N}}
\newcommand{\ZZ}{\mathbb{Z}}

\newcommand{\EE}{\mathbb{E}}
\newcommand{\QQ}{\mathbb{Q}}
\newcommand{\PP}{\mathbb{P}}
\newcommand{\TT}{\mathbb{T}}
\newcommand{\DD}{\mathbb{D}}

\newcommand{\mL}{\mathcal{L}}

\newcommand{\mP}{\mathcal{P}}
\newcommand{\mQ}{\mathcal{Q}}
\newcommand{\mC}{\mathcal{C}}

\newcommand{\mM}{\mathcal{M}}

\newcommand{\mF}{\mathcal{F}}
\newcommand{\mD}{\mathcal{D}}
\newcommand{\mX}{\mathcal{X}}

\newcommand{\mS}{\mathcal{S}}
\newcommand{\mB}{\mathcal{B}}
\newcommand{\mA}{\mathcal{A}}
\newcommand{\mH}{\mathcal{H}}

\newcommand{\LLL}{\mathscr{L}}

\newcommand{\mf}[1]{\mathfrak{#1}}

\renewcommand{\l}{\ell}

\newcommand{\para}{\varolessthan}
\newcommand{\rpara}{\varogreaterthan}
\newcommand{\reso}{\varodot}

\newcommand{\hh}{\frac{1}{2}}
\newcommand{\ve}{\varepsilon}
\newcommand{\vr}{\varrho}
\newcommand{\vt}{\vartheta}

\newcommand{\bigslant}[2]{{\raisebox{.1em}{$#1$}\left/\raisebox{-.1em}{$#2$}\right.}}
\newcommand{\lqm}{``}
\newcommand{\rqm}{'' }

\newcommand*{\ud}{\mathrm{\,d}}
\newcommand{\supp}{\operatorname{supp} } 

\makeatletter
\newcommand*{\mint}[1]{%
  \mint@l{#1}{}%
}
\newcommand*{\mint@l}[2]{%
  \@ifnextchar\limits{%
    \mint@l{#1}%
  }{%
    \@ifnextchar\nolimits{%
      \mint@l{#1}%
    }{%
      \@ifnextchar\displaylimits{%
	\mint@l{#1}%
      }{%
	\mint@s{#2}{#1}%
      }%
    }%
  }%
}
\newcommand*{\mint@s}[2]{%
  \@ifnextchar_{%
    \mint@sub{#1}{#2}%
  }{%
    \@ifnextchar^{%
      \mint@sup{#1}{#2}%
    }{%
      \mint@{#1}{#2}{}{}%
    }%
  }%
}
\def\mint@sub#1#2_#3{%
  \@ifnextchar^{%
    \mint@sub@sup{#1}{#2}{#3}%
  }{%
    \mint@{#1}{#2}{#3}{}%
  }%
}
\def\mint@sup#1#2^#3{%
  \@ifnextchar_{%
    \mint@sup@sub{#1}{#2}{#3}%
  }{%
    \mint@{#1}{#2}{}{#3}%
  }%
}
\def\mint@sub@sup#1#2#3^#4{%
  \mint@{#1}{#2}{#3}{#4}%
}
\def\mint@sup@sub#1#2#3_#4{%
  \mint@{#1}{#2}{#4}{#3}%
}
\newcommand*{\mint@}[4]{%
  \mathop{}%
  \mkern-\thinmuskip
  \mathchoice{%
    \mint@@{#1}{#2}{#3}{#4}%
	\displaystyle\textstyle\scriptstyle
  }{%
    \mint@@{#1}{#2}{#3}{#4}%
	\textstyle\scriptstyle\scriptstyle
  }{%
    \mint@@{#1}{#2}{#3}{#4}%
	\scriptstyle\scriptscriptstyle\scriptscriptstyle
  }{%
    \mint@@{#1}{#2}{#3}{#4}%
	\scriptscriptstyle\scriptscriptstyle\scriptscriptstyle
  }%
  \mkern-\thinmuskip
  \int#1%
  \ifx\\#3\\\else_{#3}\fi
  \ifx\\#4\\\else^{#4}\fi  
}
\newcommand*{\mint@@}[7]{%
  \begingroup
    \sbox0{$#5\int\m@th$}%
    \sbox2{$#5\int_{}\m@th$}%
    \dimen2=\wd0 %
    \let\mint@limits=#1\relax
    \ifx\mint@limits\relax
      \sbox4{$#5\int_{\kern1sp}^{\kern1sp}\m@th$}%
      \ifdim\wd4>\wd2 %
	\let\mint@limits=\nolimits
      \else
	\let\mint@limits=\limits
      \fi
    \fi
    \ifx\mint@limits\displaylimits
      \ifx#5\displaystyle
	\let\mint@limits=\limits
      \fi
    \fi
    \ifx\mint@limits\limits
      \sbox0{$#7#3\m@th$}%
      \sbox2{$#7#4\m@th$}%
      \ifdim\wd0>\dimen2 %
	\dimen2=\wd0 %
      \fi
      \ifdim\wd2>\dimen2 %
	\dimen2=\wd2 %
      \fi
    \fi
    \rlap{%
      $#5%
	\vcenter{%
	  \hbox to\dimen2{%
	    \hss
	    $#6{#2}\m@th$%
	    \hss
	  }%
	}%
      $%
    }%
  \endgroup
}

\makeatletter
\newcommand{\opnorm}{\@ifstar\@opnorms\@opnorm}
\newcommand{\@opnorms}[1]{%
  \left|\mkern-1.5mu\left|\mkern-1.5mu\left|
   #1
  \right|\mkern-1.5mu\right|\mkern-1.5mu\right|
}
\newcommand{\@opnorm}[2][]{%
  \mathopen{#1|\mkern-1.5mu#1|\mkern-1.5mu#1|}
  #2
  \mathclose{#1|\mkern-1.5mu#1|\mkern-1.5mu#1|}
}
\makeatother

\begin{document}

\title{The spatial $\Lambda$--Fleming--Viot process in a random environment}

\maketitle
\thispagestyle{empty}

\vspace{-0.5cm}

\centerline{\sc Tommaso Rosati\footnote{Department of Mathematics, Imperial
College London, GB, \texttt{t.rosati@imperial.ac.uk}} and Aleksander
Klimek{\footnote{School of Mathematics, University of Edinburgh, GB, {\tt aleksander.klimek@ed.ac.uk}}}}
\renewcommand{\thefootnote}{}

\bigskip

\vspace{.5cm} 
 



\begin{abstract}
We study the large scale behaviour of a population consisting of two types
which evolve in dimension $d=1,2$ according to a spatial Lambda-Fleming-Viot
process subject to random time-independent selection.
    If one of the two types is rare compared to the other, we prove that its
    evolution can be approximated by a super-Brownian motion in a random
    (and singular) environment.
	Without the sparsity assumption, a diffusion approximation leads to a Fisher-KPP equation in a random potential.
    The proofs build on two-scale Schauder estimates and semidiscrete approximations of the Anderson Hamiltonian.
\end{abstract}





\maketitle

\tableofcontents

\section*{Introduction}

A fundamental challenge in population genetics is to understand the interplay between different evolutionary and ecological factors and their overall contribution to genetic variety, i.e.\ the distribution of different types within a population.
A prominent example of such a force is the random neutral process of `genetic
drift', which occurs due to the random reproduction of organisms. 
Another one is the adaptive process of selection.
Both genetic drift and selection work, in different ways, to reduce the genetic variability of populations. 
However, other ecological and evolutionary forces may counterbalance those factors and explain durable heterogeneity within the populations.

Starting with the pioneering works by Wright \cite{wright:1943}, spatial
structure has  played a key role in understanding genetic diversity. Since
individuals inhabit different, possibly distant geographical regions and do not
move too far from their place of birth, the likelihood of mating between
geographically distant populations is very small. This leads to a greater
differentiation between subpopulations, as distant individuals evolve
essentially independently of each other. In extreme cases, this mechanism,
which is usually referred to as isolation by distance may even lead to the creation of different species. 
Even though, in principle, selection acts to reduce the genetic variety, Wright
argued in the same article that if the selection is spatially heterogeneous,
that is, if selection favors different types of individuals in different
regions in space, it may further enhance the differentiation coming from
isolation by distance. A large body of empirical evidence suggests that this
may indeed be the case. Studies on plants
\cite{pausas/carreras/ferre/font:2003}, bacteria \cite{rainey/travisano:1998},
animals \cite{kerr/packer:1997} seem to all confirm that inhomogeneity in the
spatial environmental enhances diversity. For more in-depth description of biological literature, including less favorable viewpoints of the phenomena we are concerned with, we refer to \cite{tews/etal:2004, hedrick:2006, stein/gerstner/kreft:2014}.

Our work is similarly motivated by the question: 
 does spatially heterogeneous selection enhance the genetic diversity?

There are many approaches one could take to model a spatially structured population.
The stepping stone models (see e.g. \cite{kimura:1953}), where the population
evolves in separated islands distributed on a lattice and interacts only with
neighboring islands, lead to an artificial subdivision of the population.
Approaches based around the Wright–Mal\'ecot formula \cite{barton/depaulis/etheridge:2002, malecot:1948, wright:1943} (which was introduced to study the isolation by distance phenomena) suffer from either inconsistencies in their assumptions or lead to unnatural `clumping' of the population. 
We refer to \cite{barton/etheridge/veber:2013} for an overview of the difficulties associated with modelling spatially distributed populations.
The spatial Lambda-Fleming-Viot (SLFV) class of models, introduced in \cite{etheridge:2008} and formally constructed in \cite{barton/etheridge/veber:2010}, has been proposed specifically to overcome those difficulties, and is at the basis of our work.
Here the population is distributed over continuous space, and
reproductive events involve macroscopic regions of space (in this work balls of
a fixed radius $ 1/n$, for $ n \in \NN $) and follow a space-time Poisson point process. 

In the neutral SLFV there is no bias in the relative fitness of the populations
at hand. Our work considers instead the case in which the population consists
of just two types ($\mf{a}$ and $\mf{A}$) and their relative fitness is
modeled by a sign changing selection coefficient $s_n(x), x \in \TT^d$ ($
\TT^{d} $ being the $d$-dimensional torus), so that $\mf{a}$ is favored in the
location $x$ if $s_n(x)>0$ and $\mf{A}$ is favored in the opposite case.
Instead of choosing a specific selection coefficient, we sample it from a
probability distribution $\PP$. We will consider the proportion $X^n(\omega, t,
x)$, evaluated at time $t \geq 0$ and position $x \in \TT^d$, of particles of
type $\mf{a}$ with respect to the total population, given the realization
$s_n(\omega)$ of the selection coefficient. The parameter $n \in \NN$ indicates
the size of the impact area of reproductive events: we are interested in the
limit $n \to \infty$ and will scale the magnitude of the reproductive events and the strength of the selection coefficient $s_n$  according to $n$ as well. All our scaling limits are diffusive and the effect of selection is \textit{weak} with respect to neutral events.

We study two different scenarios. In the first one, we assume that type
$\mf{a}$ is rare compared to $\mf{A}$. The rarity is described by considering
an initial condition $X^n(\omega, 0,x)$ of order $n^{-\vr}$ for certain values
of $\vr>0$, so that the process will be of order $ n^{-\varrho} $ for very long
times. In this scenario $\mf{a}$ represents a mutation which tries to establish itself among the wild type $\mf{A}$. Just as a small sub-population in the Wright-Fisher model is described by a branching process, we expect the limit to be a superBrownian motion (see \cite{Etheridge2000} for an introduction to superprocesses) in a random time-independent environment.
A similar scaling result without selection was first obtained by \cite{chetwynd-diggle/etheridge:2017} (see also \cite{Cox-Durrett-Perkins2000RescaledVoter} for an analogous result regarding the voter model) and recently extended in \cite{cox/perkins:2019} to critical values of the parameter $\vr$. A scaling limit for a model with a selection coefficient which is white in time and correlated in space, was obtained by \cite{chetwynd-diggle/klimek:2019} using a lookdown representation.

We will assume, instead, that $s_n$ scales to a spatial white noise $\xi$ on the torus $\TT^d$ and consider only dimension $d=1,2$. In this setting, the limit (cf. Theorem~\ref{thm:convergence_rsbm}) is the rough superBrownian motion introduced in \cite{PerkowskiRosati2019RSBM}, which formally solves the following stochastic partial differential equation (SPDE) for some $\nu_0>0$ (in $d=2$ the SPDE has to be replaced by the associated martingale problem):
\begin{equation}\label{eqn:rsbm-introduction} 
\partial_t Y  = \nu_0 \Delta Y + (\xi - \infty 1_{\{d=2\}}) Y + \sqrt{Y} \widetilde{\xi}, \qquad Y(0) = Y_0.
\end{equation}
Here $\widetilde{\xi}$ is a space-time white noise independent of $\xi$. The
$\infty$ appearing in $d=2$ is a formal representation of the renormalisation
required to make sense of the Anderson mode (see the discussion below), which is described by the SPDE
\begin{equation}\label{eqn:pam-introduction} \partial_t Y = \nu_0 \Delta Y + (\xi - \infty 1_{\{d=2\}}) Y, \qquad Y(0) = Y_0.\end{equation}
The latter equation is \textit{singular} in $d=2$ because the expected
regularity of the solution $Y$ is not sufficient to make sense of the product
$\xi \cdot Y$ and requires theories such as regularity structures or
paracontrolled distributions (cf. Section~\ref{sec:semidiscrete_PAM} or see
\cite{Hairer2014, GubinelliPerkowski2017KPZ} for complete works on singular
SPDEs). In particular, there is no understanding of the Anderson model in
dimension $d \geq 4$. We restrict to $d \leq 2$ as these are the biologically
interesting cases and in $d=3$ the renormalization procedure is more involved. The quoted
solution theories for singular SPDEs work pathwhise, conditional on the
realization of the noise $\xi$ and some functionals thereof. As a consequence,
solutions to \eqref{eqn:rsbm-introduction} are defined as martingale solutions
conditional on the realization of $\xi$ and uniqueness in distribution of
solutions to \eqref{eqn:rsbm-introduction} is proven through a conditional
duality argument. This lies in contrast to cases where the environment is white in time \cite{Mytnik1996}, where the martingale term can contain also the environment.

A crucial step in the proof of the scaling limit is to show that the continuous
Anderson Hamiltonian $\mH= \nu_0 \Delta + \xi - \infty1_{\{d=2\}}$ is the limit
of approximations $\mH_n = \mA_n + \xi_n - c_n 1_{\{d=2\}}$ (cf. Theorem~\ref{thm:semidiscrete-pam-approximation}). In the latter operator, the approximate Laplacian $\mA_n$ acts on $L^2(\TT^d)$ and is expressed in terms of local averages of functions: we call this setting semidiscrete, as opposed to the fully discrete setting, where the underlying space is for example a lattice. Fully discrete approximations of singular SPDEs have been the object of many studies (see \cite{MourratWeber2017, ErhardHairer2019Discretizations, ChoukGairingPerkowski2017, MartinPerkowski2019Bravais} for a partial literature). Instead, approximations in the present semidiscrete case appear new. In the study of such SPDEs the smoothing effect of the Laplacian is crucial: the first step towards understanding the convergence of the operators is to establish the regularization properties of the approximate Laplacian $\mA_n$, commonly known as Schauder estimates. Through a two-scale argument, we separate macroscopic scales in frequency space, at which $\mA_n$ regularizes analogously to the Laplacian, and microscopic scales, which are small but see no regularization (see Theorem~\ref{thm:regularization-estimates-main-results}). Once we are provided with the Schauder estimates and the convergence of $\mH_n$, the scaling limit is proven through an application of the Krein-Rutman theorem. At this point it is particularly important that the space is compact, while all other results in this work seem to extend from $\TT^d$ to $\RR^d$.

In the second scenario, $s_n$ is chosen to scale to a smooth random function
$\overline{\xi}$, and we do not make the sparsity assumption. 
This regime corresponds to studying the long time behaviour of a large population.
In this case under diffusive scaling one obtains (cf. Theorem~\ref{thm:convergence_fkpp}) convergence to a solution of the (in $d=1$ stochastic) Fisher-KPP equation
\begin{equation}\label{eqn:fkpp-introduction} 
\partial_t X = \nu_0 \Delta X + \overline{\xi} X(1-X)+ \sqrt{X(1-X)} \widetilde{\xi} 1_{\{d=1\}}, \qquad X(0) = X_0.
\end{equation}
As before $\widetilde{\xi}$ is a space-time white noise independent of
$\overline{\xi}$. In a nutshell, the intensity of the martingale term is
governed by a parameter $\eta \geq 0$ and there exists a critical value $\eta_c
(d) \geq 0$ such that the martingale term is of order $n^{-(\eta - \eta_c)}$. In dimension $d=1$ we consider $\eta = \eta_c$, while in dimension $d=2$ we take $\eta > \eta_c$. In some models, by taking into account dual processes, cf. \cite{etheridge:2008, forien/penington:2017} , one can prove that in $d=2$ the deterministic limit holds also \textit{at} the critical value. 
To the best of our knowledge the process we consider does not have a dual: hence although a similar result is expected, it remains open as the quoted methods do not apply.
Due to the same lack of duality, in $d=1$ uniqueness of the solutions seems out of reach. Similar results have been obtained in \cite{etheridge/veber/yu:2014} where the selection coefficient is constant in space and time (in this case the process admits a dual) and in \cite{biswas/etheridge/klimek:2018}, where the selection coefficient is fluctuating in time and space and correlated in the latter, giving rise to an additional martingale term. 

The treatment of this second regime is apparently much simpler, as the solution
is bounded between $0$ and $1$. The only difficulty is to prove convergence in
a topology, in which one can pass to the limit inside the nonlinearity. Unlike
the previous works \cite{etheridge/veber/yu:2014, biswas/etheridge/klimek:2018}
we can make good use of the Schauder estimates and directly prove tightness for
a smoothed version of $X^{n}$ in a Sobolev space of positive regularity (see Theorem~\ref{thm:convergence_fkpp}). 

In conclusion, this work extends previous scaling limits to incorporate a sign changing, possibly rough, selection. Choosing the selection at random provides a natural setting which exhibits interesting longtime behavior. We believe this could be the starting point for some ulterior studies: for example Equation~\eqref{eqn:fkpp-introduction} in $d=1$ with $x \in \RR$ (so globally in space) can be recovered with the same methods and could have interesting longtime properties, as the selection could balance out the genetic drift. The methods we used are based on two-scale Schauder estimates and do not rely on duality. They allow us to establish a connection to singular SPDEs, but appear to be a fairly simple, powerful tool to treat nonlinearities appearing in the SLFV.

\subsection*{Discussion} Let us briefly come back to our first question,
concerning the influence of spatially heterogeneous selection on genetic
diversity. Our scaling limits show how strong the effect of weak selection can
be, even in presence of genetic drift. One observation is that the degree of
spatial heterogeneity, encoded in the regularity of the environment, plays a
key r\^ole. In fact, we see that if the environment is sufficiently irregular, a
rare type will invade the population in finite time, in our parabolic scale:
this is captured by the necessity of \emph{renormalisation}. Our analysis then
moves on to finer scales. Once we remove this first impact via an appropriate
exponential damping -- captured by the sequence of renormalisation constants --
we obtain \eqref{eqn:rsbm-introduction}: here the survival of the rare type is
determined by the longtime behavior of the Anderson model, despite the presence
of genetic drift \cite{PerkowskiRosati2019RSBM} and becomes ever more likely on
large domains, since the top eigenvalue of the Anderson Hamiltonian tends to
increase with the volume size \cite{Chouk2019} (for the 2D case). Yet, if we
pass to the nonlinear regime \eqref{eqn:fkpp-introduction} the effect of weak
selection is not strong enough to overcome genetic drift
(at least not on finite volume) and we expect fixation to one of the stable
states in finite time. Instead, without the presence of genetic drift, we will
observe longtime coexistence (that is, a limiting stable fixed point that is
neither $ X\equiv 0 $ not $ X \equiv 1 $) as long as both the linearization at
$ 0 $ and $ 1 $ are unstable: so once more the dynamic is governed by the top
eigenvalue of the Anderson Hamiltonian (with smooth noise), see
\cite[Theorem 10.1.5]{Henry1981GeometricTheoryPDEs}.

\subsection*{Structure of the paper}
In Section~\ref{sec:models_statment}, we introduce the model and state our main results.
    Section~\ref{sec:scaling_to_rsbm} is devoted to the relation between the
    Spatial Lambda-Fleming-Viot process with fluctuating selection and the
    rough super-Brownian motion, whereas in Section~\ref{sec:scaling_fkpp} we
    establish the scaling limit to the Fisher-KPP equation in rough potential. 
The rest of the paper is devoted to analytical backbone of our results.
 Section~\ref{sec:schauder_estimates} covers Schauder estimates
and Section~\ref{sec:semidiscrete_PAM} discusses analytic and probabilistic aspects of the Anderson model.

\subsection*{Acknowledgements}
{
We would like to thank Nicolas Perkowski for many helpful discussions and
comments, Guglielmo Feltrin for an enlightening conversation and the anonymous
referees for pointing out certain mistakes and their numerous suggestions to
improve the article. 

AK acknowledges that the funding for this work has been provided by the Alexander von Humboldt Foundation in the framework of the Sofja Kovalevskaja Award endowed by the German Federal Ministry of Education and Research. 

TR gratefully acknowledges support by the IRTG 1740: this paper was developed within the scope of the IRTG 1740/TRP2015/50122-0, funded by the DFG/FAPESP.
}

\section{Models and statement of main results}
\label{sec:models_statment}
We begin with introducing our notation. 
In Subsection~\ref{sec:slfv_definition} we describe the
Spatial-Lambda-Fleming-Viot process. 
In Subsection~\ref{sub:sparse_regime} we study the convergence to the rough
super-Brownian motion.  Subsection~\ref{sub:fkpp} is devoted to the diffusive scaling which leads to Fisher-KPP equation. In Subsection~\ref{subsec:proof_methods} we describe the main analytical components of the proofs.

\subsection{Notation}
We write \(\NN= \{0, 1,2, \ldots\}\), and
\(\RR_{+} = [0, \infty)\). 
The \(d-\)dimensional torus $\TT^{d}$ is defined as \(\TT^{d} =
\bigslant{\RR^{d}}{\ZZ^{d}}\), where \(\ZZ^{d}\) acts by translation
on $ \RR^{d} $. 

We indicate with \(|A|\) the Lebesgue
measure of a Borel set \(A \subseteq \TT^{d}\).
Let \( \overline{B}_{n}(x) \subseteq
\RR^{d}\) be the ball (with respect to the Euclidean norm) of volume \(n^{-d}\)
around \(x\). Similarly, let \(\overline{Q}_{n}(x) \subset \RR^d\) be the $d$-dimensional cube
\[ 
  y \in \overline{Q}_n(x) \iff (y {-} x)_{i} \in [ - (2n)^{-1} ,
  (2n)^{-1} ), \qquad \forall i \in \{1, \ldots ,
  d\}. 
\]
As we work on the \( d- \)dimensional torus we denote with
\[ B_{n} (x) = \bigslant{\overline{B}_{n}(x)}{\ZZ^{d}} \subseteq
\TT^{d}, \quad Q_{n}(x) =\bigslant{ \overline{Q}_{n}(x)}{\ZZ^{d}} \subseteq
\TT^{d} \] 
the projections of \( \overline{Q}_{n}, \overline{B}_{n} \) on the torus.
To make sure that these still satisfy the normalization
\[ |B_{n}(x)| = |Q_{n}(x)|= n^{-d},\]
 observe that for every \( d \in \NN
\) there exists a \( c(d) \in \NN \) such that
\[ \overline{B}_{n}(0), \overline{Q}_{n}(0) \subseteq ( -
1/2 , 1/2)^{d}, \qquad \forall n \geqslant c(d).\] 
For this reason, throughout this work we consider only
$ n \geqslant c(d). $ 
We will not repeat this assumption to avoid an additional burden on the
notation. Next, consider the lattice \[\ZZ_{n}^{d} = \big( n^{-1} 
\ZZ^{d} \big) \cap \TT^{d}.\]
Since $ \{Q_n(x)\}_{x \in \ZZ_{n}^{d} \cap \TT^d}$ is a collection of
disjoint sets, the torus is the disjoint union
\[\TT^{d} = \bigcup_{x \in \ZZ_{n}^{d} \cap \TT^d} Q_{n}(x).\]
For integrable \(w \colon \TT^{d} \to \RR\) define  \( \Pi_{n} w(x)\) as an average integral of $w$ over $B_{n}(x)$, that is
\[ 
  \Pi_n w (x) : = \mint{-}_{B_{n}(x)} w(y) \ud y  : =
  \frac{1}{|B_{n} (x)|} \int_{B_{n}(x)} w(y) \ud y. 
\]

Next, we make often use of the Fourier transform both on the torus
and in the full space. We denote with $ \mS(\TT^{d}), \mS^{\prime}
(\TT^{d}) $ the space of (smooth) Schwartz test functions and (its dual) of Schwartz
distributions respectively.
For \(\varphi \in \mS^{\prime}(\TT^{d})\), we define
\begin{align*}
  \widehat{\varphi} (k) = \mF_{\TT^{d}} \varphi (k) = \int_{\TT^{d}} e^{- 2 \pi
  \iota k \cdot x} \varphi (x) \ud x, \ \ k \in \ZZ^{d}.
\end{align*}
Analogously, for \(\psi \in \mS^{\prime}(\RR^{d}): \mF_{\RR^{d}} \psi (k) = \int_{\RR^{d}} e^{- 2 \pi \iota k
  \cdot x} \psi(x) \ud x\), for \(k \in \RR^{d}\). These Fourier transforms admit inverses, which we denote with
\(\mF_{\TT^{d}}^{-1}, \mF_{\RR^{d}}^{-1}\) respectively.

 
For   \(a \colon \ZZ^{d} \to \RR\) with at most polynomial growth we define the Fourier multiplier as an operator of the form
\[ a(D) \varphi = \mF_{\TT^{d}}^{-1} \big( a ( \cdot) \mF_{\TT^{d}} \varphi
( \cdot)\big), \qquad \forall \varphi \in
\mS^{\prime}(\TT^{d}).\] 

Since characteristic functions normalized to integrate to $1$ enter  the calculations repeatedly, for a set $A$
we write:
\[ \chi_{A}(x) = \frac{1}{|A|} 1_{A}(x).\] 
In the special case of balls and cubes we additionally define for $ x \in
\TT^{d} $ and $ k \in \ZZ^{d} $
\begin{align*}
 \chi_{n} (x)&=   n^{d} 1_{B_{n}(0)} (x) , \qquad
  \widehat{\chi}_{n}(k) = \widehat{\chi}(n^{-1} k) = \mF_{\TT^{d}} \chi_{n}(k) = \mF_{\RR^{d}} \chi_{n}(k), 
 \\
   \chi_{Q_{n}}(x)&= n^{d} 1_{Q_{n}(0)}(x), \qquad
   \widehat{\chi}_{Q_{n}}
  (k) = \widehat{\chi}_{Q} (n^{-1} k) =  \mF_{\TT^{d}}
  \chi_{Q_{n}} (k)  = \mF_{\RR^{d}} \chi_{Q_{n}}(k).
\end{align*}
Observe that in order to obtain the identity between the Fourier transform on
the torus and in the full space, we have used that \( n \geqslant c(d)\).
A special role will be played by the \emph{semidiscrete Laplace} operator \(\mA_{n}\):
\begin{equation}\label{eqn:A_ve-definition} 
  \mA_{n}(\varphi)(x) = n^{2}  \mint{-}_{B_{n}(x)}
  \mint{-}_{B_{n}(y)} \mint{-}_{B_{n}(z)}\mint{-}_{B_{n}(r)}\varphi(s) {-}
  \varphi(x) \ud s \ud r \ud z \ud y  = n^{2}  \big( \Pi_{n}^{4}\varphi {-}
  \varphi \big) (x).
\end{equation} 
Such an operator is a Fourier multiplier with
\[ 
  \mA_{n} = \vt_{n}(D), \qquad \vt_{n}(k) =n^{2} \big(
\widehat{\chi}^{4}(n^{-1} k) -1\big).
\] 

We proceed with a definition of Besov spaces. Following
\cite[Proposition 2.10]{BahouriCheminDanchin2011FourierAndNonLinPDEs},
fix a dyadic partition of the unity \(\{ \varrho_{j}\}_{j \geq -1}\). We assume
that for \(j \geq 0\), \(\varrho_{j}( \cdot) = \varrho( 2^{-j} \cdot)\) is a  radial, smooth
compactly supported function.
For a distribution \(\varphi \in
\mS^{\prime}(\TT^{d})\) define \(\Delta_{j} \varphi = \varrho_{j}(D) \varphi\)
and then define the spaces \(B^{\alpha}_{p, q}\) for \(\alpha \in \RR, p, q \in [1,
\infty]\) via the norms
\[ \| \varphi \|_{B^{\alpha}_{p,q}} = \| (2^{\alpha j} \| \Delta_{j} \varphi
\|_{L^{p}(\TT^{d})})_{j \geq {-}1} \|_{\l^{q}}.\]
Since the partition of unity was chosen to be smooth, we define the Besov
spaces on the full space via the same formula.
It is convenient to introduce the notation
\begin{align*}
  K_{j}^{x}(y) = \mF_{\TT^{d}}^{-1} \rho_{j}(x-y).
\end{align*}

For \(\alpha \in \RR_{+} \setminus \NN_{0}\) and
\(p, q = \infty\) the above definition coincides with that of classical
H\"older spaces. We therefore write (for any $ \alpha \in \RR $ and $ p \in
[1, \infty] $)
\[ \mC^{\alpha} = B^{\alpha}_{\infty, \infty}, \qquad
\mC^{\alpha}_{p} = B^{\alpha}_{p, \infty}. \] 
We shall denote the norm of the H\"older space $\mC^{\alpha}$ by
$\|\cdot\|_{\alpha}$. 

To conclude, let $\mathcal{M}(\TT^d)$ be the space of
finite positive measures over $\TT^d$. For metric spaces \(X, Y\) let \(C(X;
Y)\) and \(C_{b}(X; Y)\) be respectively the
space of continuous, and continuous and bounded functions from \(X\) to
\(Y\). If \(Y = \RR\), we may drop the second argument. In addition, for a
metric space \(X\) we define \(\DD([0, \infty); X)\) to be the space of
c\`adl\`ag functions with values in \(X\), endowed with the Skorohod topology as in
\cite[Section 3.5]{EthierKurtz1986} (similarly for finite time horizon
\(T>0\) we write \(\DD([0,T]; X)\)). If \(X\) is a Banach space we write
\(L^{2}([0, T] ; X)\) for the space of measurable functions \(\varphi\) on \([0, T]\) taking values in
\(X\) and satisfying \(\| \varphi \|_{L^{2}([0, T]; X)} = \big(
\smallint_{0}^{T} \| \varphi(s) \|_{X}^{2} \ud s  \big)^{1/2}< \infty\).
The local variant of the space for $T = \infty$ is then defined as \(L^{2}_{\mathrm{loc}}([0, \infty); X) = \bigcap_{T >0} L^{2}([0, T];
X)\).

\subsection{The spatial $\Lambda$-Fleming-Viot process in a random environment}\label{sec:slfv_definition}

We now describe the Spatial Lambda-Fleming-Viot process. In addition to the
original neutral process we consider the effect of a (later randomly chosen) spatially
inhomogeneous selection.
We consider a population that presents two genetic types, $\mf{a}$ and $\mf{A}$. 
At each time $t\geq 0$, $X_{t}^{n}$ is a random function such that
$$X_{t}^{n}(x) = \text{ proportion of individuals of type } \mf{a}
\text{ at time } t \text{ and at position } x.$$ 

The dynamics of the Spatial Lambda-Fleming-Viot model are determined by reproduction events. In order to incorporate selection, we distinguish two types of reproduction events, neutral and selective. These events are driven by independent Poisson point processes. In simple terms 
\begin{align*}
\textbf{Neutral event:} & \ \ \text{Both types have the same chance of reproducing,} \\
\textbf{Selective event:} & \ \ \text{One of the two types is more likely to reproduce than the other.}
\end{align*}
The strength, and the direction of the selection are encoded respectively by the magnitude and sign of a function $s_{n}$. 
In our setting the function $ s_{n} $ will be chosen at random (so it will
depend on events $ \omega $ in a probability space $ (\Omega, \mF, \PP) $),
thus implying that the entire process $ X^{n}_{t} $ will also depend on $
\omega $ and be a Markov process \emph{only conditionally on the realization of
the environment.}
The function $s_{n}$ should satisfy the following requirements.
\begin{assumption}\label{assu:slfv:random-environment-general}
	Consider a probability space \((\Omega, \mF, \PP)\) and fix \(n \in
      \NN\). We assume that \(s_n\) is a measurable map $s_n \colon \Omega \to
    L^\infty(\TT^d; \RR),$ such that:
	\[|s_n(\omega, x)| < 1, \qquad \forall \omega \in \Omega, x \in \TT^d.\]
\end{assumption}

Conditional on the realization $s_n(\omega)$ of the environment, $X^n(\omega)$ will be a Markov process. Its dynamics
are defined below, deferring technical steps of the construction to
Appendix~\ref{app:construction}. We write:
 \[M = \big\{ w \colon \TT^d \to [0,1], \ \ w \ \text{ measurable}\big\}.\]

\begin{definition}[Spatial $\Lambda$-Fleming-Viot process with random selection]
\label{defn:spatial_lamdba_fv_in_random_environment}

  Fix \(n \in \NN, \mf{u} \in (0,1)\) and consider \(s_{n}\) and \((\Omega,
\mF, \PP)\) satisfying Assumption~\ref{assu:slfv:random-environment-general} and $X^{n, 0} \in M$. Define the process
\(X^{n}\) on the probability space \[ (\Omega \times \DD([0, \infty);M),
\mF \otimes \mB(\DD([0, \infty);M)), \PP \ltimes \PP^{\omega, n}),\] %
where \(\mA \otimes \mA^{\prime} \) is the product sigma-field of $ \mA ,
\mA^{\prime} $, $ \PP^{\omega , n} $ the conditional law of $ X^{n} $ given the
realization $ s_{n}(\omega) $ of the environment and $ \ltimes $ the semidirect
product as defined in the Appendix~\ref{app:construction}.
Then for every \(\omega \in \Omega\) it holds that

\begin{itemize}
  \item[i)]
The space \((\DD([0, \infty);M), \PP^{\omega,n})\) supports a pair of independent
Poisson point processes $\Pi^{\mathrm{neu}}_{\omega}$ and
$\Pi^{\mathrm{sel}}_{\omega}$ on 
$\RR_+\times \TT^d$ with intensity measures 
$\mathrm{d}t \otimes (1- |s_{n}(\omega, x)|)\mathrm{d}x$
and 
$\mathrm{d}t \otimes |s_{n}(\omega, x)|\mathrm{d}x $ 
respectively. 

\item[ii)] The random process \(\RR_{+} \ni t
  \mapsto X^{n}_{t}(\omega)\) is the canonical process on \(\DD([0, \infty);M)\). It is the Markov process with law $\PP^{\omega, n}$ started in
    \(X^{n, 0}\) with values in \(M\) associated to the generator
    \[\mL(n, s_{n}(\omega), \mf{u}) \colon C_b(M ; \RR) \rightarrow C_b(M; \RR)
    \;,\] (see
    Appendix \ref{app:construction}) that can be described by the
  following dynamics.
\begin{enumerate}
  \item If $(t,x) \in \Pi^{\mathrm{neu}}_{\omega}$, a neutral event occurs at time 
  $t$ in the ball $B_{n}(x)$, namely:
\begin{enumerate}
\item Choose a parental location $y$ uniformly in 
  $B_{n}(x)$. 
\item Choose the parental type $\mf{p} \in \{\mf{a}, \mf{A}\}$ according to the
distribution $ \QQ $
\begin{align*}
  \mathbb{Q}\left[\mf{p} = \mf{a} \right] = \Pi_{n}^{2} X^{n}_{t-} (\omega, y) , 
  \quad \mathbb{Q}\left[\mf{p} = \mf{A} \right] = 1 -  \Pi_{n}^{2} X^{n}_{t-} (\omega, y).
\end{align*}
\item A proportion $\mf{u}$ of the population within
  $B_{n}(x)$ dies and is replaced by an offspring with type $\mf{p}$.
Therefore, for each point $z \in B_n(x)$, 
\begin{align*}
  X^{n}_{t}(\omega,z) = (1 - \mf{u})X^{n}_{t -}(\omega, z) +
  \mf{u}\chi_{\{\mf{p} = \mf{a}\}}. 
\end{align*}
\end{enumerate}

\item If $(t,x) \in \Pi^{\mathrm{sel}}_{\omega}$, a selective event  occurs in
  the ball $B_{n}(x)$, namely:
\begin{enumerate}
\item Choose two parental locations $y_{0}, y_{1}$ independently, 
  uniformly in $B_{n}(x)$.
\item Choose the two parental types, $\mf{p}_{0}, \mf{p}_{1}, $ independently, according to  
\begin{align*}
  \mathbb{Q}\left[\mf{p}_{i} = \mf{a} \right] = \Pi_{n}^{2} X_{t-}^{n} (\omega,y_{i}), \quad 
  \mathbb{Q}\left[\mf{p}_{i} = \mf{A} \right] = 1 - \Pi_{n}^{2} X^{n}_{t-} (\omega,y_{i}).
\end{align*}
\item A proportion $ \mf{u}$ of the population within $B_{n}(x)$ dies
and is replaced by an offspring with type chosen as follows:
\begin{enumerate}
\item If  $s_n(\omega, x) < 0$, their type is set to be $\mf{a}$ if   
  $\mf{p}_{0} = \mf{p}_{1} = \mf{a}$, and $\mf{A}$ otherwise. 
  Thus for each $z \in B_{n}(x)$ 
\begin{align*}
  X^{n}_{t}(\omega,x)  = (1 - \mf{u}) X_{t-}^{n}(\omega,z)  +
  \mf{u}\chi_{\{\mf{p}_{0} = \mf{p}_{1} = \mf{a}\}}.
\end{align*}
\item If $s_n(\omega, x) > 0$, their type is set to be $ \mf{a}$ if   
  $\mf{p}_{0} = \mf{p}_{1} = \mf{a}$ or $\mf{p}_{0} \neq \mf{p}_{1}$ and $
  \mf{A}$ otherwise, 
  so that for each $z \in B_{n}(x)$, 
\begin{align*}
  X^{n}_{t}(\omega, z)  = (1 - \mf{u}) X^{n}_{t-}(\omega,z)  +
  \mf{u} ( 1 - \chi_{\{\mf{p}_{1}= \mf{p}_{2} = \mf{A}\}} ). 
\end{align*} 
\end{enumerate}
\end{enumerate}
\end{enumerate}
\end{itemize}
\end{definition}

\begin{remark}

    In the original SLFV process the probabilities at points \( 1.b, \ 2.b \)
    of the definition do not
  depend on the local average \(\Pi_{n}^{2} X_{t-}(y)\). Instead they
  depended only on the evaluation at the exact point \(X_{t-}(y)\).
Introducing the local average is a mathematical simplification of the model:
the main implication is that the operator \( \mH_{n}^{\omega} \) considered in
Theorem~\ref{thm:semidiscrete-pam-approximation} will be selfadjoint.

\end{remark}
Most of the arguments we use take advantage of the martingale representation of
the process. We record this representation as a lemma. The proof can be found
in Appendix \ref{app:construction}. For a
function \(\varphi \colon [0, \infty) \to \RR\) we write
\[ \varphi_{t,s} = \varphi_{t}- \varphi_{s}. \]

\begin{lemma}\label{lem:martingale_problem_model_description}
    Fix \(\omega \in \Omega\) and let $X^n$ be the SLFV process as in the previous definition.
For every \(\varphi \in L^{\infty}(\TT^{d})\) the process
    \(t \mapsto \langle X^{n}_{t}(\omega), \varphi \rangle\) satisfies the following
    martingale problem, for \(t \geq s \geq 0:\)
    \begin{align*}
      \langle X^{n}_{t,s}(\omega), \varphi \rangle =  \mf{u} n^{-d} \int_{s}^{t}
      \langle & \big( \Pi^{4}_{n} {-} \mathrm{Id}\big)
      (X^{n}_{r}(\omega)), \varphi \rangle \\
      & + \langle \Pi_n \big[ s_{n}(\omega)
      \big(\Pi_n^{3} X^{n}_{r}(\omega) {-} (\Pi_{n}^{3} X^{n}_{r}(\omega))^{2}\big)\big], \varphi  \rangle \ud r +
    M^{n}_{t,s}(\varphi)
\end{align*}
  where \(M^{n}_{t,s}(\varphi)\) is the increment of a square integrable martingale with
  predictable quadratic variation given by
\begin{align*}
  \langle  M^{n}(\varphi) \rangle_{t} = 
    \mf{u}^{2} n^{-2d} \int_0^{t} & \langle (1 {+}
    s_{n}(\omega)) \Pi_{n}^{3} X^{n}_{r}(\omega), ( \Pi_{n} \varphi)^2 {-} 2 
    (\Pi_{n} \varphi) \big(\Pi_{n}(X^{n}_{r}(\omega)
    \varphi) \big) \rangle \\
  & +  \langle \big( \Pi_{n}(X^{n}_{r}(\omega) \varphi )
  \big)^2, 1 \rangle \\
& -\langle s_{n}(\omega)(\Pi_{n}^{3}X^{n}_{r}(\omega))^2
     , (\Pi_{n}\varphi)^2 - 2
    (\Pi_{n} \varphi) \big(\Pi_{n}(X^{n}_{r}(\omega) \varphi) \big)
  \rangle \ud r. 
\end{align*}
\end{lemma}

\subsection{Sparse regime}
\label{sub:sparse_regime}
First, we consider a scaling regime in which the part of the population of type
$\mf{a}$ is rare, which means that $X_{t}^{n}$ is very close to $0$. To
quantify what we mean with "close to zero", we introduce a smallness parameter
$\vr>0.$ We assume that the initial condition $X^{n,0}$ is of order $n^{-\vr}$
and we will work under the following assumptions on the parameter $\vr$
(it will be a consequence of our scaling limit that if the initial condition is
small, then \( X^{n}_{t} \) stays of order $ n^{- \vr } $, at least for times
of order one in the appropriate diffusive scaling).

\begin{assumption}[Sparsity]\label{ass:sparsity}
  Fix any $\varrho> \frac {3d} {2}$ and a sequence \( \{X^{n,
0}\}_{n \in \NN} \subseteq M\)
  such that for some \(Y^{0} \in \mM(\TT^{d})\)
  \begin{equation*} 
    \lim_{n \to \infty}  n^{\varrho}
    X^{n, 0} = Y^{0} \text{ in } \mM(\TT^{d}).
  \end{equation*}
\end{assumption}
Our selection coefficient will converge to space white noise.
To obtain a non-trivial scaling limit in dimension $d=2$, renormalisation has to be taken into account. Hence we define
\begin{equation}\label{eqn:renormalization_constant-slfv}
   c_{n} = \sum_{k \in \ZZ^2}  \frac{\hat{\chi}^2( n^{-1} k) \hat{\chi}_Q(n^{-1}
   k)}{-\vt_n (k) +1} \simeq \log{n}.
\end{equation}
The assumptions on the noise are summarized in what follows. We emphasize that
$ \xi^{n} $ is an approximation of space white noise that is constant on the
disjoint boxes $ Q_{n}(x) $. In particular for any $ y \in \TT^{d} $ there
exists a unique $ x $ such that $ y \in Q_{n}(x) $.

\begin{assumption}[White noise scaling]
\label{ass:selection_coefficient}
Fix \(d = 1\) or \(2\) and consider a probability space \( (\Omega, \mF, \PP)\) supporting for any $n
  \in \NN$ a sequence of i.i.d.\ random variables \(\{Z_n(x)\}_{x \in \znd}\)
  satisfying $\EE \big[ Z^2_n(x) \big] =  1, \ Z_n(\omega, x) \in(-2,2)$
  for all $x \in \znd$ and $\omega \in \Omega.$
  Then define \[s_{n}(\omega, y) = Z_{n}( \omega, x) - n^{-\frac d 2}c_{n}
  1_{\{d = 2\}}, \ \text{ if } y \in
  Q_n(x), \qquad \forall \omega \in \Omega, x \in \znd\] 
  and write:
  \[ 
    \xi^{n}_{e}(\omega, x) = n^{\frac{d}{2} } s_{n}( \omega, x),
    \qquad \xi^{n}(\omega, x) = \xi^{n}_{e}(\omega, x) + c_{n}
    1_{\{d=2\}}.
  \] 
\end{assumption}
Under appropriate scaling, we will prove that the process \(X^{n}\) converges to a rough superBrownian motion.
First, we recall the Anderson Hamiltonian on the torus, and its relationship to our setting.

\begin{lemma}\label{def:anderson_hamiltontian}
  Let \( (\Omega, \mF, \PP)\) be a probability space
  supporting a white noise \(\xi \colon \Omega \to \mS^{\prime}(\TT^{d})\), that
  is a process such that for all \(f \in \mS(\TT^{d})\) the projection
  \(\langle \xi, f \rangle=\smallint_{\TT^{d}} f(x) \xi( \ud x)\)
  are Gaussian random variables with covariance
  \[ \EE \big[ \langle \xi, f \rangle \langle \xi, g \rangle \big] = \langle f, g \rangle, \qquad \qquad
  \forall f, g \in \mS(\TT^{d}).\] 
  For almost every \(\omega \in \Omega\) there exists an operator
  \[ \mH^\omega \colon \mD_{\omega} \subseteq C(\TT^{d}) \to
  C(\TT^{d}),\]
  with a dense domain $\mD_\omega \subseteq C(\TT^d)$, such that 
  \[ \mH^{\omega} = \lim_{n \to \infty} \Big[ \mA_{n}+ \Pi_{n}^{2} (\xi^{n}-c_{n}
      1_{\{d= 2\}}) \Pi_{n}^{2} \Big] =: \nu_{0} \Delta + \xi,\] 
with \( \nu_{0} \) defined by:
\begin{equation}\label{eqn:for-nu-0}
\nu_{0} = \frac 1 {3} \ \ \text{ in } \ \ d = 1, \qquad \nu_{0} = \frac 1
  { \pi} \ \ \text{ in } \ \ d = 2.
      \end{equation}
  The limit is taken in distribution, with $\xi^n$ as in
Assumption~\ref{ass:selection_coefficient}. The precise meaning of the limit is
provided in Theorem~\ref{thm:semidiscrete-pam-approximation}.
\end{lemma}

This lemma is a consequence of Proposition~\ref{prop:continuous-pam-operator}
and Theorem~\ref{thm:semidiscrete-pam-approximation} below.
The rough superBrownian motion is then a Markov process conditional on the
realization of the spatial white noise and thus on the realization of the Anderson Hamiltonian.

\begin{definition}\label{defn:rough_superbrownian_motion}
  Let \((\Omega, \mF, \PP)\) be a probability space supporting a white noise
  \(\xi\) and consider $Y^{0} \in \mathcal{M}(\TT^{d})$. Consider
  an enlarged probability space \((\Omega \times \overline{\Omega}, \mF \otimes
  \overline{\mF}, \PP \ltimes \overline{\PP}^{\omega})\), where  \(\PP^{\omega}\) is the
  conditional (given the realization $\omega$ of the environment) law of a process
  $Y \colon \Omega \times \overline{\Omega} \to C([0, \infty);
  \mM(\TT^{d})).$
  For any \(\omega \in \Omega\) let \(\{\mF_{t}^{\omega}\}_{t \geq
  0}\) be the filtration
  generated by \(t \mapsto Y_{t}(\omega)\), right-continuous and enlarged with all null sets. And let
  \(\mH^\omega\) be the operator in the definition above. $Y$ is a rough superBrownian motion, if for all \(\varphi \in \mD_{\omega}\) and $T >0$, the process
      \begin{align*} 
	M_{t}^{\varphi}(\omega) = \langle Y_{t}(\omega), \varphi \rangle - \langle Y^{0}, \varphi \rangle - \int_{0}^{t}
	\langle Y_{s}(\omega),\mathcal{H}^{\omega} \varphi \rangle \mathrm{d}s 
      \end{align*}
      is a centered continuous, square integrable
      \(\mF^{\omega}_{t}\)-martingale on \([0,T]\)  with quadratic variation 
     \begin{align*}
       \langle M^{\varphi}(\omega) \rangle_{t} = \int_{0}^{t}\langle
       Y_{s}(\omega), \varphi^{2} \rangle \ \mathrm{d}s.
     \end{align*}
\end{definition}

We are now in position to state the first main result of this work.
\begin{theorem}\label{thm:convergence_rsbm}
  For any \(\varrho> \frac{3}{2} d\) consider a random environment \(s_{n}\) as in
  Assumption~\ref{ass:selection_coefficient}, and
  initial conditions \(X^{n, 0}\) as in
  Assumption~\ref{ass:sparsity}.
  Consider the process \(X^{n}\) as in
  Definition~\ref{defn:spatial_lamdba_fv_in_random_environment}, but associated, for each $\omega \in \Omega$, to
  the generator
  \[ n^{d+2+\eta} \mL (n, n^{\frac{d}{2}-2} s_{n}(\omega),
  n^{-\eta}),\] 
meaning that with respect to the process constructed in
Definition~\ref{defn:spatial_lamdba_fv_in_random_environment}, we speed up time
by a factor $ n^{d+2+\eta} $, consider impacts of order $ n^{- \eta} $ and use as
environment the function $ n^{\frac{d}{2}-2} s_{n}(\omega)$.
  Here \(\eta\) is defined by
  \begin{align}\label{eqn:definition_of_intensity_eta_rsbm}  
  \eta := \varrho + 2 - d.
  \end{align}
  Then the process \(t \mapsto Y^{n}_{t} = n^{\vr}X^{n}_{t}\) converges
  in distribution in $\DD([0, \infty);
  \mM(\TT^{d}))$ to a process $ Y $, which is 
  the unique (in distribution) rough superBrownian motion as in
  Definition~\ref{defn:rough_superbrownian_motion}, started in \(Y^{0}\).
\end{theorem}
\begin{remark}

Let us comment on the scaling in the previous theorem. The temporal speed of order $n^{d + 2 + \eta}$ corresponds to parabolic scaling. The factor $n^d$ is payed to cancel the corresponding factor appearing in Lemma~\ref{lem:martingale_problem_model_description}. The factor $n^\eta$ instead cancels with the size of the impact. So we are left with a factor $n^2$, which corresponds to parabolic scaling, since spatial distances are of order $1/n$.

As for the selection, we necessarily consider a \emph{weak} regime, that is $|s_n| \simeq n^{-2}$, which cancels with the temporal speed up, providing a term of macroscopic order.
Finally, the smallness of the impact enters only to see fluctuations of the correct order. 
\end{remark}

\subsection{Diffusive regime}
\label{sub:fkpp}
The second scaling regime we consider is a purely diffusive one. 
As before, the impact factor $\mf{u}$ is scaled as $n^{-\eta}$. 
The restrictions on the value of $\eta$ 
follow 
\begin{assumption}\label{assu:small_jumps_fkpp}
  Choose \(\eta\) such that $\eta = 1$ if $d=1,$ and $\eta >0$ if $d=2$.
\end{assumption}
In this diffusive regime we still assume that the selection coefficient may be random, but we restrict to smooth selection.
\begin{assumption}\label{assu:smoothened_noise_fkpp}
	  Consider a probability space $(\Omega, \mF, \PP)$ and let
      $\overline{\xi}$ be a measurable map: $\overline{\xi} \colon \Omega \to
    \mS(\TT^d)$. Then define:
	$$s_n (\omega, x)= (n^{-2} \bar{\xi}(\omega, x) )\vee 1 \wedge (-1).$$
\end{assumption}

The limiting process in this setting will be the (stochastic if \(d=1\)) Fisher--KPP equation in a random potential, defined as follows.

\begin{definition}\label{def:Fisher-KPP_in_rough_potential}
 Consider $\Omega$ and $\overline{\xi}$ as in Assumption~\ref{assu:smoothened_noise_fkpp}. Fix any $\alpha>0$ and \(X^{0} \in B^{\alpha}_{2,2}\). A
  (stochastic if \(d=1\)) Fisher--KPP
  process in random potential is a couple given by
  a probability space \((\Omega \times \overline{\Omega}, \mF \otimes
  \overline{\mF}, \PP \ltimes \overline{\PP}^{\omega})\) (cf.\
  Definition~\ref{defn:rough_superbrownian_motion}) and a map 
  $ X \colon \Omega \times \overline{\Omega} \to L^2_{\mathrm{loc}}([0,
  \infty); B^{\alpha}_{2,2}).$ 
  For \(\omega \in \Omega\) let \(\{\mF^{\omega}_{t}\}_{t \geq 0}\)
  be the filtration generated by \(t \mapsto X_{t}(\omega)\), right-continuous
  and enlarged with all null sets. Then for all \(\omega \in \Omega\) it is
  required that:
  \begin{itemize}
    \item[i)] In dimension $d = 1$ for all \(\varphi \in C^{\infty}(\TT)\):  
      \begin{align*}
	N^{\varphi}_{t} := \langle X_{t}(\omega), \varphi \rangle - \langle X^{0}, \varphi \rangle -
	\int_{0}^{t} \langle X_{s}(\omega), \nu_{0} \Delta \varphi \rangle -
	 \langle  \overline{\xi} (\omega)X_{s} (\omega) (1 -  X_{s}(\omega)) , \varphi\rangle \mathrm{d}s
      \end{align*}
is a continuous in  time, square integrable martingale with quadratic variation 
\begin{align*}
  \langle N^{\varphi} \rangle_{t} = \int_{0}^{t} \langle
  X_{s}(\omega)(1 -  X_{s}(\omega)), \varphi^{2} \rangle \mathrm{d}s.
\end{align*}
\item[ii)] In dimension $d = 2$ for all \(\varphi \in
      C^{\infty}(\TT^{2})\)
      \begin{align*}
	\langle X_{t}(\omega), \varphi \rangle = \langle X^{0}, \varphi \rangle +
	\int_{0}^{t} \langle X_{s}(\omega), \nu_{0} \Delta \varphi \rangle +
	 \langle  \overline{\xi}(\omega) X_{s}(\omega) (1 -  X_{s}(\omega)) ,
	 \varphi\rangle \ \mathrm{d}s.
      \end{align*}
  \end{itemize}
  \end{definition}

  \begin{remark}
    Note that in the previous definition, since \(X \in
    L^{2}_{\mathrm{loc}}([0, \infty); B^{\alpha}_{2,
    2})\) for $ \alpha > 0 $, the non-linearity
    $ \smallint_{0}^{t} \langle X_{s}^{2}, \varphi \rangle \ud s $ is
    well-defined. 
    Moreover, up to enlarging the probability space, the process can be
    represented in \(d=1\) as a solution to an SPDE of the form
\begin{align*}
  \partial_{t}X = \nu_{0}\Delta X + \overline{\xi} X(1 - X) +
  \sqrt{ X(1-X)} \widetilde{\xi},
\end{align*}
    where the spatial noise \( \overline{\xi}\) is independent of the space-time white
    noise \(\widetilde{\xi}\), following \cite{KonnoShiga1988}.
  \end{remark}

In this setting, we can prove the following scaling limit.

\begin{theorem}\label{thm:convergence_fkpp}
  Let \(\eta\) satisfy Assumption~\ref{assu:small_jumps_fkpp} and
  \(s_{n}\) be as in Assumption~\ref{assu:smoothened_noise_fkpp}. Consider
  $X_0 \in \mS(\TT^d)$ with \( 0 \leqslant X_{0}(x) \leqslant 1, \ \forall x
\in \TT^{d}\), and let $X^n(\omega)$ be the
  Markov process associated to the generator $$n^{\eta + d+2.}\mL(n,
  s_{n}(\omega), n^{-\eta})$$ and started in $X_0$, as in
  Definition~\ref{defn:spatial_lamdba_fv_in_random_environment}. There exists
  an \(\alpha>0\) such that for every \(\omega \in \Omega\) $\{t \mapsto
  \Pi_n X^{n}_{t}(\omega)\}_{n \in \NN}$ is tight in the space
  \(L^2_{\mathrm{loc}}([0, \infty); B^{\alpha}_{2,2}(\TT^d))\). Similarly, the sequence $\{t \mapsto X^n_t(\omega)\}_{n \in \NN}$ is tight in $\DD([0,\infty); \mM(\TT^d))$. In particular:
  \begin{itemize}
    \item[i)] In dimension \(d=1\) both sequences converge in distribution to the unique in law solution to the martingale problem of the stochastic Fisher-KPP process in a random potential, as in Definition~\ref{def:Fisher-KPP_in_rough_potential}.
  \item[ii)] In dimension \(d=2\) both sequences converge in distribution to
      the unique solution to the Fisher-KPP equation in a random potential as in
      Definition~\ref{def:Fisher-KPP_in_rough_potential}.
  \end{itemize}
\end{theorem}
\begin{remark}
The scaling in Theorem~\ref{thm:convergence_fkpp} is similar to that of
Theorem~\ref{thm:convergence_rsbm}  in the case $\vr =0$. The only difference
is the assumption $\eta>0$ in $d=2$. At $\eta = 2-d$ we expect to
see fluctuations, so a natural guess would be the appearance of a \emph{stochastic}
Fisher-KPP equation. Instead, the limit 
should be deterministic. Indeed if $\overline{\xi}=0$ one can
show that the dual converges to a system of coalescing Brownian motions: in
dimension $d=2$ Brownian motions cannot meet, leading to the heat equation. In
our setting we expect that the same
argument holds and the correct scaling limit should be the
deterministic Fisher-KPP equation. 
\end{remark}

\subsection{Proof methods}
\label{subsec:proof_methods}

The main ingredient in the proofs of the previous scaling limits is a careful study of the semidiscrete Laplace
operator \(\mA_{n}\). Intuitively, one expects that this operator
approximates the Laplacian with periodic boundary conditions and therefore has similar regularizing properties. To quantify this
intuition we introduce a division of scales. On large scales, namely for
Fourier modes \(k\) of order \(|k| \lesssim n \) we show that
\(\mA_{n}\) has the required regularizing properties. On small
scales, that is for modes of order \(|k| \gtrsim n \) we do not
expect any regularization at all. Instead we prove that small scales are negligible in terms of powers of $n$.
Below we state a slimmed version of the results we require. The proof of the following
theorem, as well as additional side results, is the content of Section~\ref{sec:schauder_estimates}.

\begin{theorem}\label{thm:regularization-estimates-main-results}
  Fix any smooth radial function with compact
  support \(\daleth \colon \RR^{d} \to \RR\) such that for some \(0 < r<R\)
  $ \daleth(k) = 1, \ \forall |k| \leq  r$, and  $\daleth(k) = 0, \ \forall |k|
  \geq R$. Then define
  \[ \mP_{n} = \daleth( n^{-1} D), \qquad \mQ_{n} = (1 - \daleth)(n^{-1} D).\]
  For any \(\alpha \in \RR, p \in [1, \infty]\) the following holds:
  \begin{itemize}
    \item[i)]  For $ \nu_{0}  $ as in \eqref{eqn:for-nu-0}, any \(\zeta>0\) and \(\varphi \in \mC^{\alpha}_{p}\)
  \[ 
    \mA_{n} \varphi \to  \nu_{0} \Delta \varphi \ \
      \text{in} \ \ \mC^{\alpha -2- \zeta}_{p},\qquad \text{ as } \ n \in \NN.
  \] 
    \item[ii)]  Uniformly over \(\lambda>1, n \in \NN\) and \(\varphi \in \mC^{\alpha}_{p}\) the following estimates hold:
  \[ 
    \| \mP_{n} ( - \mA_{n} + \lambda)^{-1} \varphi \|_{\mC^{\alpha +
      2}_{p}} +  n^{2}\| \mQ_{n} (- \mA_{n} +
  \lambda)^{-1} \varphi\|_{\mC^{\alpha}_{p}} \lesssim \| \varphi
  \|_{\mC^{\alpha}_{p}}. \] 
  \end{itemize}
\end{theorem}

A precise control of the regularization effects of the
semidiscrete Laplacian \(\mA_{n}\) allows us to treat semidiscrete
approximations of the Anderson model that appear in the study of the rough
superBrownian motion. In the next proposition we recall some salient features
of the continuous Anderson Hamiltonian.

\begin{proposition}\label{prop:continuous-pam-operator}
  Fix \(d=1\) or \(2\),  $\kappa>0$ and \( (\Omega, \mF,
\PP) \) a probability space supporting a space white noise $\xi \colon \Omega
\to \mS^\prime (\TT^d)$. Then the following holds true for all
$\omega\in\Omega$. 
\begin{itemize}
\item[i)] The Anderson Hamiltonian
  \[\mH^\omega = \nu_0\Delta +
  \xi(\omega)\] associated to \(\xi(\omega)\) is defined,
  as constructed
  in \cite{Fukushima1976} in \(d=1\) and
  \cite{AllezChouk2015AndersonHamiltonian2D} in \(d=2\).
\item[ii)]  The Hamiltonian, as an
  unbounded selfadjoint operator on $L^2(\TT^d)$, has a discrete
  spectrum given by pairs of eigenvalues and eigenfunctions
  \(\{(\lambda_{k}(\omega), e_{k}(\omega))\}_{k \in \NN}\) such that:
  \[ 
    \lambda_{1} (\omega) > \lambda_{2} (\omega) \geq \lambda_{3} (\omega) \geq
    \ldots, \qquad \lim_{k \to \infty} \lambda_{k}(\omega) = {-} \infty, \qquad
    e_{1}(\omega, x)>0, \forall x \in \TT^{d}.\]
  \item[iii)] In addition, for every \(k \in \NN,\) \(e_{k}(\omega) \in
  \mC^{2 - \frac{d}{2} - \kappa} (\TT^{d})\), and the following set is dense in
  \(C(\TT^{d})\):
  \[ \mD_{\omega} = \{ \text{Finite linear combination of }
  \{e_{k}(\omega)\}_{k \in \NN} \}. \] 
  \end{itemize}

\end{proposition}

The proof of this proposition is postponed to Section~\ref{sec:semidiscrete_PAM}, in Lemmata~\ref{lem:semidiscrete-pam:continuous-pam-operator}~and~\ref{lem:domain_for_anderson_hamiltonian}. For the semidiscrete Laplace operator $\mA_n$ the following holds.

\begin{theorem}\label{thm:semidiscrete-pam-approximation}

  Fix \(d = 1\) or \(2\), $\kappa>0$ and $\xi^n$ satisfying Assumption
  \ref{ass:selection_coefficient}. Up to changing probability space \( \Omega
\), the following hold true for almost all \( \omega\) in \( \Omega \).
For every \(k \in \NN\) let \(m(\lambda_{k})\) be the multiplicity of the
  eigenvalue \(\lambda_{k}\) of \(\mH^{\omega}\) and let
  \(\{e_{k}^{i} (\omega)\}_{i = 1}^{m (\lambda_{k})}\) be an associated set of
orthonormal eigenfunctions. In particular, \( m(\lambda_{1})=1 \).

Then for every \( k \in \NN \) there exists an \(n_{0}(\omega, k) \in \NN\)
  such that for every \(n \geq n_{0}(\omega, k)\) there exist orthonormal
  functions \(\{e_{k}^{i, n}(\omega)\}_{i = 1}^{m (\lambda_{k})} \subseteq
  L^{2}(\TT^{d})\) such that, considering the operator
  \[ 
    \mH_{n}^\omega := \mA_{n} + \Pi_{n}^{2} (\xi^{n}(\omega) -c_
    n) \Pi_{n}^{2}, \qquad \mH_n^\omega \colon L^2(\TT^d) \to L^2(\TT^d),
\]
  with $c_n$ as in \eqref{eqn:renormalization_constant-slfv}, one has for some
\( \ve>0 \):
  \[ 
    \lim_{n \to \infty} e^{i, n}_{k}(\omega) = e^{i}_{k} (\omega), \ \ \text{in} \ \
    L^{2}(\TT^{d}) \qquad
    \lim_{n \to \infty} \Pi_{n} e_{k}^{i,n}(\omega) = e_{k}^{i}(\omega) \ \ \text{in} \ \
    \mC^{\ve} (\TT^{d}),
  \] 
  and
  \[ \lim_{n \to \infty} \mH^{\omega}_{n} e^{i, n}_{k}(\omega) =
    \lambda_{k} e^{i}_{k}(\omega), \ \ \text{in} \ \ L^{2}(\TT^{d}), \qquad \lim_{n \to
    \infty} \Pi_{n} \mH^{\omega}_{n} e^{i}_{k}(\omega) = \lambda_{k} e^{i}_{k}
    (\omega)\ \ \text{in} \ \ \mC^{\ve}(\TT^{d}).\]
If the eigenvalue is simple, i.e.\ \( m(\lambda_{k}) =1 \), then in addition \(
e^{n}_{k} (\omega) \) is an eigenfunction for \( \mH^{\omega}_{n}:\)
\[ \mH^{\omega}_{n} e^{n}_{k} (\omega) = \lambda^{n}_{k} e^{n}_{k}(\omega), \] 
with \( \lim_{n \to \infty}\lambda^{n}_{k} = \lambda_{k}.\)
\end{theorem}

The proof of this result is the content of
Section~\ref{subsection:proof-of-pam-theorem}.

\section{Scaling to the rough super-Brownian motion}
\label{sec:scaling_to_rsbm}

This section is devoted to the proof of Theorem~\ref{thm:convergence_rsbm}.
Since we want to prove convergence in distribution for the sequence $Y^n$, the
exact choice of the probability space $\Omega$ of
Definition~\ref{defn:spatial_lamdba_fv_in_random_environment} is not important.
For this reason we make the following assumption.

\begin{assumption}\label{assu:probability-space}
	Let $(\Omega, \mF, \PP)$, the probability space in
	Definition~\ref{defn:spatial_lamdba_fv_in_random_environment} and
	Assumption~\ref{ass:selection_coefficient}, be such that results of Proposition~\ref{prop:continuous-pam-operator} and
	Theorem~\ref{thm:semidiscrete-pam-approximation} hold true for all $
	\omega \in \Omega $.
\end{assumption}

The first step towards establishing tightness is to
restate the martingale problem of
Lemma~\ref{lem:martingale_problem_model_description} to take into account the scaling assumed in  Theorem~\ref{thm:convergence_rsbm}. 

\begin{lemma}\label{lem:discrete_martingale_problem_rsbm_scaling}

  In the setting of Theorem~\ref{thm:convergence_rsbm} and under Assumption~\ref{assu:probability-space}, for every \(\omega \in \Omega\) and \(n \in \NN\), under the law
  \(\PP^{\omega}\), and for every $\varphi \in L^{\infty}(\TT^{d})$ the process
  \(t \mapsto \langle Y^{n}_{t} (\omega), \varphi \rangle\) satisfies the following martingale problem:
\begin{equation}\label{eqn:discrete_mp_drift}
  \begin{aligned}
    \langle Y^{n}_{t,s}(\omega), \varphi \rangle   =  \int_{s}^{t} &\langle
\mA_{n}(Y^{n}_{r}(\omega)) +\Pi_n[ \xi^{n}_{e}(\omega)
    \Pi_n^{3} Y^{n}_{r}(\omega) ], \varphi  \rangle \\
&  \quad- n^{-\varrho} \langle
    \big(\Pi_{n}^{3}Y^{n}_{r}(\omega) \big)^{2}, \xi^{n}_{e}
(\omega)\Pi_{n}(\varphi) \rangle \ud r + M^{n}_{t,s}(\varphi),
\end{aligned}
\end{equation}
  where \(M^{n}_{\cdot}(\varphi)\) is a square integrable martingale with
  predictable quadratic variation 
  \begin{equation}\label{eqn:discrete_mp_quadratic_variation}
\begin{aligned}
  \langle  M^{n}(\varphi) \rangle_{t} =  
  \int_0^{t} & \langle (1 {+} n^{-2 + \frac
d2}s_{n}(\omega))\Pi_{n}^{3}Y^{n}_{r}(\omega), ( \Pi_{n} \varphi)^2 {-} 2 n^{-\varrho}
  \Pi_{n} (\varphi) \Pi_{n}(Y^{n}_{r}(\omega)
  \varphi) \rangle  \\
  & + n^{-\varrho} \langle \big( \Pi_{n}(Y^{n}_{r}(\omega) \varphi )
  \big)^2, 1 \rangle \\
  &  - n^{-\varrho} \langle n^{-2 + \frac d2}s_{n}(\omega) (\Pi_{n}^{3}Y^{n}_{r}(\omega))^2,
   (\Pi_{n}\varphi)^2 {-} 2
  n^{-\varrho} \Pi_{n}(\varphi) \Pi_{n}(Y^{n}_{r} (\omega)\varphi)
  \rangle \ud r. 
\end{aligned}
\end{equation}
\end{lemma}
\begin{remark}\label{rem:bfrproof}
  The only term that is not of lower order in the quadratic
  variation is 
  $ \langle \Pi_{n}^{3} Y^{n}_{r}, (\Pi_{n}\varphi)^{2} \rangle,$ 
  which explains the superBrownian noise in the limit.
  Furthermore, at first sight this martingale problem has no relationship with the operator \[
\mH^{\omega}_{n} = \mA_{n} + \Pi_{n}^{2} \xi^{n}_{e}(\omega)
\Pi_{n}^{2} \] we introduced earlier. The reason for our definition of $
\mH^{\omega}_{n} $ is that if we test on \( \varphi = \Pi_{n} e^{n} \), with \( e^{n} \) in the domain of \(
\mH^{\omega}_{n} \), then the first line of the drift
becomes
$ \langle Y_{r} (\omega), \Pi_{n} \mH^{\omega}_{n} e^{n} \rangle, $ 
which is exactly the kind of term that
Theorem~\ref{thm:semidiscrete-pam-approximation} aims at controlling.
\end{remark}

In order to obtain the convergence, the first step is to prove a tightness
result. 

\begin{proposition}\label{prop:tightness_rough_SBM}
  In the setting of Theorem~\ref{thm:convergence_rsbm} and under Assumption~\ref{assu:probability-space} fix any $\omega \in \Omega$. For any \(T>0\) the sequence
  \(\{Y^{n}(\omega)\}_{n \in \NN}\) is tight in \(\DD([0,T]; \mM(\TT^{d}))\).
  Moreover any limit point is continuous, i.e.\ lies in \(C([0,T];
  \mM(\TT^{d}))\).
\end{proposition}

\begin{proof}[Proof of Proposition~\ref{prop:tightness_rough_SBM}]

Since \(\omega \in \Omega\) is fixed, we omit the dependence on it. The proof
relies on Jakubowski's tightness criterion \cite[Theorem 3.1]{jakubowski:1986}. The criterion consists of a
compact containment condition and the tightness of one-dimensional
projections.

In a first step of the proof, we establish the compact containment condition.
  Since for \(R >0\) sets of the form \(K_{R} =
  \{\mu \colon \langle \mu, 1 \rangle \leq R\} \subseteq \mM(\TT^{d})\) are compact in the weak
  topology,
 it is sufficient to show that
  \begin{equation}\label{eqn:compact_containment}
    \begin{aligned}
      \forall \delta>0, \quad \exists R(\delta) >0, \ n(\delta) \in \NN \
      \text{such that} \ \inf_{n \geq n(\delta)} \PP \Big( \sup_{t \in [0,T]} \langle
      Y_t^{n}, 1\rangle \leq R (\delta) \Big) \geq 1- \delta.
    \end{aligned}
  \end{equation}

  In a second step, we establish one-dimensional tightness. 
By Theorem~\ref{thm:semidiscrete-pam-approximation} (since the domain $\mD_\omega$ is dense in $C(\TT^d)$), it is 
  sufficient to show that for every \(k \in \NN\) the process \(\langle
Y^{n}_{t}, e_{k} \rangle\) is tight in \(\DD([0,T]; \RR)\),
where the sequence \( \{ e_{k}\}_{k \in \NN} \) is an
orthonormal basis of \( L^{2}(\TT^{d}) \) consisting of eigenfunctions of \(
\mH \), as in Proposition~\ref{prop:continuous-pam-operator}.
By Aldous' tightness criterion \cite[Theorem 1]{Aldous1978}, this reduces to proving that
  for any sequence  of stopping times \(\tau_{n}\), taking finitely many
  values and  adapted to the filtration
  of \(Y^{n}\), and any sequence \(\delta_{n}\) of constants such that \(\delta_{n} \to
  0\) as \(n \to \infty\)
  \begin{equation}\label{eqn:one_dimensional_tightness}
    \begin{aligned}
      \forall \delta>0, \qquad \lim_{n \to \infty} \PP \Big( | \langle
	Y^{n}_{\tau_{n}+ \delta_{n}},
      e_{k} \rangle - \langle Y^{n}_{\tau_{n}}, e_{k} \rangle| \geq \delta
      \Big) = 0 .
    \end{aligned}
  \end{equation}
  In the third step we address the continuity of the limiting process.

\textit{Step 1.}
 By Theorem~\ref{thm:semidiscrete-pam-approximation},  for any \(k \in \NN\) and
\(n \geq n_0(k)\) there exists a function $e^n_k \in L^{2}(\TT^{d})$
such that \( \Pi_{n} e_{k}^{n} \to e_{k} \) in \( \mC^{\ve}(\TT^{d}), \) and \( \Pi_{n} \mH_{n} e_{k}^{n} \to \lambda_{k}
e_{k} \) in \( \mC^{\ve}(\TT^{d}) \) for some \( \ve> 0 \).
 In particular, choose \( k =1. \) Then \( \lambda_{1} \) is simple and we can
choose \( e_{1}^{n} \) to be an eigenfunction of \( \mH_{n} \) of eigenvalue \(
\lambda_{1}^{n} \to \lambda_{1} \). Since
  \(e_{1}>0\), we may assume that \(\Pi_{n} e_{1}^{n}>0, \forall n \geq n_{0}(1)\) and hence  for any positive measure \(\mu\) there exists a \(C>0\) such that
  \[ \langle \mu , 1 \rangle \leq C  \langle
  \mu , \Pi_{n} e^{n}_{1} \rangle, \qquad \forall n \geq
n_{0}(1).\]
Therefore \eqref{eqn:compact_containment} follows if one can show that
  \begin{equation*}
    \begin{aligned}
      \forall \delta>0, \quad \exists R(\delta) >0,  \ n(\delta) \geq n_0(1) \
      \text{such that} \ \inf_{n \geq n(\delta)} \PP \Big( \sup_{t \in [0,T]} \langle
      Y_t^{n}, \Pi_{n} e_{1}^{n}\rangle \leq R(\delta) \Big) \geq 1- \delta.
    \end{aligned}
  \end{equation*}
 Let us hence focus on \(\langle Y^{n}_{t},\Pi_{n} e_{1}^{n} \rangle\). By the
 martingale representation \eqref{eqn:discrete_mp_drift} (see also the
 discussion in Remark~\ref{rem:bfrproof}) one obtains
  \begin{align*}
    \langle Y^{n}_{t}, \Pi_{n} e_{1}^{n} \rangle = \langle
    Y^{n}_{0}, \Pi_{n} e_{1}^{n} \rangle
    + \int_{0}^{t} \!\!\! \lambda_{0}^{n} \langle Y^{n}_{r},
    \Pi_{n} e_{1}^{n} \rangle - n^{-\varrho} \langle
    \big(\Pi_{n}^{3} Y^{n}_{r}\big)^{2}, \xi^{n}_{\ve} \Pi_{n}^{2}
    e^{n}_{1} \rangle \ud r + M^{n}_{t}(\Pi_{n} e_{1}^{n}).
  \end{align*}  
  To treat the nonlinear quadratic term, we shall consider a stopped process. 
Let us fix \(R>0\) and consider the stopping
  time \(\tau_{R}\) and a parameter \(\varrho_{0}\), defined as
  \[ \tau_{R} := \inf \{ t \geq 0 \ \colon \ \langle Y^{n}_{t},
  \Pi_n e^n_1  \rangle \geq R\}, \qquad \varrho_{0} = \varrho-\frac d 2 -  d.\]
 Since $|\xi^n(x)| \lesssim n^{\frac d 2}$ and since
\[ \| \Pi_{n}^{3} Y^{n}_{r}  \|_{\infty} \leqslant \| \Pi_{n} Y^{n}_{r}
\|_{\infty} \leqslant n^{d} \langle Y^{n}_{r}, 1 \rangle \lesssim n^{d} \langle Y^{n}_{r}, \Pi_{n} e_{1}^{n} \rangle,\] 
 one can bound
  \begin{align*}
    n^{-\varrho}| \langle (\Pi_{n}^{3} Y^{n}_{r \wedge \tau_{R}})^{2}, \xi^{n} \Pi_{n}^{2}
    e^{n}_{1} \rangle| \lesssim n^{-\varrho+\frac d 2 +  d} \langle
    Y^{n}_{r \wedge \tau_{R}} , \Pi_n e_1^n \rangle^{2} \lesssim R
    n^{-\varrho_{0}} \langle Y^{n}_{r \wedge \tau_{R}}, \Pi_n e^n_1 \rangle,
  \end{align*}
  and therefore
  \begin{align*}
    \EE | \langle Y^{n}_{t \wedge \tau_{R}}, \Pi_{n} e^{n}_{1} \rangle
    |^{2} \lesssim \langle Y^{n}_{0}, 1\rangle + (1 + R
    n^{-\varrho_{0}}) \int_{0}^{t} \!\!\! \EE |\langle Y^{n}_{r\wedge \tau_{R} },
    \Pi_{n} e_{1}^{n} \rangle|^{2} \ud r + \EE \langle M^{n}(\Pi_{n}
    e^{n}_{1}) \rangle_{t \wedge \tau_{R}}.
  \end{align*}
Furthermore, using the formula for the predictable quadratic variation from
  Lemma~\ref{lem:discrete_martingale_problem_rsbm_scaling} 
\begin{align*}
  \EE \langle  M^{n}(\Pi_{n} e_{1}^{n}) & \rangle_{t \wedge
    \tau_{R}} \leqslant   
  \EE \int_0^{t}  \langle (1 {+} n^{-2 + \frac
d2}s_{n})\Pi_{n}^{3}Y^{n}_{r \wedge \tau_{R}}, (
\Pi_{n}^{2} e_{1}^{n})^2 \rangle  + n^{-\varrho} \langle \big( \Pi_{n}(Y^{n}_{r \wedge
\tau_{R}} \Pi_{n} e_{1}^{n} )
  \big)^2, 1 \rangle \\
  &  + n^{-\varrho} \langle n^{-2 + \frac d2}| s_{n} | (\Pi_{n}^{3}Y^{n}_{r
\wedge \tau_{R}})^2, (\Pi_{n}^{2} e_{1}^{n} )^2 + 2 n^{-\varrho} (\Pi_{n}^{2}
e_{1}^{n})  \Pi_{n}(Y^{n}_{r \wedge \tau_{R} } \Pi_{n} e^{n}_{1})\ud r.
\end{align*}
Since by Assumption~\ref{ass:selection_coefficient} \( n^{-2 +
\frac{d}{2}}| s_{n}| \leqslant 2 n^{-2 + \frac{d}{ 2}},\) and since \(
\sup_{ n \geqslant n_{0}(1)} \| \Pi_{n} e^{n}_{1} \|_{\infty} < \infty \) as well as \( 0
\leqslant  Y_{r} \leqslant n^{\varrho}\), we can rewrite the bound as:
\begin{align*}
 \EE \langle  M^{n}(\Pi_{n} e_{1}^{n}) \rangle_{t \wedge
    \tau_{R}} &  \lesssim   
  \EE \int_0^{t} \langle \Pi_{n}^{3}Y^{n}_{r \wedge \tau_{R}}, 
\Pi_{n}^{2} e_{1}^{n} \rangle  +  \langle \Pi_{n}(Y^{n}_{r \wedge
\tau_{R}} \Pi_{n} e_{1}^{n} ) , 1 \rangle +  \langle  \Pi_{n}^{3}Y^{n}_{r
\wedge \tau_{R}},  \Pi_{n}^{2} e_{1}^{n} \rangle \ud r \\
& \lesssim \EE \int_{0}^{t} \langle Y_{r \wedge \tau_{R}}, \Pi_{n}
e^{n}_{1} \rangle \ud r.
\end{align*}
  Therefore, by Gronwall's inequality, there exists a \(C>0\) such that
  \begin{equation}\label{eqn:proof-tightness-RSBM-apriori-L2-bound}
    \sup_{0 \leq t \leq T} \EE| \langle Y^{n}_{t \wedge \tau_{R}},
    \Pi_{n} e^{n}_{1} \rangle|^{2} \lesssim e^{C(1 + R n^{-\varrho_{0}})}.
  \end{equation} 
  It follows that if $n \geq R^{\frac{1}{\varrho_{0}}}$
  \[ 
    \PP \Big( \sup_{0 \leq t
    \leq T} | \langle Y^{n}_{t}, \Pi_{n} e^{n}_{1} \rangle| \geq R \Big)
    = \PP \Big(  | \langle Y^{n}_{\tau_{R} \wedge T}, \Pi_{n} e^{n}_{1}
    \rangle| = R \Big) \lesssim R^{-2}.
  \] 
  This concludes the proof of the compact containment
  condition \eqref{eqn:compact_containment}.

\textit{Step 2.} Next we want to prove~\eqref{eqn:one_dimensional_tightness},
so let us fix \(k \in \NN, \gamma>0\) and \( \delta>0 \). In view of the calculations from Step~1 there exist \(R(\gamma),
  n(\gamma)\) for which \eqref{eqn:compact_containment} holds (with \( \delta
\) replaced by \( \gamma \)). In addition, for some \(
n(\gamma, \delta) \geqslant n (\gamma) \) we may also assume that
  \[ \forall n \geq n(\gamma, \delta)  \qquad \| e_{k} - \Pi_{n}
  e^{n}_{k} \|_{L^{\infty}} \leq \frac{\delta}{2R(\gamma)}.\] 
Hence, for every  \(n \geq n(\gamma, \delta)\)
  \[ 
    \PP \Big(  | \langle Y^{n}_{\tau_{n}+ \delta_{n}}, e_{k} \rangle -
    \langle Y^{n}_{\tau_{n}}, e_{k} \rangle| \geq \delta \Big) \leq \gamma
    + \PP \Big(  | \langle Y^{n}_{\tau_{n}+ \delta_{n}}, \Pi_{n}
      e_{k}^{n} \rangle -
  \langle Y^{n}_{\tau_{n}}, \Pi_{n}e_{k}^{n} \rangle| \geq \delta \Big). \] 
 Now, using the definition of \( R(\gamma) \) (and writing for simplicity \(R\) instead of \(R(\gamma)\)):
  \begin{align*}
    \PP \Big(  | \langle Y^{n}_{\tau_{n}+ \delta_{n}}, \Pi_{n}
    e_{k}^{n} \rangle - \langle Y^{n}_{\tau_{n}}, \Pi_{n}e_{k}^{n}
  \rangle| \geq \delta \Big) \leq & \gamma + \PP \Big(  | \langle Y^{n}_{(\tau_{n}+ \delta_{n})\wedge \tau_R}, \Pi_{n}
    e_{k}^{n} \rangle - \langle Y^{n}_{\tau_{n}\wedge \tau_R}, \Pi_{n}e_{k}^{n}
  \rangle| \geq \delta \Big).
  \end{align*}
 At this point, via the representation of
Lemma~\ref{lem:discrete_martingale_problem_rsbm_scaling} we have that
  \begin{align*}
    \langle Y^{n}_{(\tau_{n} + \delta_{n}) \wedge \tau_{R}}- Y^{n}_{\tau_{n} \wedge \tau_{R}}, \Pi_{n} e_{k}^{n} \rangle
   = \int\limits_{\tau_{n} \wedge \tau_{R} }^{(\tau_{n}+ \delta_{n}) \wedge
\tau_{R}} \!\!\!  \langle Y^{n}_{r}, & \Pi_{n} \mH_{n} e_{k}^{n} \rangle - n^{-\varrho} \langle
    \big(\Pi_{n}Y^{n}_{r}\big)^{2}, \xi^{n}_{e} \Pi_{n}^{2}
    e^{n}_{k} \rangle \ud r \\ 
     & + M^{n}_{\tau_{n}+ \delta_{n}}(\Pi_{n} e_{k}^{n}) -
    M^{n}_{\tau_{n}}(\Pi_{n} e_{k}^{n}).
  \end{align*}
Hence we obtain 
\begin{align*}
  \PP \Big(  | \langle Y^{n}_{(\tau_{n}+ \delta_{n})\wedge \tau_R}, \Pi_{n}
e_{k}^{n} \rangle - & \langle Y^{n}_{\tau_{n}\wedge \tau_R}, \Pi_{n}e_{k}^{n}
\rangle| \geq \delta \Big) \\
\leqslant & \ \PP \bigg(  \bigg\vert \int_{\tau_{n} \wedge \tau_{R} }^{(\tau_{n}+ \delta_{n}) \wedge
\tau_{R}} \!\!\!  \langle Y^{n}_{r}, \Pi_{n} \mH_{n} e_{k}^{n} \rangle - n^{-\varrho} \langle
    \big(\Pi_{n}Y^{n}_{r}\big)^{2}, \xi^{n}_{e} \Pi_{n}^{2}
    e^{n}_{k} \rangle \ud r \bigg\vert \geq \frac{\delta}{2} \bigg) \\
& \ + \frac{4}{\delta^{2}} \EE |M^{n}_{(\tau_{n}+ \delta_{n}) \wedge
\tau_{R}}(\Pi_{n} e_{k}^{n}) -
    M^{n}_{\tau_{n} \wedge \tau_{R}}(\Pi_{n} e_{k}^{n})|^{2},
\end{align*}
where we used Markov's inequality in the last line. Following the calculations
of Step~1 and using that both \( \Pi_{n} \mH_{n} e^{n}_{k} \) and \( \Pi_{n}
e^{n}_{k} \) are uniformly bounded in \( \mC^{\ve}(\TT^{d}) \) for some \(
\ve >0 \), we find
\begin{align*}
 \bigg\vert \int_{\tau_{n} \wedge \tau_{R} }^{(\tau_{n}+ \delta_{n}) \wedge
\tau_{R}} \!\!\!  \langle  Y^{n}_{r}, \Pi_{n} \mH_{n} e_{k}^{n} \rangle - n^{-\varrho} \langle
    \big(\Pi_{n}Y^{n}_{r}\big)^{2}, & \xi^{n}_{e} \Pi_{n}^{2}
    e^{n}_{k} \rangle \ud r \bigg\vert  \lesssim \int_{\tau_{n}}^{\tau_{n} + \delta_{n}} \langle Y^{n}_{r \wedge
\tau_{R}}, 1 \rangle \ud r \lesssim \delta_{n} R(\gamma).
\end{align*}
Similar calculations for the quadratic variation show that
\begin{align*}
\EE |M^{n}_{(\tau_{n}+ \delta_{n}) \wedge \tau_{R}}(\Pi_{n} e_{k}^{n}) -
M^{n}_{\tau_{n} \wedge \tau_{R}}(\Pi_{n} e_{k}^{n})|^{2} \lesssim
\delta_{n} R(\gamma).
\end{align*}
Collecting all the bounds we proves so far and passing to the limit \( n \to \infty\) we obtain that
\begin{align*}
\limsup_{n \to \infty}  \PP \Big(  | \langle Y^{n}_{\tau_{n}+ \delta_{n}}, e_{k} \rangle -
\langle Y^{n}_{\tau_{n}}, e_{k} \rangle| \geq \delta \Big) \leq 2 \gamma .
\end{align*}
  Since $\gamma$ is arbitrary, this proves
\eqref{eqn:one_dimensional_tightness}.

\textit{Step 3.}  
   So far any limit point $Y$ of the sequence $Y^n$ lies in $\DD([0,T]; \mM(\TT^d))$. Since $\mM(\TT^d)$ is endowed with the weak
   topology, to prove that actually $Y \in C([0,T]; \mM(\TT^d))$, it is
   sufficient to show that for any continuous function $\varphi$, $\langle
   Y_{t}, \varphi \rangle$ is continuous in time. Here one can apply a
   criterion \cite[Theorem 3.10.2]{EthierKurtz1986} according to
   which it is sufficient to prove that the maximum size of a jump converges
   weakly to zero. In our case such convergence happens even almost surely:
   \begin{align*}
     |\langle Y^n_t, \varphi \rangle - \langle Y^n_{t-},
     \varphi \rangle | \lesssim n^{\vr-d -\eta} \| \varphi \|_{C(\TT^d)} = n^{-2} \| \varphi \|_{C(\TT^d)}.
   \end{align*}
   This follows from the definition of the generator, as well as
 Equation~\eqref{eqn:definition_of_intensity_eta_rsbm}) for $ \eta $, which
imply that jumps are bounded by 
  $\| Y^n_t - Y^n_{t-}\|_{L^\infty} \lesssim n^{\vr-\eta}
    \lesssim 1.$
  Since a jump has impact only in a  ball $B_n(x)$ for some $x \in \TT^d$,
  integrating \(\varphi\) over such ball guarantees the claimed estimate.
\end{proof}

Finally we are in position to deduce Theorem~\ref{thm:convergence_rsbm}.

\begin{proof}[Proof of Theorem~\ref{thm:convergence_rsbm}]

  By Proposition~\ref{prop:tightness_rough_SBM} the sequence $Y^{n}(\omega)$
  is tight, for every \(\omega \in \Omega\), under
  Assumption~\ref{assu:probability-space}. 
  It remains to show that, for fixed \(\omega \in \Omega\), every
  limit point satisfies the martingale problem for
  the rough superBrownian motion as in
  Definition~\ref{defn:rough_superbrownian_motion}, which is covered by
Steps~\(1\) and \( 2 \),  and that such solutions are unique, which is covered by Step~\(3\).

  \textit{Step 1.} As in the proof of Proposition~\ref{prop:tightness_rough_SBM}, since \(\omega \in \Omega\) is
  fixed we omit writing it. Moreover it is sufficient to fix a finite but
  arbitrary time horizon \(T>0\) and check the martingale property until that time. Assume
  that (up to taking a subsequence and
  applying the Skorohod representation theorem) \(Y^{n}
  \to Y\) almost surely in \(\DD([0,T]; \mM(\TT^{d}))\).
Recall that the domain \( \mD \) of the Anderson Hamiltonian is composed of
finite linear combinations of eigenfunctions, hence we have to check the
martingale property for \( \varphi \) of the form
$ \varphi = \sum_{i = 1}^{m} \alpha_{k_{i}} e_{k_{i}},$
for some \( m \in \NN, \ k_{1}, \ldots k_{m} \in \NN, \ \alpha_{k_{i}} \in
\RR,\) and where \( \{e_{k}\}_{k \in \NN} \) is the set of eigenfunctions of \(
\mH.\)
Now consider the approximate eigenfunctions \( e^{n}_{k} \) from
Theorem~\ref{thm:semidiscrete-pam-approximation} and define \( \varphi^{n}  \)
by
$ \varphi^{n} = \sum_{i = 1}^{m} \alpha_{k_{i}} e^{n}_{k_{i}}$.
Then Theorem~\ref{thm:semidiscrete-pam-approximation} implies that for some \( \ve>0 \)
\begin{align*}
\lim_{n \to \infty} \Pi_{n} \varphi^{n} = \varphi, \qquad \lim_{n \to \infty} \Pi_{n} \mH_{n}
\varphi^{n} = \mH \varphi = \sum_{i = 1}^{m} \alpha_{k_{i}}
\lambda_{k_{i}} e_{k_{i}}, \quad \text{ in } \qquad \mC^{\ve}.
\end{align*}
  In this setting, with the notation of
Lemma~\ref{lem:discrete_martingale_problem_rsbm_scaling}, one has almost surely
  \begin{align*}
    M^{\varphi}_{t} & = \langle Y_{t,0}, \varphi \rangle - \int_{0}^{t} \langle
    Y_{s}, \mH \varphi \rangle \ud s \\
& = \lim_{n \to \infty} \bigg[ \langle Y^{n}_{t, 0}, \Pi_{n}
    \varphi^{n} \rangle -\int_{0}^{t} \langle Y^{n}_{r}, \Pi_{n}  \mH_{n} \varphi^{n} \rangle  {-} n^{-\varrho} \langle
    \big(\Pi_{n}^{3}Y^{n}_{r}  \big)^{2}, \xi^{n}
    \Pi_{n}^{2} \varphi^{n} \rangle \ud r \bigg]\\
    & = \lim_{n \to \infty} \bigg[ \langle Y^{n}_{t, 0}, \Pi_{n}
      \varphi^{n} \rangle -\int_{0}^{t} \langle \mA_{n}(Y^{n}_{r})
      +\Pi_n[ \xi^{n}_{e}
      \Pi_n^{3} Y^{n}_{r} ], \Pi_{n} \varphi^{n} \rangle  {-} n^{-\varrho} \langle
      \big(\Pi_{n}^{3}Y^{n}_{r} \big)^{2}, \xi^{n}_{e}
    \Pi_{n}^{2} \varphi^{n} \rangle \ud r \bigg]\\
    & = \lim_{n \to \infty} M^{n}_{t}(\Pi_{n} \varphi^{n}). 
  \end{align*}
  Here the convergence to zero of the non-linear term follows as in
  the proof of Proposition~\ref{prop:tightness_rough_SBM}:
  \begin{align*}
    \langle \big(\Pi_{n}Y^{n}_{r} \big)^{2}, \xi^{n}_{e}
    \Pi_{n}^{2} \varphi^{n} \rangle \lesssim n^{-\varrho + d + \frac d 2} \|
    \Pi_{n}\varphi^{n} \|_{\mC^{\ve}} \langle
    Y^{n}_{r}, 1 \rangle^{2} \to 0,
  \end{align*}
  by the assumption on \(\varrho\). Our aim is to establish the martingale
property for \( M^{\varphi}_{t}\) with respect to the filtration \(
\mF_{t} \) generated by
\( Y_{t} \). The almost sure convergence \( M^{n}_{t}(\Pi_{n}
\varphi^{n}) \to
M^{\varphi}_{t} \) is not sufficient. Instead, we will pick a sequence of stopped
martingales \( \widetilde{M}^{n}_{t}(\Pi_{n} \varphi^{n}) \), such that \(
\widetilde{M}^{n}_{t}(\Pi_{n} \varphi^{n}) \to M^{\varphi}_{t} \) almost surely and
in \( L^{1} \), for all \( t \in [0, T] \). As we will see, the additional
convergence in \( L^{1} \) will guarantee that the limit \( M^{\varphi} \) is a
martingale. Hence, let us define the following stopping time, for any path \(
z \in \DD([0,T]; \mM(\TT^{d})) \):
  \[ \tau_R (z) \colon = \inf \{ t \in [0,T]  : | \langle z_t, 1 \rangle |
     \geqslant R \}.\]
Since \( Y \) takes values in \( \DD([0,T]; \mM(\TT^{d})) \) we have that
$ \lim_{R \rightarrow \infty} \tau_R (Y) = \infty.$
  Now, Lemma \ref{lem:stopping-times} guarantees that almost surely (that is,
on the events in which \( Y^{n} \to Y \) in \( \DD([0,T]; \mM(\TT^{d}) ) \)) for any \( 0 < \ve < R
\):
\begin{align*}
\tau_{R - \ve} (Y) \leqslant \liminf_{n \to \infty} \tau_{R}(Y^{n}).
\end{align*}
We deduce, using the monotonicity \( \tau_{R}(z) \leqslant
\tau_{R^{\prime}}(z)  \) if \( R \leqslant R^{\prime} \), that for \(
\varrho_{0} = \varrho - \frac{d}{2} -d>0\) (by Assumption~\ref{ass:sparsity})
almost surely:
  $ \lim_{n \rightarrow \infty} \tau_{n^{\varrho_{0}}} (Y^n) = \infty.$
  Now, Equation~\eqref{eqn:proof-tightness-RSBM-apriori-L2-bound} implies that
  \[ \sup_{n \in \NN} \sup_{0 \leqslant t \leqslant T} \mathbb{E}
     \left[ \left| \left\langle Y^n_{t \wedge \tau_{n^{\varrho_0}} (Y^n)}, 1
     \right\rangle \right|^2 \right] < \infty . \]
  In particular, following the calculations of
Proposition~\ref{prop:tightness_rough_SBM} the sequence of stopped martingales
  $ \big\{ M^n_{t \wedge \tau_{n^{\varrho_0}} (Y^n)} (\Pi_{n}
\varphi^{n}) \big\}_{n \in \NN} $ is uniformly integrable:
\begin{align*}
\sup_{n \in \NN} \sup_{t \in [0, T]} \EE | M^n_{t \wedge \tau_{n^{\varrho_0}} (Y^n)} (\Pi_{n}
\varphi^{n}) |^{2} < \infty.
\end{align*}
 Moreover, following from the previous observations
\(\widetilde{M}^{n}_{t}(\Pi_{n} \varphi^{n}) \colon =M^n_{t \wedge \tau_{n^{\varrho_0}} (Y^n)} (\Pi_{n}
\varphi^{n}) \) converges almost surely to $M_t^{\varphi}$. The uniform
integrability implies that the convergence holds also in $L^1$.
In order to conclude that \( M^{\varphi} \) is a martingale with respect to \( \mF
\) it suffices to show that for every \( s <t, \ m \in \NN, \ 0 \leqslant s_{1} \leqslant
\cdots \leqslant  s_{m} \leqslant s \) and every $ h \in C_{b}(\RR^{m}; \RR) $ 
\begin{align*}
\EE \Big[ M^{\varphi}_{t} h (Y_{s_{1}}, \ldots , Y_{s_{m}}) \Big] = \EE
\Big[ M^{\varphi}_{s} h(Y_{s_{1}}, \ldots , Y_{s_{m}})  \Big].
\end{align*} 
From the convergence in \( L^{1} \) and almost surely that we just proved we obtain that
\begin{equation}\label{eqn:martingale-property-preserved}
\begin{aligned}
\EE \Big[ M^{\varphi}_{t} h (Y_{s_{1}}, \ldots , Y_{s_{m}}) \Big] & =
\lim_{n \to \infty} \EE \Big[ \widetilde{M}^{n}_{t} (\Pi_{n}
\varphi^{n}) h (Y^{n}_{s_{1} \wedge \tau_{n^{\varrho_{0}}}(Y^{n})}, \ldots ,
Y^{n}_{s_{m} \wedge \tau_{n^{\varrho_{0}}}(Y^{n})}) \big] \\
& = \lim_{n \to \infty} \EE \Big[ \widetilde{M}^{n}_{s} (\Pi_{n}
\varphi^{n}) h (Y^{n}_{s_{1} \wedge \tau_{n^{\varrho_{0}}}(Y^{n})}, \ldots ,
Y^{n}_{s_{m} \wedge \tau_{n^{\varrho_{0}}}(Y^{n})}) \big] \\
& =\EE \Big[ M^{\varphi}_{s}  h (Y_{s_{1}}, \ldots ,
Y_{s_{m}}) \big],
\end{aligned}
\end{equation}
where in the second line we used the martingale property for \(
\widetilde{M}^{n}(\Pi_{n} \varphi^{n}) \).

  \textit{Step 2.}  Now we have to show that the martingale has the correct quadratic
  variation. Here the problem is that we do not control moments of
$\widetilde{M}^n_t(\Pi_{n} \varphi^{n})$ higher than
  the second one. So proving that the martingale property of $( \widetilde{M}^n_t)^2 -
  \langle \widetilde{M}^n \rangle_t$ is preserved in the limit does not follow from the
  same arguments we just used. Instead we stop the martingales in a different
way. Consider the following stopping times as a sequence indexed by \( R \in
\NN \):
  \[ \{ \tau_R (Y^n) \wedge T \}_{R \in \NN} \in [0, T]^{\NN} .
  \]
Here the space \( [0, T]^{\NN} \) is endowed with the product topology and under this
topology it is both compact and separable. In particular, since we are assuming that \(
Y^{n} \to Y \) in distribution in \( \DD([0,T]; \mM(\TT^{d})) \), the
sequence
$( \{ \tau_R (Y^n) \wedge T \}_{R \in \NN} , Y^{n})_{n \in \NN}$
is tight in the space
$[0,T]^{\NN} \times \DD([0,T]; \mM(\TT^{d})).$
Hence let \( (\{ \overline{\tau}_{R} \}_{R \in \NN}, Y)  \) be any limit point
of the joint distribution. Since the space \( [0, T]^{\NN} \times
\DD([0,T]; \mM(\TT^{d})) \) is separable, by the Skorohod representation
theorem, up to changing probability space, we can pick a subsequence \(
n_{k}\), for \(k \in \NN \) such that almost surely
\begin{align*}
\lim_{k \to \infty} ( \{ \tau_R (Y^{n_{k}}) \wedge T \}_{R \in \NN} , Y^{n_{k}})
=( \{ \overline{\tau}_{R} \}_{R \in \NN}, Y), \quad \text{ in } \quad [0,
T]^{\NN} \times \DD([0, T] ; \mM(\TT^{d})).
\end{align*}
The limiting random variables still satisfy the ordering:
$ \bar{\tau}_R \leqslant \bar{\tau}_{R + m}, \ \forall m \in
     \NN,$ as well as, by Lemma~\ref{lem:stopping-times}:
\begin{equation}\label{eqn:comparison-stopping-times}
 \tau_{R - \varepsilon} (Y) \wedge T \leq \bar{\tau}_R \leq \tau_{R + \varepsilon} (Y) \wedge T, \qquad \forall \varepsilon > 0. 
\end{equation}
  Now, the same calculations leading to
Equation~\eqref{eqn:martingale-property-preserved} show that for any \( R \in
\NN \) the stopped martingales $M^{n_{k}}_{t \wedge \tau_R
(Y^{n_{k}})}(\Pi_{n_{k}} \varphi^{n_{k}})$ converge to $M_{t \wedge \bar{\tau}_R}^{\varphi}$ almost surely
and in \( L^{1} \) (note that now the martingales \( M^{n_{k}}_{t \wedge
\tau_{R}(Y^{n_{k}})}(\Pi_{n_{k}} \varphi^{n_{k}}) \) are even bounded).
Similarly we obtain that \( M_{t
\wedge \bar{\tau}_R}^{\varphi} \) is a martingale with respect to the
filtration \( \overline{\mF}_{t}^{R} \)
generated by $Y_{t \wedge \bar{\tau}_R}$. Following the calculations of
Proposition~\ref{prop:tightness_rough_SBM} we observe that
\begin{align*}
\langle M^{n_{k}}(\Pi_{n_{k}} \varphi^{n_{k}}) \rangle_{t \wedge
\tau_{R}(Y^{n_{k}})} \leqslant
C \int_{0}^{t \wedge \tau_{R}(Y^{n_{k}})} \langle Y_{s}^{n_{k}}, 1 \rangle \ud s,
\end{align*}
for some deterministic \( C>0 \). In particular, following once more the
calculations of Proposition~\ref{prop:tightness_rough_SBM}, we deduce that the martingale
$\big( M^{n_{k}}_{t \wedge \tau_{R}(Y^{n_{k}})}(\Pi_{n_{k}}
\varphi^{n_{k}}) \big)^{2} - \langle
M^{n_{k}} (\Pi_{n_{k}} \varphi^{n_{k}}) \rangle_{t \wedge
\tau_{R}(Y^{n_{k}})} $
is bounded and converges almost surely to
$\big( M^{\varphi}_{t \wedge \overline{\tau}_{R}}\big)^{2} - \int_{0}^{
 t \wedge \overline{\tau}_{R}} \langle Y_{s}, \varphi^{2} \rangle \ud s.$
We then conclude that
  $ \langle M_{\cdot \wedge \bar{\tau}_R}^{\varphi} \rangle_t = \int_0^{t \wedge
     \bar{\tau}_R} \langle Y_s, \varphi^2 \rangle \ud s  .$
Now, defining \( t^{n}_{k} = \frac{k T}{n} \), for \( k \leqslant n \in \NN \),
we can view the quadratic variation as the limit in probability:
\begin{align*}
\langle M^{\varphi}_{\cdot \wedge \overline{\tau}_{R}}  \rangle_{t} = \mathbb{P}- \lim_{n
     \rightarrow \infty} \sum_{k = 0}^{n} (M_{t \wedge \bar{\tau}_R \wedge
     t_{k + 1}^{n}} - M_{t \wedge \bar{\tau}_R \wedge t_k^{n}})^2
\end{align*}
  Similarly for the martingale whose quadratic variation we would actually
like to compute:
  \[ \langle M^{\varphi} \rangle_t =\mathbb{P}- \lim_{n \rightarrow \infty}
     \sum_{k =0}^{n} (M_{t \wedge t_{k + 1}^{n}} - M_{t \wedge t_k^{n}})^2.\]
     Now, for any $\delta > 0$ and \( t \in [0, T) \) we can choose an $R \in \NN$ such that
  $ \mathbb{P} (\bar{\tau}_R > t) \geq 1 - \delta, $
  by comparison with the stopping time $\tau_{R - \varepsilon} (Y)$ for any
  $\varepsilon > 0$ (see Equation~\eqref{eqn:comparison-stopping-times}) and
since \( \lim_{R \to \infty} \tau_{R}(Y) = \infty \). So we conclude
that for any \( t < T\) it holds that
  $ \mathbb{P} \big( \langle M \rangle_t = \int_0^t \langle Y_s, \varphi^2
     \rangle \ud s \big) \geq 1 - \delta. $
Since \(\delta, T>0 \) are arbitrary we obtain the correct quadratic variation for
all times.

  \textit{Step 3.} We conclude by explaining the uniqueness in law of any process
  \(Y\) satisfying the martingale problem of the rough superBrownian motion (in the following
  as always \(\omega \in \Omega\) is fixed, and we omit from writing it). The uniqueness is the consequence of a duality argument. For
  any \(\varphi \geq 0, \varphi \in C^{\infty}\) we
   find a process \(t \mapsto U_{t} \varphi\) such that
  \begin{equation}\label{eqn:proof-convergence-rsbm-duality}
    \EE \Big[ e^{- \langle Y_{t}, \varphi \rangle} \Big] = e^{- \langle
    Y^{0}, U_t \varphi \rangle}.
  \end{equation}
  Hence the distribution of \(\langle Y_{t}, \varphi \rangle\) is uniquely
  characterized by its Laplace transform. This also characterizes the law of
  the entire process \(\langle Y_{t}, \varphi \rangle\) through a monotone class argument (see \cite[Lemma 3.2.5]{DawsonMaisonneuve1993SaintFlour}), proving the required
  result. 

  We are left with the task of describing the process \(U_{t} \varphi\). This
  is the solution, evaluated at time \(t \geq 0\), of the
   non-linearly damped parabolic equation
  \begin{align*}
    \partial_{t}(U_{\cdot} \varphi) = \mH (U_{\cdot} \varphi) -
    \frac{1}{2} (U_{\cdot} \varphi)^{2}, \qquad U_0 \varphi = \varphi,
  \end{align*}
  where we consider solutions in the mild sense, namely
    $U_{t} \varphi = e^{t \mH} \varphi - \frac{1}{2} \int_{0}^{t} e^{(t {-} s) \mH}
    (U_{s} \varphi)^{2} \ud s,$
  as in Lemma~\ref{lem:solution-dual-equation}. To obtain
  Equation~\eqref{eqn:proof-convergence-rsbm-duality} consider some
  \(\zeta> 0\) and a process \(\psi \in C([0,T]; \mC^{\zeta})\) of the form
  $ \psi_{t} = e^{t \mH } \psi_{0} + \int_{0}^{t}  e^{(t-s) \mH} f_{s}\ud s,$
  with \(f \in C([0,T]; \mC^{\zeta}), \psi_{0} \in \mC^{\zeta}.\)
  Now approximate \(f\) through a piecewise constant function in time
  \(\widetilde{f}\) and in turn approximate both \( \widetilde{f}\) and
  \(\psi_{0}\) via a finite number of
eigenfunctions (here we use the density of the domain proved in
  Lemma~\ref{lem:domain_for_anderson_hamiltonian}). Using the continuity
  of the semigroup as in
  Equation~\eqref{eqn:proof-density-second-regularization}, it
  follows from the definition of the rough superBrownian motion that for \(0
  \leq s \leq t\):
  \begin{align*}
    \langle Y_{s}, \psi_{t-s} \rangle - \langle Y_{0}, \psi_{t}\rangle -
    \int_{0}^{s} \langle Y_{r} , f_{r} \rangle \ud r =:
    \widetilde{M}_{s}(\psi)
  \end{align*}
  is a continuous martingale with quadratic variation
    $\langle \widetilde{M}(\psi) \rangle_{s} = \int_{0}^{s} \langle
    Y_{r}, \psi_{t -r}^{2} \rangle \ud r.$
  We apply this observation together with It\^o's formula to deduce that
 $ [0, t] \ni s \mapsto e^{- \langle Y_{s}, U_{t-s} \varphi \rangle}$ 
  is a martingale on \([0,t]\). In particular, this implies
  Equation~\eqref{eqn:proof-convergence-rsbm-duality} and concludes the proof.
\end{proof}

The following result states the well-posedness of the dual PDE to the rough
superBrownian motion. The proof is identical to that of
\cite[Proposition 4.5]{PerkowskiRosati2019RSBM}.

\begin{lemma}\label{lem:solution-dual-equation}
  Under Assumption~\ref{assu:probability-space}, fix \(\omega \in \Omega\). For any \(\varphi \geq 0, \varphi \in C^{\infty}\), time horizon \(T>0\)
  and \(\zeta< 2- \frac d 2\), there exists a unique function  
  \((t,x) \mapsto (U_{t}^{\omega} \varphi) (x)\) such that  \(U^{\omega} \varphi \in C([0,T]
  ; \mC^{\zeta})\), where
  \begin{align*}
    U_{t}^{\omega} \varphi = e^{t \mH^\omega} \varphi - \frac{1}{2}
    \int_{0}^{t} e^{(t-s) \mH^\omega}
    ( U_{s}^{\omega} \varphi)^{2} \ud s.
  \end{align*}
\end{lemma}

We conclude the section with a consideration on stopping times and convergence
in the Skorohod topology, which is used in the proofs above. The proof of this
lemma follows from the definition of the Skorohod distance.
\begin{lemma}\label{lem:stopping-times}
Consider \( T>0 \) and \( \{z^{n}\}_{n \in \NN}, z \in \DD([0,T]; \RR ) \) such that $z^n \rightarrow z$ in $\mathbb{D} ([0,
  T], \RR)$. Define, for \( R >0 \):
\begin{align*}
\tau_{R}(z) = \inf \{ t \in [0,T]  \ \colon \ | z_{t} |
\geqslant R \},
\end{align*}
and identically also \( \tau_{R}(z^{n}) \), with the convention that \(
\inf \emptyset = \infty \). Then, for any $\varepsilon > 0$ 
  \[\tau_{R - \varepsilon} (z) \leqslant  \liminf_{n \rightarrow \infty} \tau_R (z^n)
     \leqslant \limsup_{n \rightarrow \infty} \tau_R (z^n) \leqslant \tau_{R +
     \varepsilon} (z) . \]
\end{lemma}

\section{Scaling to Fisher-KPP}\label{sec:scaling_fkpp}

As in Section~\ref{sec:scaling_to_rsbm}, we fix one realization
\(\omega \in \Omega\) of the environment and work conditionally on that realization. 
The first step towards the scaling limit is to restate the martingale problem
of Lemma~\ref{lem:martingale_problem_model_description} under the present
scaling. The proof of the next result is an immediate consequence of the aforementioned lemma.

\begin{lemma}\label{lem:martingale_problem_discrete_fkpp}
  In the setting of Theorem~\ref{thm:convergence_fkpp} fix any
  \(\omega \in \Omega\). For all $\varphi \in L^{\infty} (\TT^d)$, the process
  $t\mapsto \langle X^n_t(\omega), \varphi\rangle$ satisfies 
  \begin{equation}\label{eqn:discrete_mp_drift_fkpp}
  \begin{aligned}
    \langle X^{n}_{t,s}(\omega), \varphi \rangle = \int_{s}^{t} \langle
    \mA_{n}(X^{n}_{r}(\omega)), \varphi \rangle + \langle \Pi_n \big[ \bar{\xi}(\omega)
      ( \Pi_n^{3} X^{n}_{r}(\omega) -(\Pi_n^{3} X^n_r(\omega))^2 )\big], \varphi  \rangle \ud r +
      M^{n}_{t,s}(\varphi),
  \end{aligned}
  \end{equation}
  where \(M^{n}_{\cdot}(\varphi)\) is a centered square integrable martingale with
  predictable quadratic variation
\begin{equation}\label{eqn:discrete_mp_quadratic_variation_fkpp}
\begin{aligned}
    \langle  M^{n}(\varphi) \rangle_{t} = n^{-\eta - d + 2}
  \int_0^{t}  & \langle (1 {+} s_{n}(\omega)) \Pi_{n}^{3}X^{n}_{r}(\omega), (
  \Pi_{n} \varphi)^2 {-} 2   \Pi_{n} (\varphi)
  \Pi_{n}(X^{n}_{r}(\omega) \varphi) \rangle \\
   & + \langle \big( \Pi_{n}(X^{n}_{r}(\omega) \varphi )
  \big)^2, 1 \rangle \\
& - \langle s_{n}(\omega) (\Pi_{n}^{3}X^{n}_{r}(\omega))^2,
   (\Pi_{n}\varphi)^2 {-} 2
  \Pi_{n}(\varphi) \Pi_{n}(X^{n}_{r}(\omega) \varphi) \rangle \ud r. 
\end{aligned}
\end{equation}
\end{lemma}

Now we are able to prove tightness of our process.

\begin{proposition}\label{prop:tightness_fkpp}
  In the setting of Theorem~\ref{thm:convergence_fkpp} fix any
  \(\omega \in \Omega, T>0\) and $\alpha$ such that
  \begin{align*}
    \begin{cases}
      \alpha \in (0, 1/2) \qquad & \text{ if } d=1, \\
      \alpha \in (0, \min \{ \eta, 1\}) \qquad &\text{ if } d=2.
    \end{cases}
  \end{align*}
  Then the sequence $\{ s\mapsto
  \Pi_n X^n_s(\omega)\}_{n \in \NN}$ is tight in the space
$L^2([0,T];B^\alpha_{2,2}).$
In addition, the sequence $\{s \mapsto X^n_s(\omega)\}_{n \in \NN}$ is tight
  in $\DD([0,T]; \mM(\TT^d))$, and any limit point lies in \(C([0,T];
  \mM(\TT^{d}))\). 
\end{proposition}

To prove the proposition we will make use of the regularizing properties of the
semigroup \( e^{t \mA_{n}} \) as described by the following result.
\begin{lemma}\label{lem:parabolic-Schauder-simplified}
For any \( \gamma \in [0, 1), p \in [1, \infty], T >0 \) and \( \alpha \in \RR \) one can bound,
uniformly over \( n \in \NN, \ \varphi \in \mC^{\alpha}_{p}, t \in [0, T]\):
\[ \| \Pi_{n} e^{t \mA_{n}} \varphi \|_{\mC^{\alpha + \gamma}_{p}} \lesssim
t^{-\frac{\gamma}{2}} \| \varphi
\|_{\mC^{\alpha}_{p}}.\] 
\end{lemma}
\begin{proof}
We can bound
$\|  \mP_{n} e^{t \mA_{n}}\Pi_{n} \varphi \|_{\mC^{\alpha + \gamma}_{p}}
 \lesssim \| \mP_{n} e^{t \mA_{n}} \varphi \|_{\mC^{\alpha +
\gamma}_{p}}  \lesssim t^{- \frac{\gamma}{2}}  \| \varphi
\|_{\mC^{\alpha}_{p}},$
where in the first step we applied
Corollary~\ref{cor:regularity_gain_convolution_with_characteristic_functions}
and in the last step the large scale estimate of
Proposition~\ref{prop:schauder_estimates}. Instead, on small scales we find
$\| \mQ_{n} e^{t \mA_{n}} \Pi_{n} \varphi \|_{\mC^{\alpha + \gamma}_{p}}
 \lesssim n^{\gamma}\| \mQ_{n} e^{t \mA_{n}} \varphi \|_{\mC^{\alpha}_{p}}
 \lesssim t^{- \frac{\gamma}{2}} \| \varphi \|_{\mC^{\alpha}_{p}},$
where we again applied
Corollary~\ref{cor:regularity_gain_convolution_with_characteristic_functions}
and Proposition~\ref{prop:schauder_estimates}.
\end{proof}

\begin{proof}[Proof of Proposition~\ref{prop:tightness_fkpp}]
  As \(\omega \in \Omega\) is fixed, we omit writing
  it.

\textit{Step 1.} The tightness of the sequence \(X^{n}\) in \(\DD([0,T];
\mM(\TT^{d}))\) is a consequence of the bound \(0 \leq
X^{n}_{t} \leq 1\). In fact, we can apply Jakubowski's tightness criterion
\cite[Theorem 3.1]{jakubowski:1986}. The criterion
consists in proving first a compact containment condition. This is immediately
satisfied since
$\PP( \sup_{ 0 \leqslant t \leqslant T} | \langle X^{n}_{t}, 1 \rangle | >1)
=0,$
from the boundedness of \( X^{n}. \) The second and last requirement for Jakubowski's
tightness criterion is the tightness of one dimensional distributions. Namely
it suffices to prove that for any \( \varphi \in C^{\infty}(\TT^{d}) \) the sequences of process 
$ \{ t \mapsto \langle X^{n}_{t}, \varphi \rangle \}_{n \in \NN} $ 
is tight in \( \DD([0.T]; \RR) \). For this purpose we use Aldous' tightness
criterion (this is the same approach as in the proof of
Proposition~\ref{prop:tightness_rough_SBM}). Let us define
\[ D^{n}_{t, s} (\varphi) = \langle X^{n}_{t, s}, \varphi \rangle -
M^{n}_{t,s}(\varphi),\] 
where we used the notations of
Lemma~\ref{lem:martingale_problem_discrete_fkpp}. Now to prove tightness of the
one-dimensional distributions Aldous' criterion guarantees that it suffices to
show that for any sequence of stopping times \( \tau^{n} \) and any
deterministic sequence \( \delta_{n} \) with \( \delta_{n} \to 0 \) one has
\[ \forall \delta > 0 \qquad \lim_{n \to \infty} \PP \Big( |\langle
X^{n}_{\tau_{n} + \delta_{n}, \tau_{n}}, \varphi \rangle| \geqslant \delta \Big)
= 0. \] 
In particular it suffices to show that for any \( \delta> 0 \)
\[ \lim_{n \to \infty} \PP \Big( |D^{n}_{\tau_{n} + \delta_{n},
\tau_{n}}(\varphi)| \geqslant \delta \Big)
= \lim_{n \to \infty} \PP \Big( |M^{n}_{\tau_{n} + \delta_{n},
\tau_{n}}(\varphi)| \geqslant \delta \Big)=0.\]
Now by Proposition~\ref{prop:elliptic_schauder_estimates} we find that
(since \( \varphi \) is smooth)
$ \sup_{n \in \NN} \| \mA_{n} \varphi \|_{L^{\infty}} < \infty.  $ 
Hence the following deterministic bound holds (since \( 0 \leqslant
X^{n}_{t} \leqslant 1 \) ):
\[ | D_{\tau_{n} + \delta_{n}, \tau_{n}}^{n}(\varphi) | \lesssim_{\varphi}
\delta_{n},\] 
which proves the first limit. As for the second one, we observe that
\begin{align*}
\PP \Big( |M^{n}_{\tau_{n} + \delta_{n}, \tau_{n}}(\varphi)| \geqslant \delta
\Big) \leqslant \frac{1}{\delta^{2}} \EE \Big[ \langle M^{n}(\varphi) \rangle_{\tau_{n} +
\delta_{n}, \tau_{n}}  \Big] \lesssim \frac{\delta_{n}}{\delta^{2}} \to 0,
\end{align*}
where for the quadratic variation we used similar bounds as for the drift.
 Finally, to show that any limit point lies in \(C([0,T];
\mM(\TT^{d}))\) note that for any \(\varphi \in C(\TT^{d})\)
\begin{align*}
  | \langle X^{n}_{t}, \varphi \rangle - \langle X^{n}_{t-}, \varphi
  \rangle | \lesssim n^{-\eta - d} \| \varphi \|_{L^{\infty}} \leq n^{-2}\| \varphi \|_{L^{\infty}},
\end{align*}
so that the maximal jump size is vanishing as \(n \to \infty\). The continuity of
the limit points follows then through \cite[Theorem 3.10.2]{EthierKurtz1986}.

\textit{Step 2.} Tightness in the space of measures is not sufficient to make sense of the
nonlinearity in the limit. 
Hence from now on we now concentrate on proving the tightness of the sequence
$\Pi_{n}X_{s}^{n}$ in \(L^{2}([0,T]; B^{\alpha}_{2,2})\) for some \( \alpha>0
\). Our aim is to apply Simon's tightness criterion, which we recall in
Proposition~\ref{prop:simon_criterion}.
We will apply the criterion for some $\alpha^{\prime} >\alpha>\alpha^{\prime
\prime}$ and
\begin{align*}
X = B^{\alpha^{\prime}}_{2,2}, \quad  Y = B^{\alpha}_{2,2}, \quad  Z =
B^{\alpha^{\prime\prime}}_{2,2}.
\end{align*}
Then, as a first objective, we derive a uniform bound for the second moment of the
\(B^{\alpha}_{2,2}\) norm (this in particular implies boundedness of the
sequence \(\Pi_{n} X^{n}\) in \(L^{2}([0,T]; B^{\alpha}_{2,2})\)):
\begin{equation}\label{eqn:proof_tightness_fkpp_uniform_L^2_bound}
  \begin{aligned}
    \sup_{n \in \NN} \sup_{0 \leq t \leq T} \EE \| \Pi_{n} X_{t}^{n}
    \|_{B^{\alpha}_{2,2}}^{2} < \infty. 
  \end{aligned}
\end{equation}
To obtain this bound it is convenient to prove the following stronger estimate.
for \(s \in [0,T]\)
\begin{equation}\label{eqn:first_step_tightness_fkpp}
 \sup_{s \leq t \leq T}\EE \big[ \| \Pi_n X^n_t \|^2_{B^{\alpha}_{2,2}}
  \big\vert \mF_s \big] \lesssim_{T} 1 + \|\Pi_n
  X^n_s\|_{B^\alpha_{2,2}}^2, 
\end{equation}
where \(\{ \mF_{t}\}_{t \geq 0}\) is the natural filtration generated by
\(X^{n}\) (we omit the dependence of the filtration on \(n\)). We state the bound with the
conditional expectation, since in this form it is
simpler to derive, via a Gronwall-type argument.
For brevity, fix the notation $ \overline{X}^n = \Pi_n X^n.$
By the martingale representation of Lemma~\ref{lem:martingale_problem_discrete_fkpp}
and  a change of variables formula
  \begin{align*}
    \overline{X}^{n}_{t} = e^{(t -s) \mA_{n}} \overline{X}^{n}_{s} +
    \int_{s}^{t} e^{(t - r) \mA_{n}} \Pi_{n}^{2} \big[ \overline{\xi}
\big( \Pi_{n}^{2} \overline{X}^{n}_{r} - (\Pi_{n}^{2}  \overline{X}^{n}_{r})^{2}\big) \big] \ud
    r + \int_{s+}^{t} \Pi_{n} e^{(t-r) \mA_{n}} \ud M^{n}_{r},
  \end{align*}
  where the last integral is understood as a stochastic integral against a martingale measure (cf.
  \cite{Walsh1986}). For the purpose of the proof it is sufficient to consider its  one dimensional projections, that is for $\varphi \in C(\TT^d)$
  \begin{align*}
    \langle \overline{X}^{n}_{t}, \varphi \rangle = \langle 
    \overline{X}^{n}_{s}, e^{(t - s) \mA_{n}} \varphi \rangle +
    \int_{s}^{t} & \langle \Pi_{n}^{2} \big[ \overline{\xi} 
    \big( \Pi_{n}^{2}\overline{X}^{n}_{r} - (\Pi_{n}^{2} \overline{X}^{n}_{r})^{2}\big)
    \big],e^{(t - r) \mA_{n}} \varphi\rangle   \ud
    r \\
    & + \int_{s+}^{t}  \ud M^{n}_{r}(\Pi_{n} e^{(t-r) \mA_{n}} \varphi).   
  \end{align*}
  The $B^{\alpha}_{2,2}$ norm is then estimated by 
\begin{align*}
\EE \big[ \| \overline{X}^n_t \|^2_{B^{\alpha}_{2,2}} \big\vert \mF_s \big]
  \lesssim \| \overline{X}^n_s\|_{B^\alpha_{2,2}}^2 &+ \EE \bigg[ \bigg\|
  \int_s^t e^{(t-r)\mA_n}\Pi_n^2 \big[ \bar{\xi} (
    \Pi_{n}^{2}\overline{X}^{n}_{r} -(\Pi_{n}^{2}\overline{X}^n_r)^2 )\big] \ud r
    \bigg\|^2_{B^\alpha_{2,2}}\bigg\vert \mF_s \bigg]\\
    & + \EE \bigg[ \bigg\| \int_{s+}^t \Pi_{n} e^{(t-r)\mA_n}\ud M_r^n
    \bigg\|^2_{B^\alpha_{2,2}} \bigg\vert \mF_s\bigg]. 
\end{align*}
An extension of the paraproduct estimates of Lemma
\ref{lem:paraproduct-estimates} to the $B^\alpha_{p,q}$ scale (see
\cite[Theorems 2.82, 2.85]{BahouriCheminDanchin2011FourierAndNonLinPDEs})
guarantees that
$$\| f^2 \|_{B^\alpha_{2,2}} \leq 2\| f \para f \|_{B^{\alpha}_{2,2}} +
\|f\reso f\|_{B^\alpha_{2,2}} \lesssim \| f\|_{L^\infty}
\|f\|_{B^\alpha_{2,2}}.$$ 
Now we apply the Schauder estimates of
Proposition~\ref{prop:schauder_estimates}. Note that here we do not need the
real strength of the estimates, as we do not need to gain any regularity. Note
also that the estimates are proven
on the scale of \( B^{\alpha}_{p, \infty} \) spaces but extend verbatim to \(
B^{\alpha}_{p, q} \) spaces for \( q \in [1, \infty) \). Hence, using the $L^\infty$ bound on
$\overline{X}^n$ and the fact that \( \overline{\xi}\) is smooth one obtains
\begin{align*}
\Big\| e^{(t-r)\mA_n}\Pi_n^2 \big[ \bar{\xi} ( \Pi_{n}^{2}\overline{X}^{n}_{r}
-(\Pi_{n}^{2}\overline{X}^n_r)^2 )\big] \Big\|_{B^{\alpha}_{2, 2}} & \lesssim \Big\|
\Pi_n^2 \big[ \bar{\xi}( \Pi_{n}^{2}  \overline{X}^{n}_{r}
-( \Pi_{n}^{2} \overline{X}^n_r)^2 )\big] \Big\|_{B^{\alpha}_{2 ,2}} \lesssim\|  \overline{X}^{n}_{r} \|_{B^{\alpha}_{2 ,2}},
\end{align*}
so that:
\begin{align*}
\EE \bigg[ \bigg\| \int_s^t e^{(t-r)\mA_n}\Pi_n^2 \big[ \bar{\xi}
(\Pi_{n}^{2}  \overline{X}^{n}_{r} -(\Pi_{n}^{2} \overline{X}^n_r)^2 )\big] \ud r
\bigg\|^2_{B^\alpha_{2,2}}\bigg\vert \mF_s \bigg]\lesssim |t {-} s|^{2} \sup_{s \leq t \leq T}\EE \big[ \| \overline{X}^n_t \|^2_{B^{\alpha}_{2,2}} \big\vert \mF_s \big].
\end{align*}
As for the martingale term, let us introduce a parameter \(\lambda\) according
to the following definition:
\begin{align*}
  \begin{cases}
    \text{If } d=1, \qquad \eta= 1 \ \ \Rightarrow \ \ \text{ set } \lambda = 0, \\
    \text{If } d=2, \qquad \eta > 0 \ \ \Rightarrow \ \ \text{ set } \lambda =
\min \{ \eta, 1\}.
  \end{cases}
\end{align*}
Then, from the definition of the space $B^{\alpha}_{2,2}$
we have
\begin{align*}
  &\EE \bigg[ \bigg\|  \int_{s+}^t \Pi_n e^{(t-r)\mA_n}\ud M_r^n
  \bigg\|^2_{B^\alpha_{2,2}} \bigg\vert \mF_s\bigg] =\sum_{j \geq -1} 2^{2 \alpha j} \int_{\TT^d} \EE \bigg[ \bigg\vert\int_{s+}^t  \ud
  M^n_r(e^{(t-r)\mA_n}\Pi_n K_j^x)\bigg\vert^2 \  \bigg\vert \ \ \mF_s
  \bigg] \ud x
\end{align*}
where  $K_j^x$ stands for the function $K_j^x(y) = \mF_{\TT^d}^{-1} \varrho_j (x-y),$
with $\vr_j$ the elements of the dyadic partition of the unity that define the Besov spaces.
Using the predictable quadratic variation computed in Lemma~\ref{lem:martingale_problem_discrete_fkpp} one obtains, uniformly over $x$
\begin{equation}\label{eqn:proof_tightness_fkpp_qv_bound}
\begin{aligned}
  \EE \bigg[ & \bigg\vert\int_{s+}^t \ud M^n_r(e^{(t-r)\mA_n} \Pi_n K_j^x)\bigg\vert^2 \ \ \bigg\vert \ \ \mF_s \bigg]\\
  & \leqslant  n^{- \lambda}\EE \bigg[ \int_s^{t} \langle \Pi_{n}^{2} \overline{X}^{n}_{r}, (1 {+}
s_{n}) \big[ ( \Pi_{n}^2 e^{(t-r)\mA_n} K_j^x)^2 {-} 2   \Pi_{n}^2
(e^{(t-r)\mA_n} K_j^x) \Pi_{n}(X^{n}_{r} \Pi_n e^{(t-r)\mA_n} K_j^x)\big] \rangle \\
  & \qquad \qquad + \langle \big( \Pi_{n}(X^{n}_{r} \Pi_n e^{(t-r)\mA_n} K_j^x)
  \big)^2, 1 \rangle \\
  & \qquad \qquad - \langle (\Pi_{n}^{2}\overline{X}^{n}_{r})^2,
  s_{n} \big[ (\Pi_{n}^2e^{(t-r)\mA_n} K_j^x)^2 {-} 2
  \Pi_{n}^2(e^{(t-r)\mA_n} K_j^x) \Pi_{n}(X^{n}_{r} \Pi_n e^{(t-r)\mA_n} K_j^x) \big]\rangle \ud r \ \ \bigg\vert \ \ \mF_s\bigg] \\
  & \lesssim n^{- \lambda} \int_s^t \big\| \Pi_n \big\vert \Pi_{n}
  e^{(t-r)\mA_n}K_j^x \big\vert \big\|_{L^2}^2 \ud r, 
\end{aligned}
\end{equation}
since $|s_n|,|X^n|, | \overline{X}^{n} | \leq 1$. Now, for $\zeta \in \RR$, for
example via the Poisson summation formula in Lemma
\ref{lem:poisson_summation_formula} and a scaling argument on $\RR^d$
$\| K_j^x \|_{\mC^{\zeta}_1} \lesssim 2^{j \zeta}$.
Therefore by the Schauder estimates that we recalled in
Lemma~\ref{lem:parabolic-Schauder-simplified}, for $\gamma \in [0,1)$
$\| \Pi_n e^{(t-r)\mA_n} K_j^x\|_{\mC_1^{\zeta + \gamma}} \lesssim
(t{-}r)^{-\frac{\gamma}{2}} 2^{j \zeta}.$

Now, for clarity, dimension \(d=1\) and dimension  \(d=2\) will be treated separately. 
In dimension
\(d=1\) choose $-\frac 1 2 < \zeta < -\alpha$ and fix $\gamma \in (0,1)$ such that
$\zeta+\gamma>\frac 1 2$. Then, by Besov embedding, one has
\begin{align*}
 \| \Pi_n \big\vert \Pi_{n}
  e^{(t-r)\mA_n}K_j^x \big\vert \|_{L^2}^2 \leq \|\Pi_n e^{(t-r)\mA_n}
  K_j^x\|_{L^2}^2 \lesssim\|\Pi_n e^{(t-r)\mA_n} K_j^x\|^2_{\mC_1^{\zeta +
  \gamma}} \lesssim (t{-}r)^{-\gamma} 2^{2 j \zeta}.
\end{align*}

In dimension \(d=2\), we make additional use of the regularizing properties of
\( \Pi_{n} \) together with the factor \( n^{- \lambda} \) appearing in front of the
quadratic variation. Note that
Corollary~\ref{cor:regularity_gain_convolution_with_characteristic_functions}
allows only to gain one degree of regularity, which is why we have defined \(
\lambda = \min \{ 1, \eta \}\) (we have no use for additional powers of \(
n \)).
Now, choose \(\kappa > 0\)  such that
\(\alpha < \lambda - 5 \kappa\) and set $\gamma = 1 - \kappa$. Then
Corollary~\ref{cor:regularity_gain_convolution_with_characteristic_functions}
implies that
\begin{align*}
\| \Pi_{n} \big\vert \Pi_{n} e^{(t-r)\mA_n}K_j^x \big\vert  \|_{L^{2}}
  & \lesssim n^{\lambda - \kappa}\| \big\vert \Pi_{n} e^{(t-r)\mA_n}K_j^x
  \big\vert \|_{\mC^{-\lambda + 2 \kappa}_{2}},
\end{align*}
and Besov embeddings~\ref{prop:besov_embedding} additionally guarantee the
following chain of inequalities (here the main aim is to get rid of the
absolute value):
\begin{align*}
  \| \big\vert \Pi_{n} e^{(t-r)\mA_n}K_j^x
  \big\vert \|_{\mC^{- \lambda + 2 \kappa}_{2}}
  & \lesssim\| \big\vert \Pi_{n} e^{(t-r)\mA_n}K_j^x
  \big\vert \|_{\mC^{-\kappa}_{\frac{2}{1 + \lambda - 3\kappa} } } \lesssim \| \Pi_{n} e^{(t-r)\mA_n}K_j^x
  \|_{L^{\frac{2}{1 + \lambda - 3\kappa} } } \\
  & \lesssim \| \Pi_{n} e^{(t-r)\mA_n}K_j^x
  \|_{\mC_{1}^{1 - \lambda + 4 \kappa}} \lesssim  (t - r)^{\frac {1 - \kappa}{2}} \| K_j^x
  \|_{\mC_{1}^{ - \lambda + 5 \kappa}} \\
  & \lesssim  (t - r)^{\frac { \gamma} {2}} 2^{-j(\lambda - 5 \kappa)}.
\end{align*}
Overall, we have obtained that
\begin{align*}
\| \Pi_{n} \big\vert \Pi_{n} e^{(t-r)\mA_n}K_j^x \big\vert  \|_{L^{2}}
  & \lesssim n^{\lambda - \kappa}(t - r)^{\frac { \gamma} {2}} 2^{-j(\lambda -
5 \kappa)}.
\end{align*}
In this way, in both dimensions, substituting the estimate into
\eqref{eqn:proof_tightness_fkpp_qv_bound} one obtains
\begin{align*}
\EE \bigg[ \bigg\| \int_s^t \Pi_n e^{(t-r)\mA_n}\ud  M_r^n
  \bigg\|^2_{B^\alpha_{2,2}} \bigg\vert \mF_s\bigg] \lesssim |t-s|^{1-\gamma}.
\end{align*}
For  sufficiently small and  deterministic $T^{*}$, chosen uniform over all
parameters, inequality \eqref{eqn:first_step_tightness_fkpp} is then proven for
all \( s \leqslant t \) such that \( (t -s) \leq T^{*}\). 
Due to the presence of the conditional expectation, one can obtain the estimate
for all \(s \leqslant t \) via a Gronwall-type argument. Indeed, to extend the estimate to 
\(2T^*\), observe there exists a \(C(T^*)\) such that
\begin{align*}
  \sup_{t \in [s,s + 2T^*]} \EE \big[ \| \Pi_n X^n_{t} \|^2_{B^{\alpha}_{2,2}} \big\vert \mF_s \big]
  & \leq C(T^*) \bigg( 1 + \sup_{t \in [s, s+T^*]}\EE \big[ \|\Pi_n X^n_{t}\|_{B^\alpha_{2,2}}^2 \vert
  \mF_{s}\big] \bigg)\\
  & \le C(T^*) \bigg( 1 + C(T^*) \Big( 1 + \EE \big[ \|\Pi_n X^n_{s}\|_{B^\alpha_{2,2}}^2 
  \big] \Big)\bigg) .
\end{align*}
Iterating this argument yields the bound for arbitrary $T$.

\textit{Step 3.}
The next goal is a bound for the expectation of an increment. For this reason fix
$0<\beta<\alpha,$ with $\alpha$ as in Step~1. We shall prove
that there exists a \(\zeta > 0\) such that
\begin{equation}\label{eqn:second_step_tightness_fkpp}
\EE \Big[ \| \overline{X}^n_t - \overline{X}^n_s\|_{B^\beta_{2,2}}^2 \Big]
  \lesssim |t-s|^{4\zeta}.
\end{equation}
Indeed, arguments similar to those in Step~1 show that 
\begin{align*}
\EE \Big[ \| \overline{X}^n_t -  \overline{X}^n_s\|_{B^\beta_{2,2}}^2 \Big] 
& \lesssim \EE \Big[ \| \overline{X}^n_t - e^{(t-s)\mA_n}
\overline{X}^n_s\|_{B^\beta_{2,2}}^2 \Big] + \EE \Big[ \| e^{(t-s)\mA_n}
\overline{X}^n_s - \overline{X}^n_s\|_{B^\beta_{2,2}}^2 \Big] \\
& \lesssim \EE \Big[ \| \overline{X}^n_t - e^{(t-s)\mA_n}
\overline{X}^n_s\|_{B^\beta_{2,2}}^2 \Big] +
|t-s|^{\alpha-\beta} \EE \| \overline{X}^n_s\|_{B^\alpha_{2,2}}^2
\\
& \lesssim
|t-s|^{1-\gamma} (1 + \EE \| \overline{X}^n_s\|^2_{B^\beta_{2,2}})
+ |t-s|^{\alpha-\beta} \EE \| \overline{X}^n_s\|_{B^\alpha_{2,2}}^2,
\end{align*}
where the penultimate step follows from
Lemma~\ref{lem:schauder_short_time_contraction}.
This is enough to establish
\eqref{eqn:second_step_tightness_fkpp}.

\textit{Step 4.} Notice that 
\eqref{eqn:proof_tightness_fkpp_uniform_L^2_bound} and
\eqref{eqn:second_step_tightness_fkpp} together guarantee that
\begin{align*}
  \sup_{n \in \NN} \EE \Big[ \| \overline{X}^{n} \|_{L^{2}([0,T]; B^{\alpha}_{2,2})}^{2} + \|
  \overline{X}^{n} \|_{W^{2, \zeta} ( [0, T] ; B^{\beta}_{2, 2})} \Big] < \infty, 
\end{align*}
with   \(\zeta\) as in \eqref{eqn:second_step_tightness_fkpp}. This implies tightness in \(L^{2}([0,T]; B^{\alpha^{\prime}}_{2,2})\) for any
\(\alpha^{\prime} < \alpha\), which is still sufficient for the result, since
\(\alpha\) varies in an open set.

\end{proof}

Below we recall for convenience Simon's tightness criterion. Here the space $W^{2, \zeta}([0, T]; Y)\subset
L^2([0,T]; Y)$ is defined by the Sobolev-Slobodeckij norm
\begin{align*}
	\| f \|_{W^{2, \zeta}([0, T]; Y)} = \|f\|_{L^2([0,T]; Y)} + \bigg(
	\int_0^T\int_0^T \frac{\|f(t) -f(r)\|_Y^2}{|t-r|^{2 \zeta +1}} \ud t \ud r\bigg)^{1/2}.
\end{align*}
\begin{proposition}[Corollary 5, \cite{Simon1987ComactSetsInLp}]
\label{prop:simon_criterion}
Let $X,Y,Z$ be three Banach spaces such that $X \subset Y \subset Z$ with the
embedding $X \subset Y$ being compact. 
Then also the following embedding is compact, for any \(s>0\):
\[ L^{p}([0,T]; X) \cap W^{s, p}([0,T]; Z) \subseteq L^{p}([0,T]; Y).\] 
\end{proposition}

At this point, the last step is to prove that any limit point satisfies the
required martingale problem (in \(d=1\)) or solves the required PDE (in
\(d=2\)).

\begin{proof}[Proof of Theorem~\ref{thm:convergence_fkpp}]

  As in all previous cases, we fix \(\omega \in \Omega\) and do not state
  explicitly the dependence on it. We treat the drift and the martingale part
  differently.

  \textit{Step 1.} We start with the drift, which is the same in both
  dimensions. Let \(X\) be any limit point of \(X^{n}\) in
  \(C([0,T]; \mM(\TT^{d}))\). The previous proposition guarantees that 
  \(X\) lies almost surely in \(L^{2}([0,T];B^{\alpha}_{2,2})\) for some
  \(\alpha >0\). In addition, through Skorohod representation, we assume
  that \(\Pi_{n} X^{n} \to X\) in \(L^{2}([0,T]; B^{\alpha}_{2,2})\) almost
  surely (in fact this is the reason why we prefer to work
  with \( \Pi_{n}X^{n} \) instead of \( X^{n} \) -- the latter converging only
  as a positive measure).
  In particular, for \(\varphi \in C^{\infty}(\TT^{d})\), define
  \begin{align*}
    N_{t}^{\varphi} = \langle X_{t,0}, \varphi \rangle - \int_{0}^{t} \langle
    X_{s}, \nu_{0} \Delta \varphi \rangle + \langle  \bar{\xi}
    ( X_{s}-X^2_s), \varphi  \rangle\ud s.
  \end{align*}
  Then, regarding the nonlinear term, since both $ X_{s}, \Pi_{n} X^{n}_{s} \in
  [0, 1]$, we can estimate:
  \begin{align*}
    \int_{0}^{t} \int_{\TT^{d}} | X_{s}^{2} - (\Pi_{n} X^{n})^{2}| \ud x
    \ud s \leq  \int_{0}^{t} \int_{\TT^{d}} 2 | X_{s} - \Pi_{n} X^{n}| \ud x
    \ud s \lesssim \| X_{s} - \Pi_{n} X^{n} \|_{L^{2}( [0,T];
    B^{\alpha}_{2,2})},
  \end{align*}
  so that applying Lemma~\ref{lem:convergence_A_ve_to_laplace}, we have almost surely:
  \begin{align*}
    N_{t}^{\varphi} &= \lim_{n \to \infty} \bigg[ \langle \Pi_{n}
      X^{n}_{t, 0}, \varphi \rangle - \int_{0}^{t} \langle \mA_{n}
      \Pi_{n} X^{n}_{s}, \varphi \rangle + \langle \Pi_{n}^{2}\overline{\xi} \big[
      \Pi_{n}^{3} X^{n}_{s} -  (\Pi_{n}^{3} X^{n}_{s})^{2}\big],
     \varphi \rangle \ud s \bigg]\\
    & =: \lim_{n \to \infty} N^{n, \varphi}_{t}.
  \end{align*}
  \textit{Step 2.} Now we prove that \(N_{t}^{\varphi}\) is a centered continuous martingale. In \(d=2\) the quadratic
  variation will be zero and hence \(N^{\varphi} \equiv 0\), proving that the limit is
  deterministic (conditional on the environment). Since \(N^{n,
  \varphi}_{t}\) is a sequence of martingales, by
  Lemma~\ref{lem:martingale_problem_discrete_fkpp}, the fact that also
  \(N^{\varphi}_{t}\) is a martingale follows from the uniform bound
  of Equation~\eqref{eqn:proof_tightness_fkpp_uniform_L^2_bound} (the
  continuity of $N^\varphi$ is as well a consequence of 
Proposition~\ref{prop:tightness_fkpp}). The quadratic
  variation of \(N^{n, \varphi}\) is 
  \begin{align*}
    \langle N^{n , \varphi} \rangle_{t} = n^{-\lambda}
  \int_0^{t} & \langle (1 {+} s_{n}) \Pi_{n}^{3}X^{n}_{r}, (
  \Pi_{n}^{2} \varphi)^2 {-} 2   \Pi_{n}^{2} (\varphi)
  \Pi_{n}(X^{n}_{r} \varphi) \rangle \\
  & + \langle \big( \Pi_{n}(X^{n}_{r} \Pi_{n}\varphi )
  \big)^2, 1 \rangle - \langle s_{n} (\Pi_{n}^{3}X^{n}_{r})^2,
  (\Pi_{n}^{2}\varphi)^2 {-} 2
  \Pi_{n}^{2}(\varphi) \Pi_{n}(X^{n}_{r} \Pi_{n}\varphi) \rangle \ud r,
  \end{align*}
  with \(\lambda = 0\) in \( d =1 \) and \( \lambda = \eta >0\) in \( d =2 \).
In the latter case (\( d=2, \lambda>0 \)) the bounds \( 0 \leqslant
X^{n} \leqslant 1, | s_{n} | \lesssim n^{-2} \) guarantee that
$ \lim_{n \to \infty} \langle N^{n, \varphi} \rangle_{t} = 0.$ 
Instead if \( d=1, \lambda=0 \) we have to take more care. As before, the bound \( | s_{n} | \lesssim
n^{-2} \) guarantees that all terms multiplied by \( s_{n} \) vanish in the
limit, so we are left with considering
\begin{align*}
\lim_{n \to \infty} \int_0^{t} \langle  \Pi_{n}^{3}X^{n}_{r}, (
  \Pi_{n}^{2} \varphi)^2 {-} 2   \Pi_{n}^{2} (\varphi)
  \Pi_{n}(X^{n}_{r} \varphi) \rangle + \langle \big( \Pi_{n}(X^{n}_{r} \Pi_{n}\varphi )
  \big)^2, 1 \rangle  \ud r.
\end{align*}
We can rewrite the quantity in the limit as:
\begin{align*}
\int_0^{t} \langle &  \Pi_{n}^{3}X^{n}_{r}, (
  \Pi_{n}^{2} \varphi)^2 {-} 2   (\Pi_{n}^{2} \varphi)
  [(\Pi_{n}X^{n}_{r}) \varphi ] \rangle + \langle \big( (\Pi_{n}X^{n}_{r}) \Pi_{n}\varphi 
  \big)^2, 1 \rangle  \ud r \\
& + \int_{0}^{t}-2\langle \Pi_{n}^{3}X^{n}_{r}, (\Pi_{n}^{2} \varphi)
  [D^{\Pi, n}(X^{n}_{r}, \varphi)] \rangle + \langle ( D^{\Pi,
n}(X^{n}_{r}, \Pi_{n} \varphi))^2, 1 \rangle \\
& \qquad + 2 \langle D^{\Pi, n }(X^{n}_{r}, \Pi_{n} \varphi),
(\Pi_{n}X^{n}_{r}) \Pi_{n}\varphi \rangle \ud r,
\end{align*}
where we have defined the commutator (cf.
Lemma~\ref{lem:commutator_with_averaging} for a similar construction)
\[ D^{\Pi, n}( \psi, \psi^{\prime}) = \Pi_{n}(\psi \cdot \psi^{\prime} ) - (\Pi_{n} \psi )
\cdot \psi^{\prime} . \] 
Now we observe that for \( \delta \in [0, 1]:\)
\begin{align*}
 \sup_{x \in \TT^{d}} | D^{\Pi, n}(\psi, \psi^{\prime} ) | (x) & = \sup_{x \in
\TT^{d}} \Big\vert \mint{-}_{B_{n}(x)} \psi(y) (\psi^{\prime}  (y) -
\psi^{\prime} (x)) \ud y \Big\vert  \lesssim n^{-\delta} \| \psi \|_{L^{\infty}} \| \psi^{\prime} \|_{\mC^{\delta}}.
\end{align*}
We can apply this bound to our quadratic variation, observing that \(\varphi
\in C^{\infty}(\TT^{d})\) and \( \| X^{n} \|_{\infty} \leqslant 1 \), so that:
\begin{align*}
\lim_{n \to \infty} & \int_{0}^{t} -2\langle \Pi_{n}^{3}X^{n}_{r}, (\Pi_{n}^{2} \varphi)
  [D^{\Pi, n}(X^{n}_{r}, \varphi)] \rangle + \langle ( D^{\Pi,
n}(X^{n}_{r}, \Pi_{n} \varphi))^2, 1 \rangle \\
& \qquad \qquad \qquad \qquad \qquad \qquad + 2 \langle D^{\Pi, n }(X^{n}_{r}, \Pi_{n} \varphi),
(\Pi_{n}X^{n}_{r}) \Pi_{n}\varphi \rangle \ud r = 0.
\end{align*}
Finally we are left with computing the limit
\begin{align*}
 \lim_{n \to \infty} \langle N^{n , \varphi} \rangle_{t}  & = \lim_{n \to \infty} \int_0^{t} \langle  \Pi_{n}^{3}X^{n}_{r}, (
  \Pi_{n}^{2} \varphi)^2 {-} 2   (\Pi_{n}^{2} \varphi)
  [(\Pi_{n}X^{n}_{r}) \varphi ] \rangle + \langle \big( (\Pi_{n}X^{n}_{r}) \Pi_{n}\varphi 
  \big)^2, 1 \rangle  \ud r \\
& =\int_{0}^{t} \langle X_{r}, \varphi^{2} - 2 X_{r} \varphi^{2} \rangle +
    \langle X^{2}_{r}, \varphi^{2} \rangle \ud r = \int_{0}^{t} \langle X_{r}(1 - X_{r}),
    \varphi^{2} \rangle \ud r.
\end{align*}
Here the second equality follows by calculations analogous to those in Step \(
1 \), since now the quadratic nonlinearity is a function of \( \Pi_{n}
X^{n} \) and the latter is converging in \( L^{2}([0, T];
B^{\alpha}_{2,2,}) \).

Finally, since the martingale \( (N^{n , \varphi}_{t})^{2} - \langle N^{n, \varphi}
\rangle_{t}\) is bounded (using that \( 0 \leqslant X^{n} \leqslant 1 \)), also the limiting process \( (N^{\varphi}_{t})^{2} - \lim_{n \to
\infty} \langle N^{n, \varphi} \rangle_{t} \) is a martingale, implying that \(\langle N^{\varphi}
    \rangle_{t} = \lim_{n \to \infty} \langle N^{n, \varphi} \rangle_{t}\). 
  Hence the quadratic variation is of the required form for Theorem~\ref{thm:convergence_fkpp}.

  So far we have proven that any limit point solves the required equation. To
  deduce the convergence, we have to prove that such solutions are unique. In $d=2$, that for every $\omega \in \Omega$ there exists a unique solution to the equation
  \[\partial_t X = \nu_0 \Delta X +  \overline{\xi}(\omega) X(1-X),\qquad X(0) = X_0\]
follows from classical solution theory. Instead in $d=2$ uniqueness in law can be established via a Girsanov transform, as we show in Lemma~\ref{lem:uniqueness-martingale-problem-fkpp} below.

\end{proof}

\begin{lemma}\label{lem:uniqueness-martingale-problem-fkpp}
In $d=1$ and under Assumption~\ref{assu:smoothened_noise_fkpp}, solutions to the stochastic Fisher-KPP equation as in Definition~\ref{def:Fisher-KPP_in_rough_potential} are unique in distribution.
\end{lemma}

\begin{proof}
As usual, the argument works for fixed $\omega \in \Omega$, so we omit writing the dependence on it.
First, the same calculations as in Proposition~\ref{prop:tightness_fkpp} prove
that any solution $X$ to the martingale problem of the stochastic Fisher-KPP
equation lives in $L^2([0, T]; B^{\alpha}_{2,2})$, for some $\alpha>0$ and
arbitrary $T>0$. Then, following the same arguments as in the proof of
\cite[Theorem 2.18]{PerkowskiRosati2019RSBM}, we see that (up to enlarging the probability space) $X$ is a solution to the SPDE:
\begin{align*}
	\partial_t X = \nu_0 \Delta X + \overline{\xi} X(1-X)+ \sqrt{X(1-X)} \tilde{\xi}, \qquad X(0) = X_0,
\end{align*}
where $\tilde{\xi}$ is a space time white noise. Here we mean solutions in the sense that for any $\varphi \in C^{\infty}(\TT^d)$ and $t \in [0, T]$:
\begin{align*}
 \langle X_t , \varphi \rangle = \langle X_0 , \varphi \rangle = & \int_0^t \langle X_s, \nu_0 \Delta \varphi \rangle+ \langle \overline{\xi} X_s(1-X_s), \varphi \rangle \ud s \\
 & \ \  + \int_0^t\int_{\TT^d} \sqrt{X_s(x)(1-X_s(x))} \varphi(x) \ud \tilde{\xi}(s, x),
\end{align*}
where the latter is understood as an integral against a martingale measure, in the sense of Walsh \cite{Walsh1986}. Now we can use a Girsanov transform \cite[Theorem 5.1]{Dawson1978GeostochasticCalculus} (see also \cite[Theorem IV.1.6]{Perkins2002SuperprocessesStFlour} and  \cite[Section 2.2]{MuellerMytnikRyzhik2019SpeedRandomFront} for more recent accounts). Let us denote with $\PP$ the law of $X$ on $L^2([0, T]; B^{\alpha}_{2,2})$ and define the measure $\QQ$ by:
\[\frac{\ud \QQ}{\ud \PP} = \exp \bigg( -\int_0^T \int_{\TT^d} \frac{\overline{\xi}(x) X_s(x)(1-X_s(x))}{\sqrt{X_s(x)(1-X_s(x))}} \ud \tilde{\xi}(s,x)- \frac{1}{2} \int_0^T \int_{\TT^d} \frac{\big(\overline{\xi}(x) X_s(x)(1-X_s(x))\big)^2}{X_s(x)(1-X_s(x))} \ud s \ud x \bigg).\]
Clearly, this transformation defines a change of measure, since
\[\int_0^T \int_{\TT^d} \frac{\big(\overline{\xi}(x) X_s(x)(1-X_s(x))\big)^2}{X_s(x)(1-X_s(x))} \ud s \ud x \leq T\| \overline{\xi}\|_{\infty}^2.\]
Under this change of measure, for every $\varphi \in C^{\infty}(\TT^d)$, the process:
\[\langle X_t ,\varphi \rangle -\langle X_0 ,\varphi \rangle - \int_0^t \langle X_s ,\nu_0 \Delta \varphi \rangle \ud s =: L^{\varphi}_t\]
is a continuous $\QQ-$martingale with quadratic variation:
$\langle L^{\varphi} \rangle_t = \int_0^t \langle X_s(1-X_s), \varphi^2 \rangle \ud s.$
This means that under $\QQ$, the process $X_t$ is the unique (in law) solution to the SPDE:
\begin{align*}
	\partial_t X = \nu_0 \Delta X + \sqrt{X(1-X)} \tilde{\xi}, \qquad X(0) =X_0.
\end{align*}
The uniqueness in law of solutions to the latter equation follows by duality,
see e.g. \cite{Shiga1988SteppingStone}.
\end{proof}

\section{Schauder estimates}
\label{sec:schauder_estimates}

This section is devoted to the proof of
Theorem~\ref{thm:regularization-estimates-main-results} and other similar
results. Since the central object in this section, the semidiscrete Laplace
operator \(\mA_{n}\), is defined through convolutions with characteristic
functions, the following result collects some information that will be useful in
the upcoming discussion.

\begin{lemma}\label{lem:taylor_expansion_and_decay_ft_characteristic_function}
  Let $(D \varphi)_{i} = \frac{d \varphi}{d x_{i}}$ and \((D^{2}
  \varphi)_{i,j} = \frac{d^{2} \varphi}{d x_{i}d x_{j}}\) indicate the gradient
  and the Hessian matrix of a smooth function \(\varphi \colon \RR^{d} \to \RR\)
  respectively. Recall that \(\hat{\chi}_{n}(k) = \hat{\chi}(n^{-1} k) = \mF_{\RR^{d}}
  (n^{d} 1_{\{B_{n}(0)\}})(k)\). Then:
  \[ 
    D \hat{\chi} (0) = 0, \qquad D^2 \hat{\chi} (0) = {-} \frac{(2 \pi)^2}{4}
    \nu_0 \mathrm{Id},
  \]
  with $ \nu_{0} $ as in \eqref{eqn:for-nu-0}. Next recall that
$\vt_{n}(k) = n^{2} \big( \hat{\chi}^{4}_{n}(k) -1 \big).$
Then for any choice of constants $c<1<C,$ there exists a $\kappa(c,C)>0$ such that
  \begin{align*}
  c \leq \frac{\vt_n( k )}{-(2\pi)^{2}\nu_0 |k|^2} \leq C, \qquad \forall k \colon \ |k| n^{-1} \leq \kappa(c, C).
  \end{align*}
  Finally, the decay of \( \hat{\chi}\) can be controlled as follows for any
  \(n \in \NN \cup \{ 0 \}\) and \(i_{1}, \dots, i_{n} \in \{1, \dots, d\}\):
  \begin{align*}
    \Big\vert \frac{d^{n} \hat{\chi} (k)}{d x_{i_{1}} \cdots dx_{i_{n}}}
    \Big\vert \lesssim_{n} (1 {+}
    |k|)^{{-} \frac{d {+} 1}{2}}.
  \end{align*}
\end{lemma}
The proof of this result is deferred to
Appendix~\ref{app:prf-of-lemma}. Instead, we pass to the
central result of this section, from which all other will follow.
Recall that \(\mA_{n}\) is a Fourier multiplier, therefore also the
exponential \(e^{t \mA_{n}}\) and the resolvent \( (- \mA_{n} +
\lambda)^{-1}\) (for \(\lambda > 1\)) are naturally defined as Fourier
multipliers.
As explained already in other points, the action of \(\mA_{n}\) is different
on large and small Fourier modes. 
\begin{proposition}\label{prop:crucial_for_schauder}

  For some, and hence for all, \(\kappa_{0}>0\) the following holds. 
  For any \(p \in [1, \infty], \alpha \in \RR\) and \(j \geq {-} 1\) there
  exists a \(c>0\) such that uniformly over \(n \in \NN, t \geq 0, j \geq
  -1\) and \(\varphi \in \mC^{\alpha}_{p}\) one can bound:
  \begin{equation}\label{eqn:crucial-for-schauder-laplacian}
  \begin{aligned}
    \| \Delta_{j} \mA_{n} \varphi \|_{L^{p}(\TT^{d})} & \lesssim 2^{{-}
    (\alpha {-} 2)j} \| \varphi \|_{\mC^{\alpha}_{p}}, \ \text{
    if } \ 2^{j} \leq \kappa_{0} n, \\
    \| \Delta_{j} \mA_{n} \varphi \|_{L^{p}(\TT^{d})} & \lesssim
    n^{2} 2^{{-} \alpha j} \| \varphi \|_{\mC^{\alpha}_{p}}, \  \text{ if
    } 2^{j} > \kappa_{0} n.
  \end{aligned}
\end{equation}
  And similarly for the exponential: 
    \begin{equation}\label{eqn:crucial-for-schauder-exponential}
	\begin{aligned}
      \| \Delta_{j} e^{t \mA_{n}} \varphi \|_{L^p(\TT^d)} & \lesssim e^{{-} c
      t 2^{2j}} 2^{{-} \alpha j} \| \varphi \|_{\mC^{\alpha}_p}, \ \text{ for }
      \ 2^{j}  \leq \kappa_{0} n, \\ 
      \| \Delta_{j} e^{t \mA_{n}} \varphi \|_{L^p(\TT^d)} & \lesssim e^{{-} c
      t n^{2}} 2^{{-} \alpha j} \| \varphi \|_{\mC^{\alpha}_{p}}, \ \text{
      for } \ 2^j  > \kappa_{0} n,  
    \end{aligned}
    \end{equation} 
    and for the resolvent (uniformly over \(\lambda>1)\):
      \begin{equation}\label{eqn:curcial-for-schauder-resolvent}
	\begin{aligned}
	  \| \Delta_{j} (-\mA_{n}+ \lambda)^{-1} \varphi \|_{L^p(\TT^d)} & \lesssim 
	  \frac{1}{2^{2j} + \lambda} 2^{- \alpha j} \| \varphi \|_{\mC^{\alpha}_p}, \ \text{ for }
      \ 2^{j} \leq \kappa_{0} n, \\ 
      \| \Delta_{j} (- \mA_{n} + \lambda)^{-1} \varphi \|_{L^p(\TT^d)} & \lesssim 
      \frac{1}{n^{2} + \lambda} 2^{- \alpha j} \| \varphi \|_{\mC^{\alpha}_{p}}, \ \text{
      for } \ 2^j  > \kappa_{0} n,  
	\end{aligned}
      \end{equation}

\end{proposition}

\begin{proof}
If the estimates hold for a certain $\kappa_0>0$, it is evident that they hold for all $\kappa_0>0$ (up to changing proportionality constants). In fact, for $2^j \simeq n$ the first and second estimate in every pair are equivalent.

  Since all of the estimates follow the same pattern and the first one is
  particularly simple, we will mainly discuss the proof of 
  the inequalities in~\eqref{eqn:crucial-for-schauder-exponential}, pointing out how to
  adapt the calculations to the other cases. We also restrict to the case
  $ j \geq 0,$ 
  since the case \(j = -1\) is immediate.
  We begin by restating the inequalities for distributions on \(\RR^{d}\). This
  is useful because on the entire space we can use scaling arguments.
  Then we examine the behaviour on large and small scales separately.
  The precise separation of modes is chosen based on
  Lemma~\ref{lem:taylor_expansion_and_decay_ft_characteristic_function}. 
   
  \textit{Step 1.} To restate the problem on \(\RR^{d}\) we extend
  distributions on the torus periodically.
  Let  \(\pi \colon \mS^{\prime}(\TT^{d}) \to \mS^{\prime}(\RR^{d})\) denote
  such periodic extension operator mapping distributions on $\TT^{d}$ to
  distributions on the full
  space. Its adjoint is the operator $\pi^{*}
  \colon \mS(\RR^{d}) \to \mS(\TT^{d})$, given by 
  \[  \pi^{*} \varphi (\cdot) = \sum_{k \in \ZZ^{d}} \varphi(\cdot + k). \] 
  We observe that \(\pi (\mA_{n} \varphi)  = \mA_{n} \pi(\varphi)\), where
  with a slight abuse of notation we have extended \(\mA_{n}\) to act on
  distributions on the whole space (simply through
  Equation~\eqref{eqn:A_ve-definition} -- and note that it is
still a Fourier multiplier, since for \(\varphi \colon \RR^{d} \to
\RR\), \(\mA_{n} \varphi = \mF_{\RR^{d}}^{-1} \vt_{n} \mF_{\RR^{d}} \varphi\)).
  Similarly, by the Poisson summation formula of
  Lemma~\ref{lem:poisson_summation_formula} we have
  \(\pi(\Delta_{j} \varphi) = \Delta_{j} \pi( \varphi)\).
  As a consequence of  this last observation, and since $\pi(\Delta_j \varphi)$ is periodic, for any \(a > \frac dp\)
  (or \(a \geq 0\) if \(p = \infty\)): 
  \[ \| \Delta_{j} \pi( \varphi) \|_{L^{p}( \RR^{d}; p(a))} \simeq_{a, p} \| \Delta_{j} \varphi \|_{L^{p}(\TT^{d})},\]
  where \(\| f \|_{L^{p}(\RR^{d}; p(a))} = \| f( \cdot) /(1 {+} |
  \cdot|^{2})^{\frac a2} \|_{L^{p}(\RR^{d})}\). Therefore in order to show
  \eqref{eqn:crucial-for-schauder-exponential} it is sufficient to show that
  for all \(\varphi \in \mS^{\prime}(\RR^{d})\) and setting \(a = d+1\):
  \begin{align*} 
    \| \Delta_{j} e^{t \mA_{n}} \varphi \|_{L^p(\RR^d; p(d+1))} & \lesssim e^{{-} c
    t 2^{2j}} \| \Delta_{j} \varphi \|_{L^{p}(\RR^{d}; p(d+1))}, \ \text{ for } \ 2^{j} \leq \kappa_{0}  n,\\
      \| \Delta_{j} e^{t \mA_{n}} \varphi \|_{L^p(\RR^d; p(d+1))} & \lesssim e^{{-} c
    t n^{2}} \| \Delta_{j} \varphi
    \|_{L^{p}(\RR^{d}; p(d+1))}, \ \text{ for } \ 2^j > \kappa_{0} n.
  \end{align*} 
  The same holds for \eqref{eqn:crucial-for-schauder-laplacian} and
  \eqref{eqn:curcial-for-schauder-resolvent}, with natural changes.
  Hence, from now on let us consider all functions and operators defined on
  \(\RR^{d}\).
  Let \(\psi\) be a smooth radial function with compact support in an annulus
  (i.e. \(\psi(k)=0\) if \(|k| \leq c_1\) or \(|k| \geq c_{2}\) for some
  \(0< c_{1} < c_{2}\)) such that \(\rho
  \psi = \rho\) (here \( \rho \) is associated to the dyadic partition of the
  unity through which we define Besov spaces: see
the notations section). By Young's inequality for
  convolutions and by estimating uniformly over \(x,y \in \RR^{d}\)
  \[ (1 +|x|^{2})^{-\frac{(d+1)}{2}} \lesssim  (1 +|y|^{2})^{-\frac{(d+1)}{2}}  (1
  +|x-y|^{2})^{\frac{d+1}{2}},\]
  one obtains:
  \begin{equation*}
    \| \Delta_j e^{t \mA_{n}} \varphi \|_{L^{p}(\RR^{d}; p(d+1))}  \lesssim \|
    \mF^{{-} 1}_{\RR^{d}} ( e^{t \vt_{n}( \cdot) } \psi(2^{{-} j}
    \cdot) ) \|_{L^1(\RR^d; p(-d-1))} \| \Delta_{j} \varphi
    \|_{L^{p}(\RR^{d}; p(d+1))}. 
  \end{equation*}
  In this way, through a change of variables, we reduced the problem to a bound for
  \begin{equation}\label{eqn:proof-crucial-for-schauder-reduced-problem}
    \int_{\RR^{d}} (1 + 2^{-2j} |x|^2)^{\frac{d+1}{2} }
    \Big\vert \mF_{\RR^{d}}^{-1} \Big[ e^{t \vt_{n}(2^{j} \cdot)} \psi(\cdot)\Big](x)\Big\vert \ud x
  \end{equation}
  (and similarly for
    \eqref{eqn:crucial-for-schauder-laplacian} and
    \eqref{eqn:curcial-for-schauder-resolvent}, with \(e^{t \vt_{n}}\)
    replaced by \(\vt_{n}\) and \( (- \vt_{n} + \lambda)^{-1}\) respectively).
  Before we move on, we finally observe that by
  Lemma~\ref{lem:taylor_expansion_and_decay_ft_characteristic_function}, there
  exists a \(\kappa_{0} > 0\) such that for \(2^j n^{-1} \leq \kappa_{0}\): 
  \[ 
    \frac{1}{2}  \leq \frac{\vartheta_{n}(2^j k )}{{-} (2 \pi)^2
    \nu_0 2^{2j}|k|^2} \leq \frac{3}{2}, \qquad \forall k \in \supp (\psi).
  \]

  \textit{Step 2.}
  We now estimate \eqref{eqn:proof-crucial-for-schauder-reduced-problem} on
  large scales, i.e. \( 2^{j} n^{-1} \leq \kappa_{0}\). In this case the term can
  be bounded by: 
  \begin{align*}
    \Big\| \mF^{{-} 1}_{\RR^{d}}[ e^{t \vt_{n}( 2^j \cdot)} \psi( \cdot )
    ]+ \sum_{i = 1}^{d} & \big\vert\mF^{{-} 1}_{\RR^{d}}[
      \partial_{k_{i}}^{2(d+1)} e^{t \vt_{n}( 2^j \cdot)} \psi( \cdot )
    ]\big\vert \Big\|_{L^{\infty}(\RR^{d})} \\
    & \lesssim \sup_{k \in \supp
    (\psi)} \Big[ \big\vert  e^{t \vt_{n}( 2^j k)} \psi(k)
      \big\vert + \sum_{i =1}^{d} \big\vert \partial_{k_{i}}^{2(d+1)}
    e^{t \vt_{n}(2^{j} k)} \psi(k) \big\vert\Big].
  \end{align*} 
    To bound the term involving  derivatives we observe that:
  \[ D [t \vt_{n}(2^j \cdot)](k) = f(2^j n^{-1} k) t 2^{2j} | k |, \qquad
  f(k) = 4 \hat{\chi}^{3}(k) \frac{D \hat{\chi}(k)}{|k|}.\] 
  where 
  \(f\) is smooth on \(\RR^d\), again by
  Lemma~\ref{lem:taylor_expansion_and_decay_ft_characteristic_function}.
  In particular, since \(2^{j} \lesssim n\), taking higher order
  derivatives one has for any \(\l \in \NN\): \(\big\vert \partial_{k_{i}}^{\l} [t
  \vt_{n}(2^{j} \cdot )]\big\vert(k) \lesssim t 2^{2j}\) for \(k \in \supp
  (\psi)\). Now recall Fa\'a di Bruno's formula for $\l \in \NN$:
  \begin{align*}
  \partial_{x}^{\l} f(g(x)) = \sum_{\{m\}}C(\{m\}, \l) (\partial_{x}^{m_{1} + \dots +
    m_{\l}} f) (g(x)) \prod_{r = 1}^{\l} \Big( \partial_{x}^{r} g (x)
    \Big)^{m_{j}},
  \end{align*}
  where the sum runs over all \( \{m\}:= (m_{1}, \ldots, m_{\l})\) such that
  \(m_{1}+2 m_{2}+ \dots + \l m_{\l}=\l\).
  Applying this formula and by our choice of \(\kappa_{0}\), there exists a constant \(c >0\) such that:
  \[ 
    \sup_{k \in \supp (\psi)} \Big[ | e^{ t \vt_{n}(2^j k) \psi(k) }
      | + \sum_{i = 1}^{d} \big\vert \partial_{k_{i}}^{2(d+1)} e^{t
    \vt_{n}(2^{j} k)} \psi(k) \big\vert \Big] \lesssim e^{{-} \frac 1 2  (2 \pi)^2
    \nu_0 t  2^{2j}} (1 {+} t 2^{2j})^{2(d+1)} \lesssim e^{ {-} c (t 2^{2j})}. 
  \]
  This concludes the proof of the large-scale bound in
  \eqref{eqn:crucial-for-schauder-exponential}. For the resolvent equation
  one similarly has to bound:
  \[
    \sup_{k \in \supp (\psi)} \Big[ \Big\vert \frac{\psi(k)}{-
      \vt_{n}(2^{j} k) + \lambda}  \Big\vert + \sum_{i =1}^{d} \Big\vert \partial_{k_{i}}^{2(d+1)}
    \frac{\psi(k)}{- \vt_{n}(k) + \lambda}  \Big\vert\Big]. 
   \] 
   Here as before, for the derivative term one has, through the choice of
   \(\kappa_{0}\):
   \begin{align*}
     \bigg\vert &\partial_{k_{i}}^{\l} \frac{1}{- \vt_{n}(k) + \lambda}
     \bigg\vert  \lesssim \sum_{ \{m\}} \bigg\vert \frac{1}{- \vt_{n}(k) + \lambda}
     \bigg\vert^{1+m_{1}+ \dots + m_{\l}} \prod_{r = 1}^{\l} \Big(
     2^{2j}\Big)^{m_{r}} \\
     & \lesssim \sum_{ \{m\}} \bigg\vert \frac{1}{\frac{3}{2} (2
     \pi)^{2} \nu_{0} 2^{2j} + \lambda }  \bigg\vert^{1 + m_{1} + \dots +
     m_{\l}} (2^{2 j})^{m_{1}+ \dots + m_{\l}} \lesssim \frac{1}{\frac{1}{2} (2 \pi)^{2} \nu_{0} 2^{2j}+ \lambda}
     \lesssim \frac{1}{2^{2j} + \lambda}, 
   \end{align*}
   as requested for \eqref{eqn:curcial-for-schauder-resolvent}.
   The estimate \eqref{eqn:crucial-for-schauder-laplacian} follows similarly.

  \textit{Step 3.} 
  We pass to the small-scale estimates, namely for \(j\) such that
  \(2^{j} n^{-1} > \kappa_{0}\). Here we will need
  tighter control on the decay of \( \hat{\chi}(k)\): 
  since \(\chi\) is not smooth, the decay at infinity is not faster than any
  polynomial and is quantified in
  Lemma~\ref{lem:taylor_expansion_and_decay_ft_characteristic_function}. We now
  estimate \eqref{eqn:proof-crucial-for-schauder-reduced-problem} by
  \begin{align*}
    \bigg( &\int_{\RR^{d}}  \frac{1}{(1 + |x|)^{d+1}} \ud x \bigg) \sup_{x \in \RR^{d}}
    \bigg[ (1 + |x|^{d+1} + 2^{-j(d+1)} |x|^{2(d+1)}) \Big\vert \mF^{-1}_{\RR^{d}} \Big[
    e^{t \vt_{n}( 2^{j} \cdot)} \psi (\cdot) \Big]\Big\vert(x) \bigg] \\
    & \lesssim \| e^{t \vt_{n}(2^{j} \cdot)} \psi(\cdot) \|_{L^{\infty}} +
    \| (1- \Delta)^{\frac {d+1} {2}} e^{t \vt_{n}(2^{j} \cdot)} \psi( \cdot)
    \|_{L^p(\RR^{d})} + \sum_{i = 1}^{d} 2^{-j(d+1)}  \| \partial_{k_{i}}^{2(d+1)}
    e^{t \vt_{n}(2^{j} \cdot)} \psi(\cdot) \|_{L^{\infty}},
  \end{align*}
  for any \(p \in (1, \infty)\). As for the first term, since \(
  |\hat{\chi}(k)| < 1\) for \(k \neq 0\) and it
  decays to zero at infinity, up to reducing the value of \(c>0\) we can assume
  that:
  \[ \vt_{n}(2^j k) \leq {-} c n^{2}. \] 
  This is sufficient to show $\| e^{t \vt_{n}(2^j \cdot)} \psi(\cdot)  \|_{L^{\infty}} \lesssim e^{{-} ct n^{2}},$
  which is a bound of the required order. 
  
	   Now bounding these derivatives is similar to
  bounding the last term:
  \begin{align*}
    \sum_{i = 1}^{d} 2^{-j(d+1)} \| \partial_{k_{i}}^{2(d+1)}  e^{t
    \vt_{n}(2^{j} \cdot)} \psi(\cdot)\|_{L^{\infty}},
  \end{align*}
  so we concentrate on the latter, which has the added difficulty of containing
  derivatives of higher order, counterbalanced by the factor \(2^{-j(d+1)}\).
  Here observe that for \(1 \leq \l \leq 2(d+1)\):
  \[ \partial_{k_i}^{\l}  e^{t \vt_{n}(2^j k)} =
    \partial_{k_i}^{\l-1} \big[ e^{t \vt_{n}(2^j k)} 4 \hat{\chi}^{3}(2^j n^{-1} k)
  [ \partial_{k_i} \hat{\chi}] (2^j n^{-1} k) \big] \cdot (2^j n^{-1}) \cdot (t
n^{2}).\] 
Iterating the above procedure, we apply Fa\'a Di Bruno's formula again to obtain
\[ \big\vert 2^{{-} j(d {+} 1)}\partial_{k_i}^{\l}   e^{t \vt_{n}(2^j
  k)} \big\vert \lesssim 2^{{-} j(d {+} 1)} e^{t \vt_{n}(2^j k)}
  \sum_{ \{m\}} \prod_{r =1}^{\l} \big( \partial_{k_i}^{r {-} 1} \big[ 4
\hat{\chi}^{3}( \cdot ) [ \partial_{k_i} \hat{\chi}( \cdot)] \big]\big\vert_{2^{j} n^{-1} k} \cdot
(2^{j} n^{-1})^{r} \big)^{m_{r}} \cdot (t n^{2})^{m_r}.\] 
 In view of
 Lemma~\ref{lem:taylor_expansion_and_decay_ft_characteristic_function}, for any $r \in \NN$: 
\[ \sup_{ k \in \supp (\psi)} |\partial_{k_i}^{r {-} 1} \big[ 4
\hat{\chi}^{3}(
    \cdot ) [ \partial_{k_i} \hat{\chi}( \cdot)] \big]\big\vert_{2^{j} n^{-1} k}|
  \lesssim \frac{1}{1 {+} |2^j n^{-1} |^{2(d {+} 1)}}.  \] 
  Hence, as before up to further reducing the value of \(c>0\): 
  \begin{align*}  
    \| \partial_{k_i}^{\l} e^{t \vt_{n}(2^j \cdot)} \|_{L^\infty} & \lesssim e^{{-} c t n^{2}} 2^{{-} j(d {+}
    1)}(2^{j} n^{-1})^{\l} \sum_{\{m\}} \prod_{r = 1}^{\l} (1  +|2^{j}
n^{-1}|)^{{-} 2m_{r}(d {+} 1)} \\
    & \lesssim e^{{-} c t n^{2}} 2^{- j (d+1)} (2^{j} n^{-1})^{\l - 2(d+1)}
    \lesssim e^{-c t n^{2}},
  \end{align*}
since at least one of the elements of the sequence \(m_{r}\) is strictly
positive and since \(\l \leq 2(d+1)\). This concludes the proof of
\eqref{eqn:crucial-for-schauder-exponential}. Regarding the resolvent, one can
follow mutatis mutandis the previous discussion until one has, as before, to bound:
\begin{align*}
  \sum_{i = 1}^{d} 2^{- j(d+1)}\Big\| \partial_{k_{i}}^{2(d+1)} \frac{\psi(\cdot)}{-
  \vt_{n}(2^{j} \cdot) + \lambda}  \Big\|_{\infty} \lesssim \sum_{i =
  1}^{d} \sum_{\l = 0}^{2(d+1)} 2^{- j (d+1)}\Big\| \partial_{k_{i}}^{\l} \frac{1}{-
  \vt_{n}(2^{j} \cdot) - \lambda} \Big\|_{L^{\infty}}.
\end{align*}
Then again, with Fa\'a di Bruno's formula:
\begin{align*}
  \Big\vert \partial_{k_{i}}^{\l} \frac{1}{- \vt_{n}(2^{j}k) + \lambda}
  \Big\vert & \lesssim \sum_{\{m\}} \Big\vert \frac{1}{- \vt_{n}(2^{j}k) +
  \lambda} \Big\vert^{1 + m_{1} + \dots + m_{\l}} \prod_{r = 1}^{\l} \big\vert
  \partial_{k_{i}}^{r -1} ( \hat{\chi}^{3} (\cdot) \partial_{k_{i}}
  \hat{\chi}( \cdot) ) \vert_{2^{j} n^{-1} k} \big\vert^{m_r} \cdot (2^{j}
  n^{-1})^{r m_{r}} \\
  & \lesssim \frac{1}{ n^{2} + \lambda} \sum_{\{m\}} \Big\vert \frac{1}{ n^{2}+
  \lambda} \Big\vert^{ m_{1} + \dots + m_{\l}} \prod_{r = 1}^{\l}
  \Big( \frac{1}{1 + |2^{j} n^{-1}|} \Big)^{2 m_{r}(d+1)}  (2^{j} n^{-1})^{r m_{r}}
  \\
  & \lesssim \frac{1}{ n^{2} + \lambda} 2^{j (d+1)}.
\end{align*}
Plugging this into the previous formula provides us the correct bound.
Similarly one can also treat the small-scale estimate for
\eqref{eqn:crucial-for-schauder-laplacian}.
\end{proof}

The previous proposition motivates the introduction of cut-off operators as
follows.

\begin{definition}\label{def:cut_off_operators}
  Let \(\daleth \colon \RR^{d} \to \RR\) be a smooth radial function with compact
  support. Let us define the annulus \(A_{r}^{R} = \{ x \in \RR^{d} \colon r
  \leq |x| \leq R\}\) for \(0<r<R\) and additionally assume that
  $ \daleth(x) = 1, \ \forall x \in
  A_{0}^{r},$ and $\daleth(x) = 0, \ \forall x \in A_{R}^{\infty},$
  for some \( 0 < r < R < \infty \). Then define
  \[ \mP_{n} = \daleth( n^{-1} D), \qquad \mQ_{n} = (1 - \daleth)(n^{-1} D). \] 
  We say that \(\mP_{n}\) is a projection on \textbf{large scales}, since
  those Fourier modes describe a function macroscopically, whereas
  \(\mQ_{n}\) is a projection on \textbf{small scales}. 
\end{definition}

The next lemma states that the cut-off operators are bounded.

\begin{lemma}\label{lem:boundedness_cut_off_operators}
  Consider \(\alpha \in \RR\) and \(p \in [1, \infty]\). For \(\daleth\) as in Definition~\ref{def:cut_off_operators} one can bound\
  uniformly over \(n \in \NN\):
  \[ 
    \| \mP_{n} \varphi \|_{\mC^{\alpha}_{p}} \lesssim \| \varphi
    \|_{\mC^{\alpha}_{p}}, \qquad \| \mQ_{n} \varphi
    \|_{\mC^{\alpha}_{p}} \lesssim \| \varphi \|_{\mC^{\alpha}_{p}}.  
  \] 
\end{lemma}

\begin{proof}
Define the inverse Fourier transform \( \widehat{\daleth}(x) = \mF^{-1}_{\RR^{d}} \daleth
  (x)\).
  By an application of the Poisson summation formula
  (Lemma~\ref{lem:poisson_summation_formula}) and a scaling argument
  \begin{align*}
    \| \daleth( n^{-1} D) \varphi\|_{\mC^{\alpha}_{p}} & = \sup_{j \geq -1}
    2^{j \alpha} \| (\mF^{-1}_{\TT^{d}} [\daleth (n^{-1} \cdot)]) *
    \Delta_{j} \varphi \|_{L^{p}} \lesssim \| \mF^{-1}_{\TT^{d}} [\daleth
    ( n^{-1} \cdot) ] \|_{L^{1}(\TT^{d})} \| \varphi \|_{\mC^{\alpha}_{p}} \\
    & \lesssim \| n^{d} \widehat{\daleth} ( n \cdot) \|_{L^{1}(\RR^{d})} \| \varphi
    \|_{\mC^{\alpha}_{p}} \lesssim \| \varphi \|_{\mC^{\alpha}_{p}}.
  \end{align*}
  The same argument shows that \((1 - \daleth(a \cdot))\) is bounded.
\end{proof}

\subsection{Elliptic regularity}\label{subsec:elliptic_regularity}

In this subsection we prove
Theorem~\ref{thm:regularization-estimates-main-results}. This theorem is a
direct consequence of the lemma and the proposition that follow.

\begin{lemma}\label{lem:convergence_A_ve_to_laplace}
  Fix any \(\alpha \in \RR, \zeta>0, p \in [1, \infty]\). Uniformly over \(\varphi \in
  \mC^{\alpha}_{p}\) and \(n \in \NN\):
  \[ \| \mA_{n} \mP_{n} \varphi \|_{\mC^{\alpha -2}_{p}} \lesssim \|
  \varphi \|_{\mC^{\alpha}_{p}}.\] 
  Moreover, as $n \to \infty$ and with $ \nu_{0} $ as in \eqref{eqn:for-nu-0}
  \[ 
    \mA_{n} \varphi \to  \nu_{0} \Delta \varphi \ \
    \text{in} \ \ \mC^{\alpha -2- \zeta}_{p}.
  \] 
  
\end{lemma}
\begin{proof}
  
  On large scales, Proposition~\ref{prop:crucial_for_schauder} and
  Lemma~\ref{lem:boundedness_cut_off_operators} imply that
  \begin{align*}
    \| \mA_{n} \mP_{n} \varphi \|_{\mC^{\alpha-2}_{p}} \lesssim \|
    \mP_{n} \varphi \|_{\mC^{\alpha}_{p}} \lesssim \| \varphi
    \|_{\mC^{\alpha}_{p}}.
  \end{align*}
  Moreover on small scales the same results guarantee that for any
  \(\zeta \geq 0\):
  \begin{align*}
    \| \mQ_{n} \mA_{n} \varphi \|_{\mC^{\alpha - 2 - \zeta}_{p}} \lesssim
    n^{2} \sup_{2^j \gtrsim n} 2^{j(\alpha - 2 - \zeta)}\| \Delta_{j}
    \mQ_{n} \varphi \|_{L^{p}} \lesssim n^{-\zeta} \| \varphi \|_{\mC^{\alpha}_{p}},
  \end{align*}
  which tends to $0$ as $n$ tends to $\infty$ if \(\zeta>0\).
  Combining these two observations provides the first bound and 
  guarantees compactness in \(\mC^{\alpha -2-\zeta}_{p}\). Convergence
  follows since, by
  Lemma~\ref{lem:taylor_expansion_and_decay_ft_characteristic_function}, for
any \( k \in \ZZ^{d} \):
  \begin{align*}
    \mF_{\TT^{d}} [\mA_{n} \mP_{n} \varphi](k) = \daleth(n^{-1} k)
    n^{2} \big( \hat{\chi}^{2}(n^{-1} k) {-}
    1 \big) \hat{\varphi}(k) \to - (2 \pi)^{2} \nu_{0} |k|^{2}
    \hat{\varphi} (k) = \mF_{\TT^{d}}[\nu_{0} \Delta \varphi](k).
  \end{align*}

\end{proof}

The regularity gain provided by the operator
\(\mA_{n}\) can be described as follows (for the proof of
Theorem~\ref{thm:regularization-estimates-main-results} we require the result
only for \(\delta=0\)).

\begin{proposition}\label{prop:elliptic_schauder_estimates}
Fix any \(\alpha \in \RR, \ \delta \in [0,1]\) and \(p \in [1, \infty].\)
Uniformly over \(\lambda \geqslant 1, n \in \NN\) and \(\varphi \in \mC^{\alpha}_{p}\) the following estimates hold:
  \[ 
    \lambda^{\delta} \| \mP_{n} ( - \mA_{n} + \lambda)^{-1} \varphi
    \|_{\mC^{\alpha + 2(1 -
    \delta)}_{p}} + \lambda^{\delta} n^{2(1 - \delta)}\| \mQ_{n} (- \mA_{n} +
  \lambda)^{-1} \varphi\|_{\mC^{\alpha}_{p}} \lesssim \| \varphi
  \|_{\mC^{\alpha}_{p}}. \] 
  Moreover, as $n \to \infty$,
  \begin{align*}
  \mP_{n}(-\mA_{n} + \lambda )^{-1} \varphi \to (-\nu_{0}\Delta+
    \lambda )^{-1} \varphi
  \end{align*}
  where the convergence is in $\mC^{\alpha+2-\zeta}_{p}$ for any \(\zeta>0\) and
  $\nu_{0}$ is as in Lemma~\ref{lem:convergence_A_ve_to_laplace}.
\end{proposition}

\begin{proof}
  Consider the large-scale estimate.
  Proposition~\ref{prop:crucial_for_schauder} and
  Lemma~\ref{lem:boundedness_cut_off_operators} guarantee that for \(2^j \lesssim
  n\):
  \begin{align*}
    \| \Delta_{j} \mP_{n}(- \mA_{n} + \lambda)^{-1} \varphi
    \|_{L^{p}} \lesssim \frac{1}{2^{2j} + \lambda}2^{-\alpha j} \| \mP_{n}
    \varphi \|_{\mC^{\alpha}_{p}} \lesssim 2^{-2j (1 - \delta) - \alpha j}
    \lambda^{- \delta} \| \varphi \|_{\mC^{\alpha}_{p}},
  \end{align*}
  which is a bound of the correct order. All other bounds follow similarly, and
  the proof of the convergence is analogous to the one in
  Lemma~\ref{lem:convergence_A_ve_to_laplace}.
\end{proof}

\subsection{Parabolic regularity}

In this subsection we study the regularization effect of the semigroup
\(e^{t \mA_{\ve}}\). This discussion requires certain spaces of time-dependent
functions, which we introduce in the following. Let us fix \(T>0\) an arbitrary
time horizon. All function spaces will implicitly depend on \(T\).
For time dependent functions taking values in a Banach space \(\mX\) 
the $\alpha$-H\"older norm (with \(\alpha \in (0,1)\)) is defined as
\[ \| f \|_{C^{\alpha}\mX} = \sup_{t \in [0,T]} \| f(t) \|_{\mX} +
\sup_{t, s \in [0,T]} \frac{ \| f(t) {-} f(s) \|_{\mX}}{|t {-} s |^{\alpha}}.\] 
It is convenient to incorporate a blow-up at time \(t=0\). This reflects the
fact that the regularization of the semigroup occurs only at strictly positive
times.
\[ \mM^{\gamma} \mC^{\alpha}_{p} = \{ f \colon (0, T] \to
\mC^{\alpha}_{p} \ | \ \| f \|_{\mM^{\gamma} \mC^{\alpha}_{p}} =
\sup_{t \in [0,T]} t^{\gamma} \| f(t) \|_{\mC^{\alpha}_{p}} < \infty\}, \]
and one can combine the previous spaces in the following way:
\[ \LLL^{\gamma, \alpha}_{p} = \{ f \in \mM^{\gamma} \mC^{\alpha}_{p} \ | \ \|
f \|_{\LLL^{\gamma, \alpha}_{p}} = \| f \|_{\mM^{\gamma}
\mC^{\alpha}_{p}} + \| t \mapsto t^{\gamma}f(t) \|_{C^{\alpha/2}L^{p}} < \infty \}.\] 
Now we state the main result of this section, the   parabolic
Schauder estimates.

\begin{proposition}\label{prop:schauder_estimates}
  Fix \(p \in [1, \infty], T >0, \gamma \in [0,1)\) and \(\alpha \in ({-} 2, 0), \beta
  \in [\alpha, \alpha {+} 2) \cap (0, 2) \). 
  Uniformly over \(\varphi \in \mC^{\alpha}_{p}\) and \(f \in
  \mM^{\gamma} \mC^{\alpha}_{p}\) and locally uniformly over \(T>0\):
  \begin{align}
   \label{eq:par_schauder_est_spatial_large}
   \| t \mapsto \mP_{n} e^{t \mA_{n}} \varphi \|_{\LLL^{(\beta {-}
    \alpha)/2, \beta}_{p}} \lesssim \| \mP_{n} \varphi\|_{\mC^{\alpha}_{p}},
    \end{align}
    \begin{align}
     \label{eq:par_schauder_est_temporal_large}
     \ \ \Big\| t \mapsto \int_{0}^{t} \mP_{n} e^{(t {-} s)
    \mA_{n}} f(s) \ud s \Big\|_{\LLL^{\gamma, \alpha {+} 2 }_{p}} \lesssim \|
    \mP_{n} f \|_{\mM^{\gamma}\mC^{\alpha}_{p}}.
    \end{align}
  Next let $\zeta_{1}, \zeta_{2} \in [0,1)$ be such that \(\zeta_{1} {+}
  \zeta_{2} <1\) and  \(\delta_{1},\delta_{2},\delta_{3} \in [0,1]\) such that \(\delta_{1} {+}
  \delta_{2} {+} \delta_{3} =1\). Then:
  \begin{align}
   \label{eq:par_schauder_est_spatial_small} 
    \| t \mapsto t^{\zeta_{1}+ \zeta_{2}}\mQ_{n} e^{t \mA_{n}} \varphi
    \|_{C^{\zeta_{1} } \mC^{\alpha}_{p}} \lesssim n^{-2 \zeta_{2}}\| \mQ_{n} \varphi\|_{\mC^{\alpha}_{p}},
     \end{align}
    \begin{align}
     \label{eq:par_schauder_est_temporal_small}
     \ \ \Big\| t \mapsto t^{\gamma}
    \int_{0}^{t} e^{(t {-} s) \mA_{n}} \mQ_{n}f(s) \ud s \Big\|_{C^{
    \delta_{1}} \mC^{\alpha}_{p}} \lesssim n^{-2
    \delta_{2}}T^{\delta_{3}}\| \mQ_{n} f \|_{\mM^{\gamma}
    \mC^{\alpha}_{p}}.
    \end{align}
     with constants independent of $f,\varphi,T$. 
\end{proposition}

In many steps the proof mimics proofs in 
\cite{GubinelliImkellerPerkowski2015} and \cite{GubinelliPerkowski2017KPZ}, to
which we will often refer.

\begin{proof}

  \textit{Step 1.} 
  We begin with large scales, namely \eqref{eq:par_schauder_est_spatial_large}.
  By Proposition~\ref{prop:crucial_for_schauder}: 
  \begin{align*}
    \sup_{j \geq {-} 1} 2^{\beta j} \| \Delta_{j} \mP_{n} e^{t
    \mA_{n}} \varphi \|_{L^{p}(\TT^{d})} & \lesssim \sup_{j \geq -1}e^{{-} c t
    2^{2j}} 2^{(\beta {-} \alpha)j} \| \mP_{n}\varphi
    \|_{\mC^{\alpha}_{p}}\\
    & = t^{{-} \frac{\beta {-} \alpha}{2}  } \sup_{j \geq {-} 1 } e^{{-} c t
    2^{2j}} (t 2^{2j})^{\frac{\beta {-} \alpha}{2} } \| \mP_{n} \varphi
    \|_{\mC^{\alpha}_{p}} \lesssim t^{{-} \frac{\beta {-} \alpha}{2} } \|
    \mP_{n} \varphi \|_{\mC^{\alpha}_{p}}.
  \end{align*}
    Therefore 
  $ \| t \mapsto \mP_{n}e^{t \mA_{n}} \varphi \|_{\mM^{(\beta {-}
  \alpha)/2} \mC^{\beta}_{p}} \lesssim \| \mP_{n} \varphi \|_{\mC^{\alpha}_{p}}. $
  Similarly, for \eqref{eq:par_schauder_est_temporal_large}
  \begin{align*}
    \sup_{j \geq {-} 1} 2^{j(\alpha {+} 2)} \Big\| \int_{0}^{t}
    \Delta_{j} e^{(t {-} s) \mA_{n}} \mP_{n} f(s) \ud s
    \Big\|_{L^{p}(\TT^{d})} \lesssim \| \mP_{n} f
    \|_{ \mM^{\gamma} \mC^{\alpha}_{p}} \sup_{j \geq {-} 1}2^{j 2}
    \int_{0}^{t} e^{{-} c s 2^{2j}} (t {-} s)^{{-} \gamma}\ud s,
  \end{align*}
  which can be bounded by \(\| \mP_{n} f \|_{\mM^{\gamma}
  \mC^{\alpha}_{p}}\) by the same arguments as in the proof of
  \cite[Lemma A.9]{GubinelliImkellerPerkowski2015}.
  We still need to address the temporal regularity for both terms. 
  By Proposition~\ref{prop:crucial_for_schauder} 
  \begin{equation}\label{eqn:proof-schauder-large-scale-time}
    \begin{aligned}
    \| (e^{t \mA_{n}} {-} \mathrm{Id}) \mP_{n} \varphi \|_{L^{p}(\TT^{d})}
    & = \Big\| \int_{0}^{t} e^{s \mA_{n}} \mA_{n} \mP_{n} \varphi \ud s
    \Big\|_{L^{p}(\TT^{d})} \\
    & \lesssim \int_{0}^{t} s^{{-} 1 {+}
    \frac{\alpha}{2} } \| \mP_{n} \varphi \|_{\mC^{\alpha}_{p}}\ud s \simeq
    t^{\frac{\alpha}{2} } \| \mP_{n} \varphi \|_{\mC^{\alpha}_{p}}.
    \end{aligned}
  \end{equation}
  To conclude the proof of both \eqref{eq:par_schauder_est_spatial_large} and
  \eqref{eq:par_schauder_est_temporal_large}  it is now sufficient to follow
  the same steps as in \cite[Lemma 6.6]{GubinelliPerkowski2017KPZ}.

  \textit{Step 2.} 
  We  turn our attention to the small scale bounds \eqref{eq:par_schauder_est_spatial_small} and
  \eqref{eq:par_schauder_est_temporal_small}.
  Fix \(\zeta_{1} = \delta_{1} = 0\) first. 
  With calculations in the same spirit as in the Step~1, we arrive at 
  \begin{align*}
    \| \mQ_{n} e^{t \mA_{n}} \varphi \|_{\mC^{\alpha}_{p}} = \sup_{j \geq
    {-} 1} 2^{\alpha j} \| \Delta_{j} \mQ_{n} e^{t \mA_{n}} \varphi
    \|_{L^{p}(\RR^{d})} \lesssim e^{{-} c t n^{2}} \| \mQ_{n} \varphi
    \|_{\mC^{\alpha}_{p}} \lesssim (t n^{2})^{{-} \zeta_2 } \| \mQ_{n}
    \varphi \|_{\mC^{\alpha}_{p}}.
  \end{align*}
  For \eqref{eq:par_schauder_est_temporal_small}, if  \(
  \delta_{3} >0\)  the spatial bound follows from the previous result. If
  \(\delta_{3}=0\), we bound
  \begin{align*}
    \Big\| \int_{0}^{t} \mQ_{n} e^{(t {-} s) \mA_{n}} f(s) \ud s
    \Big\|_{\mC^{\alpha}_{p}} \lesssim \| \mQ_{n} f
    \|_{\mM^{\gamma}\mC^{\alpha}_{p}} \int_{0}^{t} e^{{-} c s n^{
    2}}(t {-} s)^{{-} \gamma} \ud s \lesssim n^{-2} t^{{-} \gamma} \| \mQ_{n} f
    \|_{\mM^{\gamma} \mC^{\alpha}_{p}}.
  \end{align*}
  The last bound in the above inequality is obtained in the same way as
  \cite[Lemma A.9]{GubinelliImkellerPerkowski2015}. 
  \begin{align*}
    \int_{0}^{t} e^{{-} c s n^{2}}(t {-} s)^{{-} \gamma}\ud s \lesssim
    t^{1 {-} \gamma} \lesssim t^{{-} \gamma} n^{-2}.
  \end{align*}

  \textit{Step 3.} We now investigate the full temporal regularity for \eqref{eq:par_schauder_est_spatial_small} and \eqref{eq:par_schauder_est_temporal_small},
that is, we allow for  \(\zeta_{1}, \delta_{1}>0\). 
 We first observe  that for \(\delta \in
  [0,1)\)
  \begin{equation}\label{eqn:proof-schauder-small-scale-time}
  \begin{aligned}
    \| ( e^{t \mA_{n}} {-} \mathrm{Id}) \mQ_{n} \varphi
    \|_{\mC^{\alpha}_{p}} &= \Big\| \int_{0}^{t} e^{s \mA_{n}} \mA_{n}
    \mQ_{n} \varphi \ud s \Big\|_{\mC^{\alpha}_{p}} \\
    & \lesssim \| \mQ_{n} \varphi \|_{\mC^{\alpha}_{p}} \int_{0}^{t}(s n^{2})^{{-} \delta}
    n^{{-} 2} \ud s =  \| \mQ_{n} \varphi \|_{\mC^{\alpha}_{p}}
    n^{-2(\delta {-} 1)} t^{1 {-} \delta}.
    \end{aligned}
  \end{equation}
  Hence for \(\zeta = \zeta_{1} {+} \zeta_{2} \in [0, 1)\),
  the temporal regularity of the first terms can be established via 
  \begin{align*}
    \| t^{\zeta}e^{t \mA_{n}} & \mQ_{n} \varphi {-} s^{\zeta} e^{s
    \mA_{n}} \mQ_{n} \varphi \|_{\mC^{\alpha}_{p}} \lesssim (t^{\zeta} {-}
    s^{\zeta}) t^{- \zeta_{2}} n^{-2 \zeta_{2}} \| \mQ_{n} \varphi
    \|_{\mC^{\alpha}_{p}} + s^{\zeta} \| ( e^{(t {-} s)
    \mA _{n}} {-} \mathrm{Id}) e^{s \mA_{n}} \mQ_{n} \varphi
    \|_{\mC^{\alpha}_{p}} \\
    & \lesssim (t^{\zeta} {-} s^{\zeta}) t^{- \zeta_{2}} n^{-2
    \zeta_{2}} \|
    \mQ_{n} \varphi \|_{\mC^{\alpha}_{p}} +
    s^{\zeta}(t {-} s)^{1 {-} \delta} n^{-2( \delta {-} 1)}
    \| e^{s \mA_{n}} \mQ_{n} \varphi \|_{\mC^{\alpha}_{p}} \\
    & \lesssim [(t^{\zeta} {-} s^{\zeta}) t^{- \zeta_{2}} n^{-2
      \zeta_{2}}  +
    (t {-} s)^{1 {-} \delta} n^{-2( \delta {-} 1)} n^{-2
  \zeta} ] \| \mQ_{\ve} \varphi \|_{\mC^{\alpha}_{p}} \\
  & \lesssim ( t {-}
  s)^{\zeta_{1}} n^{-2 \zeta_{2}} \| \mQ_{n} \varphi
    \|_{\mC^{\alpha}_{p}}, 
  \end{align*}
  where in the last step we set \(\delta = 1 {-} \zeta_{1}\)
  and notice that  \( (t^{\zeta} {-} s^{\zeta}) t^{- \zeta_{2}} \lesssim
  (t {-} s)^{\zeta_{1}}\). 
  
  The bound for \eqref{eq:par_schauder_est_temporal_small} follows similar pattern. For simplicity write \(V(t) =
  \int_{0}^{t} e^{(t {-} s) \mA_{n}} \mQ_{n} f(s) \ud s\):
  \begin{align*}
    \| t^{\gamma} V(t) {-} s^{\gamma}V(s) \|_{\mC^{\alpha}_{p}}\! \leq\!
    (t^{\gamma} {-} s^{\gamma})\| V(t) \|_{\mC^{\alpha}_{p}}\!+ \!s^{\gamma} \Big\|\!
    \int_{s}^{t} e^{(t {-} r)\mA_{n} } \mQ_{n} f(r)\! \ud r
    \Big\|_{\mC^{\alpha}_{p}} {+} s^{\gamma} \!\| (e^{(t {-} s) \mA_{n}} {-}
    \mathrm{Id}) V(s)\|_{\mC^{\alpha}_{p}}.
  \end{align*}
  The only term for which the estimation does not follow the already
  established pattern is the one in the middle, for which we observe that  
  \begin{align*}
    s^{\gamma} \Big\| &\int_{s}^{t} e^{(t {-} r) \mA_{n}} \mQ_{n} f(r) \ud r
    \Big\|_{\mC^{\alpha}_{p}}  \lesssim s^{\gamma} \int_{s}^{t} ( (t {-} r)
    n^{2})^{ {-} \delta_{2}}r^{ {-}\gamma} \ud r \| \mQ_{n} f
    \|_{\mM^{\gamma} \mC^{\alpha}_{p}} \\
    & \lesssim \| \mQ_{n} f \|_{\mM^{\gamma} \mC^{\alpha}_{p}}
    n^{-2 \delta_{2}} s^{\gamma} t^{{-} \delta_{2} {-} \gamma + 1} \int_{s/t}^{1}
    (1 {-} r)^{{-} \delta_{2}} r^{{-} \gamma}\ud r \\
    & \lesssim \| \mQ_{n} f
    \|_{\mM^{\gamma} \mC^{\alpha}_{p}} n^{-2 \delta_{2}} t^{1
    {-} \delta_{2} } \int_{s/t}^{1}(1 {-} r)^{{-} \delta_{2}} \ud r  \lesssim \| \mQ_{n} f \|_{\mM^{\gamma} \mC^{\alpha}_{p}}
    n^{-2 \delta_{2}} t^{1 {-} \delta_{2}} (1 {-}
    s/t)^{1 {-} \delta_{2}} \\
    & \leq \| \mQ_{n}f \|_{\mM^{\gamma}
    \mC^{\alpha}_{p}} n^{-2 \delta_{2}}(t {-} s)^{1 {-} \delta_{2}} \leq \| \mQ_{n}f \|_{\mM^{\gamma}
    \mC^{\alpha}_{p}} n^{-2 \delta_{2}} T^{\delta_{3}}(t {-}
    s)^{\delta_{1}},
  \end{align*}
  which completes the proof of the proposition.
\end{proof}

The following result is essentially a by-product of the previous proof.

\begin{lemma}\label{lem:schauder_short_time_contraction}
  Consider \(\alpha, \beta \in \RR\) and \(p \in [1, \infty]\) with \(\gamma :=\alpha - \beta \in [0, 2]\). Then
  uniformly over \(\varphi \in \mC^{\alpha}_{p}\) one can estimate
   $ \| (e^{t \mA_{n}} - \mathrm{Id}) \varphi \|_{\mC^{\beta}_{p}}
    \lesssim t^{\frac \gamma 2} \| \varphi \|_{\mC^{\alpha}_{p}}$.
\end{lemma}

\begin{proof}
  The proof follows from   Proposition~\ref{prop:schauder_estimates}.
  Indeed, Equation~\eqref{eqn:proof-schauder-large-scale-time} implies that for
  \(2^j \lesssim n\) one has:
  \[ 2^{j \beta}\| (e^{t \mA_{n}} - \mathrm{Id}) \Delta_{j} \varphi
    \|_{L^{p}} \lesssim t^{\frac \gamma 2} 2^{ j \beta} \|
    \Delta_{j} \varphi \|_{\mC^{\gamma}_{p}} \lesssim t^{\frac \gamma 2}\| \varphi
   \|_{\mC^{\alpha}_{p}}.\]
   While a slight modification (to \(L^{p}\) spaces) of
   \eqref{eqn:proof-schauder-small-scale-time} guarantees that for \(2^j \gtrsim n\):
   \[ 
     2^{j \beta}\| (e^{t \mA_{n}} - \mathrm{Id}) \Delta_{j}\varphi
     \|_{L^{p}} \lesssim t^{\frac \gamma 2} 2^{ j \beta} n^{\gamma} \|
     \Delta_{j} \varphi \|_{L^{p}} \lesssim t^{\frac \gamma 2} 2^{j \alpha} \|
     \Delta_{j} \varphi \|_{L^{p}} \lesssim t^{\frac \gamma 2}\| \varphi
     \|_{\mC^{\alpha}_{p}}.
   \] 
   This concludes the proof.
\end{proof}

\subsection{Besov spaces \& characteristic functions}\label{subsec:appendix-characteristic-functions}

In this subsection we collect some facts regarding Besov spaces and the
regularity of characteristic functions.
Let us begin by stating the Poisson summation formula (a proof is left to the
reader, or can be found in many textbooks and web pages).

\begin{lemma}\label{lem:poisson_summation_formula}
  For \(\varphi \in \mS(\RR^{d})\) it holds that:
  \( \mF^{-1}_{\TT^{d}} \varphi (x) = \sum_{z \in \ZZ^{d}} \mF_{\RR^{d}}^{-1}
  \varphi (x +z)\). In particular, this implies for \(\varphi \in \mS (\RR^{d})\) the bound:
  \(\| \mF_{\TT^{d}}^{-1} \varphi \|_{L^{1}(\TT^{d})} \leq \|
    \mF^{-1}_{\RR^{d}} \varphi \|_{L^{1}(\RR^{d})}.\) 
\end{lemma}
Recall that the Besov spaces \(B^{\alpha}_{p, q}(\TT^{d})\) are
defined via a dyadic partition of the unity \(\{\varrho_{j}\}_{j \geq -1}\)
such that for \(j \geq 0\), \(\varrho_{j} = \varrho(2^{j} \cdot)\) for a smooth function
\(\varrho\) with compact support in an annulus.
\begin{proposition}[Besov embeddings]
\label{prop:besov_embedding}
For any $1 \leq p_{1} \leq p_{2} \leq \infty$ and $1 \leq q_{1} \leq q_{2} \leq \infty$
the embedding $B_{p_{1},q_{1}}^{\alpha} \subseteq B_{p_{2},q_{2}}^{\alpha - d(1/p_{1} - 1/p_{2})}$ is continuous.
In addition, for \(\alpha^{\prime} < \alpha\) the embedding \(B^{\alpha}_{p_2,
  q_2} \subseteq B^{\alpha^{\prime}}_{p_1, q_1}\) is compact.
\end{proposition}

In certain cases, it will be convenient to use the following alternative
characterization of certain Besov spaces.
\begin{proposition}[Sobolev-Slobodeckij norm]\label{prop:sobolev_slobodeckij_equivalence}
For every $\alpha \in \RR_{+} \setminus\NN$ and for every $p \in [1,\infty)$ define
the Sobolev-Slobodeckij norm for \(\varphi \in \mS^{\prime}(\TT^{d})\) as:
 \[ \| \varphi \|_{W^{\alpha}_{p}} := \| \varphi \|_{L^{p}} + \sum_{|m| = \lfloor \alpha \rfloor} \bigg( \int_{\TT^{ d} \times
      \TT^{d}} \frac{| D^{m} \varphi(x) - D^{m} \varphi(y)|^{p}}{|x - y|^{d +
    (\alpha - \lfloor \alpha \rfloor)p}}  \ud x \ud y \bigg)^{ 1 /p} \in [0,
  \infty]. \] 
  There exist constants a pair of constants \(c(p), C(p) >0\) such that for
  $\varphi \in \mS^{\prime}(\TT^{d})$
  \[ c \| \varphi \|_{B^{\alpha}_{p, p}} \leq \| \varphi
  \|_{W^{\alpha}_{p}} \leq C \| \varphi \|_{B^{\alpha}_{p, p}}. \] 
\end{proposition} 
For a proof consult e.g. \cite{Triebe2010} Theorem 2.5.7
and the discussion in Section 2.2.2. The next result states the regularizing properties of convolutions.

\begin{lemma}\label{lem:regularity_of_convolutions}
   
 For \(p,q,r \in [1,
  \infty]\) satisfying \(\frac{1}{r} = \frac{1}{p} {+} \frac{1}{q} {-} 1\) and for any \(\varphi, \psi \in \mS^{\prime}(\TT^{d})\): 
  \[ \| \varphi * \psi
  \|_{C^{\alpha {+} \beta}_r} \lesssim \|f\|_{ C^{\alpha}_{p} } \| g \|_{
\mC^{\beta}_q}. \] 

\end{lemma}

\begin{proof} 

By Young's convolution inequality
\begin{equation}\label{eqn:convolution_single_paley_block} 
  \| \Delta_{i} (f * g)
  \|_{L^r} = \| \Delta_{i} f * \overline{\Delta} _{i} g\|_{L^r} \lesssim \|
  \Delta_{i} f\|_{L^p} \| \overline{\Delta}_{i}g \|_{L^q},  
\end{equation}
where \( \overline{\Delta}_{i}\) is associated with a dyadic partition
 of the unity different from the one we use for most of the proofs. Namely we require that it satisfies \( \{ \overline{\varrho_{j}}\}_{j \geq -1}\) such that
\(\varrho_{j} \overline{\varrho_{j}} = \varrho_{j}\). Then the bound follows
immediately, since the Besov norms associated to different dyadic partitions
are equivalent, cf. \cite[Remark~2.17]{BahouriCheminDanchin2011FourierAndNonLinPDEs}.  
\end{proof}
The following lemma is a special case of results obtained by \cite{sickel:1999}.
The proof is included for completeness.
\begin{lemma}
\label{lem:besov_regularity_characteristic_function}
Fix $ p \in [1, \infty), \ \zeta \in [0,\frac{1}{p})$. Then
 $\sup_{n \in \NN } n^{-\zeta -d+\frac d p} \| \chi_n \|_{W^{\zeta}_{p}} <
\infty.$
\end{lemma}
\begin{proof}
We shall make use of the characterization of fractional Sobolev spaces in terms
of Sobolev-Slobodeckij norms. 
A direct computation shows that 
\begin{align*}
\|\chi_{n}\|_{W^{\zeta}_{p}} 
& =\|\chi_{n}\|_{L^{p}} + \bigg( \int_{\TT^{d}\times\TT^{d}}  n^{dp} \frac{|1_{B_{n}}(x)-1_{B_{n}}(y)|^p}{|x {-} y|^{d {+} \zeta p}} \mathrm{d}x\mathrm{d}y \bigg)^{1/p}
\\
& \leq n^{d-\frac d p} + \bigg( 2\int_{B_{n}}\int_{\TT^{d} \setminus B_{n}}  n^{dp}\frac{|1_{B_{n}}(x)-1_{B_{n}}(y)|^p}{|x {-} y|^{d {+} \zeta p}} \mathrm{d}x\mathrm{d}y \bigg)^{1/p}.
\end{align*}
Now let $d_n(z)$ be the Euclidean distance of a point $z$ from the boundary $\partial B_n$ and let $\overline{B}_{d_n(z)}(y)$ be the ball of radius $d_n(z)$ about $y$. Then the previous integral can be estimated by:
\begin{align*}
    \bigg( \int_{B_{n}}\int_{\TT^{d} \setminus B_{n}} & n^{dp}\frac{|1_{B_{n}}(x)-1_{B_{n}}(y)|^p}{|x {-} y|^{d {+} \zeta p}} \mathrm{d}x\mathrm{d}y \bigg)^{1/p}  \leq \bigg( \int_{B_{n}}\int_{\TT^{d} \setminus \overline{B}_{d_n(y)}(y)}  n^{dp}\frac{1}{|x-y|^{d {+} \zeta p}} \mathrm{d}x\mathrm{d}y \bigg)^{1/p}  \\
    & =\bigg( \int_{B_{n}}\int_{\TT^{d} \setminus \overline{B}_{d_n(y)}(0)}  n^{dp}\frac{1}{|
    x|^{d {+} \zeta p}} \mathrm{d}x\mathrm{d}y \bigg)^{1/p}  \lesssim \bigg( \int_{B_{n}} n^{dp} d_n(y)^{-\zeta p }\mathrm{d}y \bigg)^{1/p} \\
    & \lesssim \bigg( \int_0^{\frac c n} n^{dp} (c/n-r)^{-\zeta p } r^{d-1}\mathrm{d}r \bigg)^{\frac 1 p} \lesssim n^{d} \Big( n^{\zeta p - d}\Big)^{1/p} \leq n^{d + \zeta -d/p}.
\end{align*}
\end{proof}

\begin{corollary}\label{cor:regularity_gain_convolution_with_characteristic_functions}
   Recall that we define the operator \( \Pi_{n} \) by
$ \Pi_{n} \varphi (x) = \chi_{n} * \varphi (x).$
Then, for \(\zeta \in [0,1), p\in[1, \infty]\) and \(\alpha \in \RR\) we have
   $ 
     \sup_{n \in \NN} n^{-\zeta} \| \Pi_{n} \varphi \|_{\mC^{\alpha {+}
     \zeta}_{p}} \lesssim \| \varphi \|_{\mC^{\alpha}_{p}}.  
   $ 
\end{corollary}
\begin{proof}
This is now a direct consequence of Lemma~\ref{lem:regularity_of_convolutions} and \ref{lem:besov_regularity_characteristic_function} (the latter with $p=1$).

\end{proof}

\section{Semidiscrete parabolic Anderson model}\label{sec:semidiscrete_PAM}

This section is devoted to the proof of 
Theorem~\ref{thm:semidiscrete-pam-approximation}. This theorem is an
approximation result for the continuous Anderson Hamiltonian in dimensions
\(d=1\) and \(d=2\). 
The Anderson Hamiltonian was introduced in $d=1$ by \cite{Fukushima1976}, in
$d=2$ by \cite{AllezChouk2015AndersonHamiltonian2D} and $d=3$ by
\cite{Labbe2018}. In the last two cases the construction relies on theories
from singular stochastic PDEs \cite{Hairer2014, GubinelliImkellerPerkowski2015},
which is why the proof of the theorem concentrates on the two-dimensional case. In
dimensions $d=4$ or higher these solution theories do not work,
because the noise becomes too rough (a problem known as \textit{supercriticality} \cite{Hairer2014}). 

In the construction of the
Hamiltonian in \(d=2\) we follow the results in
\cite{AllezChouk2015AndersonHamiltonian2D} that rely on paracontrolled
calculus (we refer the reader to
\cite{GubinelliImkellerPerkowski2015} and \cite{GubinelliPerkowski2017KPZ} for
a more in-depth discussion). Our main result states that semidiscrete
approximations converge in the \emph{resolvent sense} to the continuous
Anderson Hamiltonian.

\subsection{Density of the domain}

We start with some results regarding the continuous Anderson Hamiltonian, which imply Proposition~\ref{prop:continuous-pam-operator}.

\begin{lemma}\label{lem:semidiscrete-pam:continuous-pam-operator}
     Consider a probability space $(\Omega, \mF, \PP)$ supporting a space white noise $\xi \colon \Omega \to \mS^\prime (\TT^d)$. Fix any $\kappa>0$. The following hold true for almost all $\omega\in\Omega$.
  The Anderson Hamiltonian
  \[\mH^\omega = \nu_0\Delta +
  \xi(\omega)\] associated to \(\xi(\omega)\) is defined,
  as constructed in \cite{Fukushima1976} in \(d=1\) and
  \cite{AllezChouk2015AndersonHamiltonian2D} in \(d=2\). The Hamiltonian, as an
  unbounded selfadjoint operator on $L^2(\TT^d)$, has a discrete
  spectrum given by pairs of eigenvalues and eigenfunctions
  \(\{(\lambda_{k}(\omega), e_{k}(\omega))\}_{k \in \NN}\) such that:
  \[ 
    \lambda_{1} (\omega) > \lambda_{2} (\omega) \geq \lambda_{3} (\omega) \geq
    \ldots, \qquad \lim_{k \to \infty} \lambda_{k}(\omega) = {-} \infty, \qquad
    e_{1}(\omega, x)>0, \forall x \in \TT^{d}.\]
\end{lemma}

\begin{proof}

The Hamiltonian $\mH^\omega$ has been constructed in dimension
\(d=1\) in \cite{Fukushima1976} (albeit with Dirichlet boundary conditions, but
the construction for periodic boundary conditions is identical)
and in dimension \(d=2\) in \cite{AllezChouk2015AndersonHamiltonian2D}, for
almost all \(\omega \in \Omega\). In both
cases \(\mH^\omega\) is an unbounded, selfadjoint
operator on $L^{2}$, that is:
\begin{align*}
  \mH^\omega \colon \mD(\mH^\omega) \subset L^2 \to L^2.
\end{align*}
In particular, in \(d=2\) \cite[Proposition 4.13]{AllezChouk2015AndersonHamiltonian2D}
implies that the operator $\mH^\omega$ admits compact resolvents (cf.
\cite[Section 2]{Fukushima1976} for the analogous discussion in \(d=1\)). This means that for some $\overline{\lambda}(\omega)>0$ for all $\lambda \geq \overline{\lambda}(\omega)$ the operator $\mH^{\omega} - \lambda$ is invertible, and $(\mH^{\omega} - \lambda)^{-1}$ is a compact operator on $L^2$. Hence
the spectrum of $\mH^\omega$ is discrete and the eigenvalues converge to $-\infty$.
By a classical result, see  \cite[Theorem 3.3]{Pazy1983Semigroups}, the
semigroup generated by $\mH^\omega$, denoted by $e^{t\mH^\omega}$, is compact. 
Moreover, as a consequence of strong maximum principle (in \(d=2\) such a result for singular stochastic
PDEs is proven in  \cite[Theorem 5.1 and Remark 5.2]{Cannizzaro2017Malliavin}),
the semigroup  $e^{t\mH^\omega}$ is strictly positive:
that is, for any non-zero continuous function $f$ that is positive (i.e.\ $f(x) \geq 0, \ \forall x \in \TT^d$), it holds that $e^{t\mH^\omega}f(x) > 0, \ \forall x \in \TT^d$. 
Therefore since $e^{t\mH^\omega}$ is a compact, strictly positive operator, the
Krein-Rutman Theorem \cite[Theorem~19.3]{Deimling1985} implies that the largest eigenvalue of \(\mH^\omega\) has
multiplicity one and the associated eigenfunction is strictly positive.

\end{proof}

\begin{lemma}\label{lem:domain_for_anderson_hamiltonian}
  Fix \(\omega \in
  \Omega\) and consider the Anderson Hamiltonian \(\mH^\omega\)
  as in the previous lemma. Define the domain:
  \[ \mD_{\omega} = \{ \text{Finite linear combinations of }
  \{e_{k}(\omega)\}_{k \in \NN}\}.\] 
  The domain $\mD_{\omega}$ is dense in \(C(\TT^{d})\). Moreover, for arbitrary $\zeta \in (0, 1)$ and all \(\varphi \in
  C^{\infty}\), there exists a sequence \(\varphi^{k} \in
  \mD_{\omega}\) with \(\lim_{k \to \infty} \varphi^{k}= \varphi\) in
  \(\mC^{\zeta}\).
\end{lemma}

\begin{proof}
  Since \(\omega \in \Omega\) is fixed, we avoid writing the dependence on it
  to lighten the notation. As the statement regarding the approximation of \(\varphi\) in
  \(\mC^{\zeta}\) implies density in \(C(\TT^{d})\) we restrict to proving the
  approximation. First, we require some better understanding of the parabolic Anderson
  semigroup. Here we make use of some known regularization results.
  
  \textit{Step 1.} 
  Consider the operator \(\mH\) as in
  the previous lemma and the associated semigroup:
  \[ e^{t \mH} \colon L^{2}(\TT^{d}) \to L^{2}(\TT^{d}).\] 
  This semigroup inherits some of the regularizing properties of the heat
  semigroup, namely, for \(T >0\) and \(p \in [1, \infty]\) it can be extended so that:
  \begin{equation}\label{eqn:proof-density-first-regularization} 
    \sup_{0 < t \leq T } t^{\gamma}\| e^{t \mH} \varphi \|_{\mC^{\alpha}_{p}}
    \lesssim \| \varphi \|_{\mC^{\beta}_{p}}, 
  \end{equation}
  for \(\alpha, \beta\) and $\gamma$ satisfying:
  \[ \gamma> \frac{\alpha- \beta}{2} , \qquad \beta +2 > \frac{d}{2}, \qquad \alpha < 2 - \frac{d}{2}, \qquad \alpha
  > \beta.\]
  The first constraint is essentially identical to the one appearing in
  Schauder estimates (cf.  Proposition~\ref{prop:schauder_estimates}), the
  second one guarantees that the product \(e^{t \Delta} \varphi \cdot \xi\) is a
  well-defined product of distributions, while the third constraint is due to
  the fact that \(\smallint_0^t e^{(t-s) \Delta} \xi \ud s\) has always worse
  regularity than \(2 - \frac{d}{2}\). 
  Similarly, for \(\beta>2-
  \frac{d}{2}\) and \(\zeta< 2 - \frac{d}{2}\) one has:
  \begin{equation}\label{eqn:proof-density-second-regularization} 
    \sup_{0 \leq t \leq T} \| e^{t \mH} \varphi \|_{\mC^{\zeta}_{p}} \lesssim \| \varphi
    \|_{\mC^{\beta}_{p}}.
  \end{equation}
  We will not prove these results. Instead we refer to \cite[Section 6]{GubinelliPerkowski2017KPZ} for the study of singular SPDEs with irregular initial conditions.
  
  \textit{Step 2.} Applying iteratively 
  Equation~\eqref{eqn:proof-density-first-regularization} and Besov embedding
  implies that \(e_{k} \in \mC^{2 - \frac d 2 - \kappa}\) for any \(\kappa>0\).
  Hence the embedding \(\mD_{\omega} \subseteq \mC^{2 - \frac d 2 - \kappa}\)
  holds. Now we prove the statement regarding the approximability of
  \(\varphi\). For any \(\varphi \in C^{\infty}\) and \(\zeta = 1 - \kappa <
  1\) (for some \(\kappa>0\)) one has:
  \[\lim_{t \to 0^+}  \frac{1}{t} \int_{0}^{t} e^{s \mH} \varphi \ud s = \varphi
  \ \ \text{ in } \ \ \mC^{\zeta}.\] 
  This result can be seen as follows:
  Equation~\eqref{eqn:proof-density-second-regularization} implies that
  \[ \sup_{0 \leq t \leq T} \bigg\| \frac{1}{t} \int_{0}^{t} e^{s \mH} \varphi
    \ud s \bigg\|_{\mC^{\zeta^{\prime}}}
  < \infty, \] 
  for \(\zeta < \zeta^{\prime} < 2 - \frac{d}{2}.\) The estimate above implies compactness
  in \(\mathcal{C}^{\zeta}.\) Projecting on the eigenfunctions \(e_{k}\) one
  sees that any limit point is necessarily \(\varphi\).
  Hence fix any \(\ve>0\) and choose \(t(\ve)\) such that
  \[ \bigg\| \frac{1}{t(\ve)} \int_{0}^{t (\ve)} e^{s \mH} \varphi \ud s -
  \varphi \bigg\|_{\mC^{\zeta}} < \frac{\ve}{2}.\]
  Define \(\Pi_{\leq N} \varphi = \sum_{k =0}^{N} \langle \varphi,
  e_{k} \rangle e_{k} \). Since the projection commutes with the operator, the
  proof is complete if we can show that there exists an \(N(\ve)\) such that:
  \[ \bigg\| \frac{1}{t(\ve)} \int_{0}^{t (\ve)}e^{s \mH} (\Pi_{\leq
  N(\ve)} \varphi- \varphi) \ud s \bigg\|_{\mC^{\zeta}} \leq \frac \ve 2.\] 
  Here we use \eqref{eqn:proof-density-first-regularization} to bound
  for general \(\psi \in L^{2}\):
  \begin{align*}
    \bigg\| \frac{1}{t(\ve)} \int_{0}^{t (\ve)}e^{s \mH} \psi \ud s
    \bigg\|_{\mC^{\zeta}} & \lesssim \frac{1}{t(\ve)} \int_{0}^{t(\ve)} \Big(
    \frac{s}{2}  \Big)^{-\big( \frac 1 2- \frac \kappa 4\big)} \| e^{\frac s 2 \mH} \psi
    \|_{\mC^{- \frac \kappa 2}} \ud s \\
    & \lesssim \frac{1}{t(\ve)} \int_{0}^{t(\ve)} \Big(
    \frac{s}{2}  \Big)^{-\big(\frac 1 2- \frac \kappa 4\big)} \| e^{\frac s 2 \mH} \psi
    \|_{\mC^{\frac d 2 - \frac \kappa 2}_{2}} \ud s \\
    & \lesssim \bigg( \frac{1}{t(\ve)} \int_{0}^{t(\ve)} s^{-1 + \frac \kappa 4 +
    \frac \kappa 8} \ud s \bigg) \| \psi \|_{L^{2}} \lesssim t(\ve)^{-1 + \frac
    {3\kappa}{8}} \| \psi \|_{L^{2}},
  \end{align*}
  where we additionally applied Besov embedding. Choosing
  \(N(\ve)\) such that \(\| \Pi_{\leq N} \varphi - \varphi \|_{L^{2}} \lesssim
  t(\ve)^{1- \frac{3 \kappa}{8} } \frac \ve 2,\) the proof is complete.
  \end{proof}

\subsection{Convergence of eigenfunctions} 
  Before we move on to study semidiscrete approximations of the Anderson
  Hamiltonian, we recall and adapt a result by Kato concerning the convergence of
  eigenvalues and (in a generalized sense) the convergence of eigenfunctions of
  a sequence of closed linear operators
  on a Hilbert space \(H\) with norm \(\| \cdot \| = \sqrt{ \langle \cdot,
  \cdot  \rangle}\). We will denote with
  \[ \sigma (A), \varrho(A) \subseteq \CC \] 
  the spectrum and the resolvent sets of a closed linear operator \(A\) on
  \(H\) respectively. If \(A\) is bounded, we denote with \(\| A \| = \sup_{\| x \| = 1} \| A x \|\)
  its operator norm. We write \(B(H)\) for the space of bounded operators,
  endowed with operator norm. Moreover, we denote with \(\mathrm{Rng}(A)\) the image \(A(H)\) of a
closed operator on \(H\).\\
  Now, consider a bounded set \(\Omega \subseteq \CC\)
  such that the boundary \(\Gamma = \partial \Omega\) is a smooth curve
  satisfying \(\Gamma \subseteq \varrho(A)\). We write \(R(A, \zeta) =
  (A - \zeta)^{-1}\) for the resolvent of \(A\) at \(\zeta \in \varrho(A)\). Then
  we introduce the Riesz projection
  \[ P (\Omega, A) = - \frac{1}{2 \pi \iota} \int_{\Gamma} R(A, \zeta) \ud
  \zeta,\]
  which for all our purposes coincides with the projection on certain
  eigenspaces, as described in the following lemma, which is proven for example
  in \cite[Proposition 6.3]{HislopSigal-IntroductionSpectralTheory}.
  \begin{lemma}\label{lem:riesz-projection}
    Let \(A\) be a selfadjoint operator on \(H\). Suppose that
    \(\Omega\) (with boundary \(\Gamma\) as above) contains only isolated points of the spectrum:
    \(\Omega \cap \sigma(A) = \{\lambda_{i}\}_{i = 1}^{m}\). Then \(P(\Omega, A)\) coincides
    with the orthogonal projection on the space:
    \[ \bigcup_{i = 1}^{m} \mathrm{Ker}(A- \lambda_{i}). \] 
  \end{lemma}
  Next we recall that Riesz projections are continuous with
  respect to convergence in the resolvent sense. This is a weaker version of a 
  result by Kato \cite[Theorem IV.3.16]{Kato1995}.
  \begin{proposition}
    Let $A_{n}$ be a sequence of closed self-adjoint operators on \(H\). Let
    \(A\) be a closed self-adjoint operator such that, for some
    \(\zeta_{0} \in \varrho(A)\):
    \[ \zeta_{0} \in \varrho(A_{n}), \ \ \forall n \in \NN, \qquad \text{and}
    \qquad \lim_{n \to \infty} \| R(A^{n}, \zeta_{0}) - R(A, \zeta_{0}) \| =0. \] 
    Let \(\lambda\) be an isolated eigenvalue of \(A\) and consider a smooth curve
    \(\Gamma= \partial \Omega\) around \(\lambda\), such that \(\Gamma \subseteq
    \varrho (A), \ \Omega \cap \sigma(A) = \{\lambda\}.\)
    Then $\lim_{n \to \infty}  \| P( \Omega, A_{n}) - P( \Omega, A) \| = 0.$ 
\end{proposition}
The previous result allows us to deduce the following.
\begin{corollary}\label{cor:convergence-eigenfunctions}
  In the setting of the previous proposition, let $\{
  e_j \}_{j = 1}^{m(\lambda)}$ be orthonormal eigenfunctions associated to the
  eigenvalue $\lambda$ of the operator \(A\) (here \(m (\lambda)\) is the multiplicity of
  \(\lambda\)).
  There exists an $n(\lambda) \in \NN$ such that for all
  $n \geq n(\lambda)$ the following statements hold.
  \begin{itemize}
    \item[i)] \( \dim ( \mathrm{Rng} (P (\Omega, A_{n})) ) = \dim ( \mathrm{Rng}(P (\Omega,
  A))) = m(\lambda).\)
\item[ii)] For every $j \in \{ 1, \ldots m(\lambda) \}$ there exists an
  $e_j^{n} \in \mD(A_{n})$ (the domain of \(A_{n}\)) satisfying:
  \[ e_j^{n} \to e_j \quad \text{in} \quad H, \qquad
     A_{n} e^{n}_j \to \lambda e_j \quad 
     \text{in} \quad H. \]
   \item[iii)] For every \(j \in \{1, \dots, m(\lambda)\}\), \(e_{j}^{n}\) has a representation of the form
     \begin{align*}
       e^{n}_{j}= \sum_{i = 1}^{m(\lambda)} \alpha_{ij}^{n}
       \overline{e}^{n}_{i},\qquad \sum_{i = 1}^{m (\lambda)}
       (\alpha_{ij}^{n})^{2} = 1,
     \end{align*}
     with \(\{ \overline{e}_{i}^{n}\}_{i = 1, \dots, m (\lambda)}\) a set of
     eigenfunctions of \(A_{n}\). That is, for every \(i \in \{1, \dots,
       m(\lambda)\}\): 
     \[A_{n} \overline{e}^{n}_{i} =
       \lambda^{n}_{i} \overline{e}^{n}_{i}, \quad \text{for some} \quad
       \lambda^{n}_{i} \in \RR \ \text{s.t.} \
     \lim_{n \to \infty} \lambda^{n}_{i} = \lambda.\]
   \item[iv)] If \(\lambda\) is a simple eigenvalue, then \(e_{1}^{n}\) is an eigenfunction of \(A_{n}\), with eigenvalue
     \(\lambda_{1}^{n} \to \lambda\).
 \end{itemize}
\end{corollary}
\begin{proof}
  Consider $m_{n} (\lambda) = \dim ( \mathrm{Rng} (P (\Omega,
  A_{n})) )$ and $\{ \overline{e}_j^{n} \}_{j = 1}^{m_n (\lambda)}$ an
  orthonormal basis for the subspace on
  which $P (\Omega, A_{n})$ projects. In particular, in view of
  Lemma~\ref{lem:riesz-projection}, we can choose \(
  \overline{e}_{j}^{n}\) to be eigenfunctions for \(A_{n}\), each associated to an eigenvalue
  \(\lambda^{n}_{j}\). According to the same lemma, one has:
  \[ P (\Omega, A_{n}) v = \sum_{i = 1}^{m_n (\lambda)}
    \langle v, \overline{e}_j^{n} \rangle \overline{e}_j^{n}, \qquad \forall v \in
H. \]
Define for $j = 1, \ldots, m(\lambda)$:
  $ \tilde{e}^{n}_j = \sum_{i = 1}^{m_n (\lambda)} \langle e_j,
    \overline{e}_i^{n} \rangle \overline{e}_i^{n} = P (\Omega,
     A_{n}) e_j. $
  From the convergence
  $ \| P (\Omega, A_{n}) - P (\Omega, A) \| \rightarrow 0,$
  which is the content of the previous proposition, we obtain that for $j = 1,
  \ldots, m(\lambda)$:
  \[ \lim_{n \rightarrow \infty} \tilde{e}_j^{n} : =
     \lim_{n \rightarrow \infty} \sum_{i = 1}^{m_n (\lambda)} \langle
     e_j, \overline{e}_i^{n} \rangle \overline{e}_i^{n} = e_j \quad \text{in}
   \quad H. \]
   Hence we can assume that \(n(\lambda)\) is sufficiently large, so that
  \[ \Big\vert \| \tilde{e}_j^{n} - \sum_{i =
     1}^{j - 1} \langle \tilde{e}_j^{n}, \tilde{e}_i^{n}
   \rangle \tilde{e}_i^{n} \| - 1 \Big\vert \geq \frac{1}{2} >0, \qquad \forall j=
 1, \dots, m(\lambda).\] 
   Then we can define (via a Gram-Schmidt procedure)
  \[ e_j^{n} = \frac{\tilde{e}_j^{n} - \sum_{i =
     1}^{j - 1} \langle \tilde{e}_j^{n}, \tilde{e}_i^{n}
     \rangle \tilde{e}_i^{n}}{\left\| \tilde{e}_j^{n} -
     \sum_{i = 1}^{j - 1} \langle \tilde{e}_j^{n},
     \tilde{e}_i^{n} \rangle \tilde{e}_i^{n} \right\|}, \]
     and we obtain a set $\{ e_j^{n} \}_{j = 1}^{m(\lambda)}$ of orthonormal functions with
  \[ \lim_{n \rightarrow \infty} e_j^{n} = e_j \quad \text{in}
     \quad H, \qquad P (\Omega, A_{n})
     e_j^{n} = e_j^{n}. \]
     In particular, $m_n (\lambda) \geqslant m(\lambda).$ Suppose $m_n
     (\lambda) > m(\lambda)$ on a subsequence \(n_{k}\) of $n$ that converges to
  \(\infty\). Choose, along that subsequence, a unit element $e_{m(\lambda) +
  1}^{n_{k}} \in H$ with
  \[ P(\Omega, A_{n_{k}}) e^{n_{k}}_{m ( \lambda) +1} = e^{n_{k}}_{m(\lambda)+1}, \qquad \langle e_{m (\lambda) + 1}^{n_{k}}, e_j^{n_{k}} \rangle = 0,
  \qquad \forall j = 1, \ldots, m(\lambda). \]
  We can then assume for arbitrary $\delta$ (provided $n_{k}$ is large enough),
  that
  $ \sum_{j = 1}^{m (\lambda) } \| e_j^{n_{k}} - e_j \| < \delta .$
  Then
  \begin{align*}
    \| P (\Omega, A_{n_k}) (e_{m(\lambda) + 1}^{n_k}) - P
    (\Omega, A) (e_{m(\lambda) + 1}^{n_{k}}) \| & \geqslant 1 - \left\|
    \sum_{j = 1}^{m(\lambda)} \langle e_{m(\lambda) + 1}^{n_{k}},
    e_j^{n_{k}} - e_j \rangle e_j \right\| \geqslant  1 - \delta .
  \end{align*}
  Since \(\delta\) is arbitrarily small this contradicts the convergence of the
  projections. \\
  Let us pass to the convergence of \(A_{n} e^{n}_{j}\). Observe that
  $ A_{n} \tilde{e}_j = \sum_{i = 1}^{m (\lambda)}
    \lambda_{i}^{n} \langle e_j, \overline{e}_i^{n} \rangle
  \overline{e}_i^{n}.$ Hence
  \begin{align*}
    \| A_{n} \tilde{e}_j^{n} - \lambda e_j \| & \leqslant  \| A_{n} \tilde{e}_j^{n}
    - \lambda \tilde{e}_j^{n} \| + \lambda \| \tilde{e}_j^{n} - e_j \|\\
    & \leqslant \sqrt{\sum_{i = 1}^{m(\lambda)} (\lambda - \lambda_{i}^{n})^2
    \langle e_j, \overline{e}_i^{n} \rangle^2} + \lambda \|
    \tilde{e}_j^{n} - e_j \|  \leqslant \sqrt{\sum_{i = 1}^{m_n} (\lambda -
    \lambda_{i}^{n})^2}+ \lambda \| \tilde{e}_j^{n} - e_j \|,
  \end{align*}
  and the last two terms converge to zero, provided that for each \(i\)
  $ \lim_{n \to \infty} \lambda_{i}^{n} = \lambda. $
  This follows from the upper semicontinuity of the
  spectrum proven in \cite[Theorem IV.3.1]{Kato1995}. If we now use the definition of
  $e_j^{n}$ we obtain similarly that:
  \begin{align*}
    \lim_{n \rightarrow \infty} A_{n} e_j^{n} &
    = \lim_{n \rightarrow \infty} \frac{A_{n}
    \tilde{e}_j^{n} - \sum_{i = 1}^{j - 1} \langle
    \tilde{e}_j^{n}, \tilde{e}_i^{n} \rangle
    A_{n} \tilde{e}_i^{n}}{\left\|
    \tilde{e}_j^{n} - \sum_{i = 1}^{j - 1} \langle
    \tilde{e}_j^{n}, \tilde{e}_i^{n} \rangle
    \tilde{e}_i^{n} \right\|} = \lambda e_j .
  \end{align*}
  To conclude the proof, note that the representation of \(e^{n}_{j}\) in terms
  of the basis \( \{\overline{e}_{i}^{n}\}\) follows from the fact that the
  latter consists of orthonormal functions and that \(\| e_{j}^{n} \| = 1\).
  Clearly, if \(m( \lambda) = 1\) we can choose \(e^{n}_{1} =
  \overline{e}^{n}_{1}\).
\end{proof}

\subsection{Convergence in resolvent sense}

This section describes the general idea behind the convergence that we will
prove in the upcoming subsection. 
\begin{proposition}\label{prop:convergence-resolvent-sense}
  Consider a sequence of selfadjoint operators \(A_{n}\) on a Hilbert space
  \(H\). Assume there exists a \(\lambda_{0} \in \RR\) and an operator
  \(B_{\lambda_{0}} \in B(H)\) such that:
  \[ \lambda_{0} \in \varrho(A_{n}) \ \forall n \in \NN, \qquad \lim_{n \to
  \infty} \| R(A_{n}, \lambda_{0}) - B_{\lambda_{0}} \|=0, \] 
  and satisfying
  $ \mathrm{Ker}(B_{\lambda_{0}}) = \{0\}. $ 
  Then there exists a unique selfadjoint operator \(A\) on \(H\) defined by:
  \[ D(A)= \mathrm{Rng}(B_{\lambda_{0}}),  \qquad \text{and} \qquad A =
  B_{\lambda_{0}}^{-1} x + \lambda_{0}x, \quad x \in D(A).\]
  The domain \(D(A)\) and the operator \(A\) do not depend on the choice of
  \(\lambda_{0}\). Moreover, \(A\) satisfies 
  $ B_{\lambda_{0}} = R(A, \lambda_{0})$.
\end{proposition}
\begin{proof}
  First, note that if \(x \in D(A)= \mathrm{Rng}(B_{\lambda_{0}})\), then the
  preimage \(B^{-1}_{\lambda_{0}} x\) is uniquely defined, since we assumed
  that \(\mathrm{Ker}(B_{\lambda_{0}}) = \{0\}\). It remains to check that
  \(A\) is a self-adjoint operator: for this we refer, for example, to
  \cite[Proposition 8.2]{Taylor2011PDEsI}. By construction we have that
  \(B_{\lambda_{0}} = R(A, \lambda_{0})\) and through the resolvent identity
  (for all \(\lambda \in \varrho(A)\)):
  \[ R(A, \lambda) = R(A, \lambda_{0}) +(\lambda - \lambda_{0})R(A,
  \lambda_{0}) R(A, \lambda),\]
  we see that the domain does not depend on the choice of \(\lambda_{0}\).
\end{proof}
At this point, we can describe the structure of the proof of
Theorem~~\ref{thm:semidiscrete-pam-approximation} as follows:
\begin{itemize}
  \item[i)] The crux of the argument is to show that for a fixed \(\lambda \in \RR\) the resolvents
    \(R(\mH_{n}, \lambda)\) converge:
    \[ \lim_{n \to \infty} R(\mH_{n}, \lambda) = B_{\lambda},\]
    for some bounded injective \(B_{\lambda}\).
  \item[ii)] The previous proposition then guarantees the existence of a selfadjoint operator
    $\mH$ such that \(B_{\lambda} = R(\mH, \lambda)\).
  \item[iii)] Finally, the convergence of eigenfunctions and eigenvalues follows from
    Corollary~\ref{cor:convergence-eigenfunctions}.
\end{itemize}
\begin{remark}\label{rem:on-the-domain}
  This argument does not require an
  explicit construction of the operator \(\mH\) or of its domain \(D(\mH)\).
  It will appear clearly from the proof that the limiting resolvent
  \(R(\mH, \lambda)\) coincides with the resolvent constructed in
  \cite{AllezChouk2015AndersonHamiltonian2D} (although the article treats only
  the case \(d=2\), a similar but simpler construction works also in
  \(d=1\)). In particular, the latter
  article explicitly describes the range of the resolvent (i.e. the domain of
  the operator \(\mH\)), as a space of \emph{strongly paracontrolled}
  distributions and it provides an explicit representation of \(\mH\) on this
  domain.
\end{remark}

\subsection{Proof of Theorem
\ref{thm:semidiscrete-pam-approximation}}\label{subsection:proof-of-pam-theorem}
 
The paracontrolled approach in \cite{AllezChouk2015AndersonHamiltonian2D} to
construct the Anderson Hamiltonian in $d=2$ follows the Ansatz that the
solution \(\psi\) to the resolvent equation
\[ ( \nu_{0} \Delta  + \xi - \lambda) \psi = \varphi.\]
for $ \varphi \in L^{2} $ is of the form \(\psi = \psi^{\prime} \para
X_{\lambda} + \psi^{\sharp}\), the previous being a paraproduct as defined in
Lemma~\ref{lem:paraproduct-estimates}, with \(X_{\lambda}\) solving \( (-
  \nu_{0} \Delta + \lambda)X_{\lambda} =
\xi\), and \(\psi^{\sharp} \in \mC^{1 + 2\kappa}\) (we will call a \(\psi\) of this
form \textit{paracontrolled}). This should be interpreted as a ``Taylor
expansion'' in terms of functionals of the noise, and the reason why the rest
term is expected to be of better regularity is encoded in the concept of
subcriticality, introduced in \cite{Hairer2014}.  Now, for paracontrolled
\(\psi\) the previously ill-defined product can be rewritten as \(\psi \xi =
(\psi^{\prime} \para X) \xi + \psi^{\sharp} \xi\).  While the last term is
now well-defined (recall that if $d=2$, $\xi \in \mC^{-1-\kappa}$), a commutator estimate (see Lemma~\ref{lem:commutator_resonan}) guarantees that
the resonant product can be approximated as \((
\psi^{\prime} \para X) \reso \xi \simeq \psi^{\prime}( X \reso
\xi)\). The latter resonant product \(X \reso \xi\) remains still
ill-defined in terms of regularity, but one can make sense of it through
some Gaussian computations (since \(X_{\lambda}\) and $\xi$ are both Gaussian fields), up to
renormalisation. By this we mean that the product lives in two levels of the
Wiener chaos. While the second chaos part turns out to be well-defined,
the zeroth chaos is diverging. 
Eventually, one can rigorously define a distribution \(X \diamond \xi\) that formally
can be written as \(X \reso \xi - \infty = X \reso \xi - \EE \big[
X \reso \xi \big]\), which lives in the second Wiener chaos and
explains the \(\infty\) appearing in the equation. This explains why in $d=2$
the Hamiltonian can be written as
$\nu_0 \Delta + \xi - \infty,$
where the latter \lqm $\infty$\rqm comes from the renormalisation.

In the cartoon we have just sketched, we hope to explain that theories for
singular stochastic PDEs have two critical ingredients. First, some stochastic computations
guarantee the existence of certain products of random distributions. Second,
given a realization of these distributions, a purely analytic argument, based
on regularity estimates and a Taylor-like expansion guarantees the
existence of a solution to the PDE.

In the present setting we concentrate on semidiscrete approximations of the
Anderson Hamiltonian, that is we will prove that \(\psi\) as above is the limit
\(\psi = \lim_{n \to \infty} \psi_{n}\), with\( (- \mA_{n} +
    \lambda)\psi_{n}= \Pi_{n}^{2}(\xi^{n}
- c_{n} 1_{\{d=2\}}) \Pi_{n}^{2}\psi_{n} - \varphi\). Following the previous explanation we will
first state some stochastic estimates and then pass to the
main analytic result. The next definition introduces the space in which we
will control the stochastic terms.
\newcommand{\bx}{\boldsymbol\xi}

\begin{definition}\label{def:norm-enhanced-noise}
  Consider \(d=2\) and fix any \(\kappa \in (0, \frac{1}{2}) \). For any \(n \in \NN\) we will call an enhanced noise
a vector of distributions
\[ \boldsymbol{\xi}_{n} = ( \xi^{n}, Y_{n}) \in \mS^\prime(\TT^2)
\times C([1, \infty) ; \mS^{\prime} (\TT^{2})),\]
where \( Y_{n} \) is a map
$ [1, \infty) \ni \lambda \mapsto Y_{n, \lambda} \in \mS^{\prime} (\TT^{2}). $ 
For \( \bx_{n} \) we introduce the following norm, with $ X_{n, \lambda} = (-\mA_{n} + \lambda)^{-1} \xi^{n}$:
\begin{align*}
  \opnorm{ \boldsymbol{\xi}_{n} }_{n, \kappa}: =  & \sup_{\zeta \in [0,1]
}\left\{  n^{-\zeta}   \|\xi^n\|_{\mC^{- (1-\zeta) -\frac \kappa 2}}\right\}   +
n^{-1}\|\xi^n\|_{L^\infty} + n^{-1- \kappa}\| \xi^n\|_{\mC^{\kappa}_{\frac{1}{2
\kappa}}} \\
 &+ \sup_{\lambda \geqslant 1} \left\{ n \|\mQ_n X_{n, \lambda}\|_{L^\infty} +
\lambda^{-\frac{\kappa}{4}}\| Y_{n, \lambda}
\|_{\mC^{-\frac \kappa 2}} \right\}. 
\end{align*}
\end{definition}
We can immediately bound some further quantities related to \(
\bx_{n} \).
\begin{lemma}\label{lem:deterministic-estimates-on-X-n-lambda}
For \( n \in \NN \) and \( \lambda \geqslant 1 \) consider an enhanced noise \( \bx_{n} \) as in
Definition~\ref{def:norm-enhanced-noise}. Then we can bound, for any \(
\kappa \in (0, \frac{1}{2} ), \delta \in [0, 1] \) and uniformly over \( n, \lambda \):
\begin{align*}
\sup_{\zeta \in [0,1]} \lambda^{\delta} n^{- \zeta} \big\{ \| \mP_n X_{n, \lambda}\|_{\mC^{-
(1-\zeta) +2(1 - \delta)-\frac \kappa 2}}
+ n^{2(1 - \delta)} \| \mQ_n X_{n, \lambda}\|_{\mC^{- (1-\zeta) - \frac \kappa 2}} \big\} \lesssim
\opnorm{ \bx_{n}}_{n, \kappa}.
\end{align*}
\end{lemma}
\begin{proof}
This is a consequence of the elliptic Schauder estimates of
Proposition~\ref{prop:elliptic_schauder_estimates}.
\end{proof}
Now, the following stochastic estimates hold true.
\begin{proposition}\label{prop:stochastic_estimates}
  Let $(\Omega, \mF, \PP)$ be a probability space supporting a
  sequence of random functions \(\xi^{n}\colon \TT^d \to \RR\) as in
  Assumption~\ref{ass:selection_coefficient}. In dimension
  \(d=2\), for $\lambda \geqslant 1$, define
  \[  X_{n,\lambda} = (-\mA_{n}+\lambda)^{-1}\xi^{n}, \qquad \xi^{n} \diamond \Pi_{n}^2
      X_{n, \lambda} = \xi^{n} \reso \Pi_{n}^2 X_{n, \lambda} - c_{n},\]
  where
  \begin{equation*}
    \begin{aligned}
      c_{n} = \sum_{k \in \ZZ^2}  \frac{\widehat{\chi}^2( n^{-1} k)
      \widehat{\chi}_Q(n^{-1} k)}{-\vt_n (k) + 1}, \ \ \text{ with } \ \
      c_n \simeq \log {n}.
    \end{aligned}
  \end{equation*}
  If \( d=1 \) one can bound for any \( \kappa \in (0, \frac{1}{2}) \):
  \[ \sup_{n \in \NN} \EE \big[ \sup_{\zeta \in [0, 1]} n^{-\frac \zeta
      2} \| \xi^{n} \|_{\mC^{-\frac 1 2 (1-\zeta )- \frac{\kappa}{2}}} + n^{-1} \|
  \xi^{n} \|_{L^{\infty}}\big] < \infty.\] 
  If \(d=2\) define the enhanced noise
$ \bx_{n} = (\xi^{n}, (\xi^{n} \reso \Pi_{n}^2 X_{n, \lambda} -
c_{n})_{\lambda \geqslant 1} ),$
taking values in the space of Definition~\ref{def:norm-enhanced-noise}. 
For any \(\kappa>0\) one can bound
  $ \sup_{n \in \NN} \EE \big[ \opnorm{ \boldsymbol\xi_n}_{n,
  \kappa} \big] < \infty$.  
  Moreover, for any fixed \( \kappa \in (0, \frac{1}{2}) \) there exists a probability space $(\overline{\Omega},
  \overline{\mF}, \overline{\PP})$, supporting space white noise $\xi$ on
  $\TT^d$, and a sequence of random functions $\overline{\xi}^n \colon \TT^d
  \to \RR$ such that $\xi^n = \overline{\xi}^n$ in distribution
  and such that for almost all \(\omega \in \overline{\Omega}\):
  \begin{align*}
    \overline{\xi}^{n} (\omega) \to \xi(\omega) \ \ \text{in} \ \ \mC^{- \frac d 2 - \kappa}.
  \end{align*}
  In dimension \(d=2\), for any \( \lambda \geqslant 1 \) there exists also a random distribution
  \( \xi \diamond X_{\lambda}\) such that:
  \begin{align*}
    \mP_n(-\mA_n +\lambda)^{-1}\overline{\xi}^n (\omega) & \to (-\Delta +\lambda)^{-1}
    \xi (\omega)  \ \ & \text{ in } \ \ \mC^{2 - \frac d 2 - \kappa}, \\
    \overline{\xi}^{n} \diamond \Pi_{n}^2 X_{n, \lambda} & \to \xi \diamond
    X_{\lambda} (\omega)\ \ & \text{ in } \ \ \mC^{-\kappa}.
  \end{align*}
  Finally, again in \( d=2 \) and for almost all \( \omega \in
  \overline{\Omega} \), one an bound
  $\sup_{n \in \NN} \opnorm{\bx_{n}(\omega)}_{n, \kappa} < \infty.$
\end{proposition}
The proof of this result is mostly technical, and for the sake of readability
deferred to Apppendix~\ref{app:stochastic-bouns}. In view of the previous result we will work under the
following assumption.
\begin{assumption}\label{assu:on-the-probability-space-PAM}
  Consider \( \kappa \in (0, \frac{1}{2}) \) fixed. Up to changing the probability space
  \((\Omega, \mF, \PP)\), we assume that for all \(\omega \in \Omega\) outside a null-set
  \(N\) the convergences in
  Proposition~\ref{prop:stochastic_estimates} hold true. If \(d=2\) in addition
  $\sup_{n \in \NN} \opnorm{ \boldsymbol\xi_{n} (\omega) }_{n, \kappa} <
  \infty$.
\end{assumption}
Having fixed the correct probability space and having explained our method, we are now in position to prove
Theorem~\ref{thm:semidiscrete-pam-approximation}. The next result proves that the operators \( \mH_{n} \) converge in resolvent
sense.
\begin{proposition}\label{prop:semidiscrete-resolvent-convergence}
  Under Assumption~\ref{assu:on-the-probability-space-PAM} fix \(\omega \in
  \Omega \setminus N\). Consider, for \(n \in \NN\), the bounded selfadjoint operators
  \[ \mH_{n}^{\omega} \colon L^2 \to L^2, \qquad \mH_{n}^{\omega} \psi =
  (\mA_n + \Pi_{n}^{2} (\xi^n{-}c_n) \Pi^2_{n}) \psi.\] 
  There exists a \( \overline{\lambda}(\omega) \in [1, \infty)\) such that 
  \( - \mH^{\omega}_{n} + \lambda(\omega) \) is invertible for all \(n \in
\NN\) and \( \lambda(\omega) \geqslant \overline{\lambda}(\omega) \), and
  there exists an operator \(B_{\lambda}(\omega) \in B(H)\) such that 
  \[ \lim_{n \to \infty} (- \mH_{n}^{\omega} + \lambda(\omega))^{-1} =
    B_{\lambda}(\omega) \ \ \text{in} \ \ B(L^{2}(\TT^{d})). \] 
\end{proposition}

\begin{proof}

  The strategy of the proof is a perturbation of the proof in
\cite{AllezChouk2015AndersonHamiltonian2D} and is based on a fixed point argument. 
In Step~1 we describe the space in which we can solve the resolvent equation
through a fixed point argument,
uniformly over \(n\) and \(\lambda\) large enough (throughout the proof the
realization \(\omega\) is fixed and omitted to keep the notation clean). 
The estimates that will allow us to apply Banach's fixed point theorem are
discussed in Steps~2 through 4. The
convergence as \(n \to \infty\) is established in Steps 5 and 6.
Throughout the proof the parameter \(\kappa \in (0, \frac{1}{2}) \) will be chosen small enough,
so that all computations hold.

\textit{Step 1.} 
  Fix $p \in [1, \infty]$ as well as \(\varphi\in \mC^{-1+2\kappa}_{p}\). In dimension $d=1$, solving the resolvent equation
  \( (-\mH_{n} + \lambda)\psi = \varphi \) 
  is equivalent to solving (with \( c_{n} = 0 \)) the fixed point problem
  \begin{equation}\label{eqn:fixed_point_reolvent_equation}
    \psi = M_{\varphi,\lambda} (\psi) := ( -\mA_{n} + \lambda)^{- 1}
    [\Pi_{n}^{2} [\xi^{n} - c_{n}]  \Pi_{n}^2 \psi  + \varphi ].
  \end{equation} 
  In dimension $d=2$ we will not prove directly that \(M_{\varphi, \lambda}\) is
  a contraction (while in \(d=1\) this is possible: the arguments that follow are
  then superfluous and Proposition~\ref{prop:elliptic_schauder_estimates} allows to
find a fixed point \(\psi\)). Instead, to find the fixed point we look for a
paracontrolled solution. Consider a space \(\mD_n^{\lambda} \subseteq
\mS^{\prime}(\TT^d) \times \mS^{\prime}(\TT^d)\) which consists of pairs
     \( (\psi^{\prime}, \psi^{\sharp})\) and is characterized by the norm
  \begin{align*}
     \| (\psi^{\prime},\psi^{\sharp}) \|_{\mD_n^{\lambda}} :=  \|\psi^{\prime} \|_{\mC^{1- \kappa}_{p}}+\| \mP_n \psi^{\sharp}
    \|_{\mC^{1 + \kappa}_{p}} + n^{2-\kappa}\| \mQ_n \psi^\sharp \|_{\mC^{-1+2\kappa}_{p}}, 
  \end{align*}
  where we used the operators $\mP_n, \mQ_n$ as in
Definition~\ref{def:cut_off_operators}. The norm does not depend on \( \lambda
\), but to every pair $(\psi^\prime, \psi^\sharp) \in \mD^{\lambda}_{n}$ we associate a function $\psi$ by
  \[ \psi = \Pi_{n}^{2} \big\{ \psi^{\prime} \para [ (- \mA_{n} +
    \lambda)^{-1} \xi^{n} ] \big\} + \psi^{\sharp}.
  \]
  With an abuse of notation, we identify the pair \((
  \psi^{\prime},\psi^{\sharp})\) with the function \(\psi\) and write $\| \psi
\|_{\mD_n^{\lambda}} = \|(\psi^\prime, \psi^\sharp)\|_{\mD_n^{\lambda}}$.
Define the map \( \overline{M}_{\varphi,\lambda} \colon \mD_n^{\lambda } \to L^p\) as
  \[ \overline{M}_{\varphi,\lambda}(\psi) := ( -\mA_{n} + \lambda)^{- 1} [
  \Pi_{n}^{2} \xi^{n}  \Pi_{n}^2 \psi  -c_{n} \Pi_{n}^{2} \psi^{\prime} +\varphi].\] 
  The map $ \overline{M}_{\varphi,\lambda} $ can be extended to a map from
\(\mD_n^{\lambda}\) into itself by defining:
  \begin{align*} 
    \mM_{\varphi,\lambda} ( \psi)  & = ( 
    M^{\prime}_{\varphi,\lambda}(\psi), \ M^{\sharp}_{\varphi,\lambda}(\psi))\\
    & := (\Pi_{n}^2  \psi, \ \overline{M}_{\varphi,\lambda}(\psi) -
      \Pi_{n}^{2} \{ ( \Pi_{n}^2
    \psi) \para [  (- \mA_{n} + 1)^{-1} \xi^{n}] \} ) \in \mD_{n}^{\lambda}.
  \end{align*}
Any fixed point of \(\mM_{\varphi, \lambda}\) is also a fixed point for $\overline{M}_{\varphi, \lambda}$ and since the
  fixed point  satisfies $\psi^{\prime} = \Pi_{n}^2 \psi,$
  it solves also the fixed point equation
  \eqref{eqn:fixed_point_reolvent_equation} for $M_{\varphi, \lambda}$.
Similarly, if \( \psi \in L^{p} \) solves
Equation~\eqref{eqn:fixed_point_reolvent_equation}, then \( \psi \in
\mD^{\lambda}_{n} \) (for fixed \( n \in \NN \) the embedding \( L^{p}
\subseteq \mD^{\lambda}_{n} \) is continuous) and \( \psi \) is a fixed point for \(
\mM_{\varphi, \lambda}\). We
conclude that solutions \( \psi \in L^{p} \) to $(- \mH_{n}+ \lambda) \psi = \varphi$ are equivalent to fixed
points of $\mM_{\varphi, \lambda}$.
We will show that for \(\lambda\) sufficiently large \( \mM_{\varphi, \lambda} \) admits a unique
  fixed point for all \( \varphi \in \mC_{p}^{-1 + 2 \kappa} \).  

  Throughout the proof we will repeatedly make use of the elliptic
 Schauder estimates of Proposition~\ref{prop:elliptic_schauder_estimates}, the
regularization properties of \( \Pi_{n} \) of
Corollary~\ref{cor:regularity_gain_convolution_with_characteristic_functions},
the estimates on \( X_{n , \lambda} \) of
Lemma~\ref{lem:deterministic-estimates-on-X-n-lambda} \emph{which crucially allow us to gain
powers of $ \lambda $ and $ n $} and the paraproduct estimates of
Lemma~\ref{lem:paraproduct-estimates}, without stating them explicitly every
time.

  \textit{Step 2.} 
  Our aim is to control (paying particular attention to the dependence on
\(\lambda\) and the uniformity over \( n \)) the quantity: 
    \begin{align*}
	   \| \mM_{\varphi, \lambda} (\psi) \|_{\mD^{\lambda}_{n}}  = \|\Pi_{n}^2  \psi \|_{\mC^{1- \kappa}_{p}}+ \| \mP_n \ M^{\sharp}_{\varphi,\lambda}(\psi)
	    \|_{\mC^{1 + \kappa}_{p}} +\| \mQ_n \
	    M^{\sharp}_{\varphi,\lambda}(\psi) \|_{\mC^{-1+2\kappa}_{p}},
    \end{align*}
    in terms on \(\| \psi \|_{\mD_{n}^{\lambda}}\) and \( \| \varphi
\|_{\mC^{-1 + 2 \kappa}_{p}} \).
As for the first term, \( \| \Pi_{n}^{2} \psi
\|_{\mC^{1- \kappa}_{p}} \), we observe that
\begin{equation}\label{eqn:for-proof-resolvent-Pi2n-psi-estimate}
\begin{aligned}
\| \Pi_{n}^{2} \psi \|_{\mC^{1- \kappa}_{p}}  & = \big\| \Pi_{n}^{4} \big\{
\psi^{\prime} \para  X_{n , \lambda} \big\} + \Pi_{n}^{2} \psi^{\sharp} \big\|_{\mC^{1- \kappa}_{p}} \\
&  \lesssim \| \psi^{\prime} \|_{\mC^{1 - \kappa}_{p}} \Big( \| \mP_{n}
X_{n, \lambda} \|_{\mC^{1 - \kappa }} + n^{2}\| \mQ_{n} X_{n, \lambda}
\|_{\mC^{-1 - \kappa }} \Big) \\
& \qquad + \| \mP_{n} \psi^{\sharp}  \|_{\mC^{1 + \kappa}_{p}} + n^{2 - 3\kappa}\|
\mQ_{n} \psi^{\sharp}\|_{\mC^{-1 + 2 \kappa}_{p}} \\
& \lesssim \lambda^{-\frac{\kappa}{4}} \| \psi \|_{\mD^{\lambda}_{n}} (1 +
\opnorm{ \bx}_{n, \kappa}) + \| \mP_{n} \psi^{\sharp}  \|_{\mC^{1 + \kappa}_{p}} + n^{2 - \kappa}\|
\mQ_{n} \psi^{\sharp}\|_{\mC^{-1 + 2 \kappa}_{p}}.
\end{aligned}
\end{equation}
 To tackle the norms involving $M^\sharp,$ first rewrite
 \[ M^{\sharp}_{\varphi, \lambda}(\psi) = M^{\sharp, 1}_{\varphi}(\psi) + M^{\sharp,
  2}_{\varphi}(\psi),\]
  where to clean the notation we have omitted the dependence on $ \lambda $ and with
  \begin{equation*}
  \begin{aligned}
    M^{\sharp, 1}_{\varphi}(\psi)  = & (- \mA_{n}+ \lambda)^{-1} \Big\{ \varphi + \Pi_{n}^{2}[
      \xi^{n} \reso \Pi_{n}^2 \psi^{\sharp} ] +  \Pi_{n}^{2}\big\{ \xi^{n} \reso
    [\Pi_{n}^2 (  \psi^{\prime} \para X_{n, \lambda})] - c_{n} \psi^{\prime} \big\} +
  \Pi_{n}^{2}\big\{ \xi^{n} \para \Pi_{n}^2 \psi \big\} \Big\} 
  \end{aligned}
  \end{equation*}
  and $M^{\sharp, 2}_{\varphi}(\psi)  =  \Pi_{n}^{2}  C_{n, \lambda} ( \Pi^2_n
  \psi, \xi^n)$,
  with $ C_{n, \lambda} (\Pi^2_n \psi, \xi^n)$ the commutator
  \[ 
    C_{n, \lambda} (\Pi^2_n \psi, \xi^n) = (- \mA_{n} + \lambda)^{-1}
    [ (\Pi_{n}^2 \psi) \para \xi^{n} ] - [ (\Pi_{n}^2
      \psi) \para (- \mA_{n} + \lambda )^{-1}(\xi^{n})]. 
  \] 
   For clarity we divide the estimates for the two terms \(M^{\sharp,
  1}_{\varphi}, M^{\sharp, 2}_{\varphi}\) in two distinct steps.

  \textit{Step 3: Estimates for \(M^{\sharp, 1}_{\varphi}\).} Combining the Schauder estimates with the smoothing properties of
  \(\Pi_{n}\) and the paraproduct estimates one finds that
  \begin{align*}
    \lambda^{\frac \kappa 2} \big( &\| \mP_n M^{\sharp, 1}_{\varphi}(\psi) 
      \|_{\mC^{1+ \kappa}_{p}}  + n^{2-\kappa}  \| \mQ_n M^{\sharp, 1}_{\varphi}(\psi)
  \|_{\mC^{-1+ 2\kappa}_{p}}\big) \\
  & \lesssim  \| \varphi  \|_{\mC^{-1 + 2 \kappa}_{p}} + \| \Pi_n^2 \psi^{\sharp} \|_{\mC^{1 +
    \kappa}_{p}} \| \xi^{n} \|_{\mC^{-1 - \frac \kappa 2}} + \| \xi^{n} \reso
    [\Pi_{n}^2 (  \psi^{\prime} \para X_{n, \lambda})] - c_{n}
    \psi^{\prime}\|_{\mC^{-1+ 2 \kappa}_{p}}.
  \end{align*}
  To treat
    $\| \xi^{n} \reso [\Pi_{n}^2 (  \psi^{\prime} \para X_{n, \lambda}) - c_{n}
    \psi^{\prime}]\|_{\mC^{-1+ 2 \kappa}_{p}}$,
 we introduce (cf. Definition~\ref{def:all_commutators_at_once}) the commutators
  $$C^\Pi_n(f, g)= \Pi_n^2 (f \para g) - f\para \Pi_n^2 g, \qquad C^{\reso}(f,g,h) = f \reso (g \para h) - g(f \reso h).$$ 
  Then the previous resonant product can be split into:
  \begin{equation}\label{eqn:proof-semidiscrete-AM-support-commutators}
  \begin{aligned}
	  \| \xi^{n} \reso [\Pi_{n}^2 (  \psi^{\prime} &\para X_{n, \lambda}) - c_{n} \psi^{\prime}]\|_{\mC^{-1+ 2
    \kappa}_{p}} \leq  \| \xi^n \reso C^\Pi_n(\psi^\prime, X_{n,
    \lambda})\|_{\mC^{-1+2\kappa}_{p}} \\ 
    & + \|C^{\reso}(\xi^n, \psi^\prime,
    \Pi_n^2 X_{n, \lambda})\|_{\mC^{-1+ 2\kappa}_{p}} 
    + \| \psi^{\prime} (\xi^{n} \reso \Pi^{ 2}_{n} X_{n, \lambda} -
    c_{n}) \|_{\mC^{-1+2 \kappa}_{p}}.
  \end{aligned}
  \end{equation}
  Starting with the first term, by Lemma~\ref{lem:commutator_with_averaging}
  \begin{align*}
  \| \xi^n \reso C^\Pi_n(\psi^\prime, & X_{n, \lambda}) 
\|_{\mC^{-1+2\kappa}_{p}} \\
   &\lesssim \|\xi^n\|_{\mC^{-1-\frac \kappa 2}} \|\mP_n C^\Pi_n(\psi^\prime,
   X_{n, \lambda})\|_{\mC^{1+\kappa}_{p}} + \|\xi^n\|_{\mC^{-1+ \kappa }}\|\mQ_n
   C^\Pi_n(\psi^\prime,\mP_n X_{n, \lambda})\|_{\mC^{1-\frac \kappa 2}_p} \\
   & \ \ \ \ + \|\xi^n \reso \mQ_n C^\Pi_n(\psi^\prime, \mQ_n X_{n, \lambda})\|_{\mC^{-1+2\kappa}_{p}}\\
   & \lesssim \|\psi^\prime\|_{\mC^{1-\kappa}_p} \opnorm{ \boldsymbol\xi_n
   }_{n, \kappa}^2 + \|\xi^n \reso \mQ_n C^\Pi_n(\psi^\prime, \mQ_n X_{n,
\lambda})\|_{\mC^{-1+2\kappa}_{p}}.
  \end{align*}
  The last quantity requires a bit of attention, since at first sight none of
  the two terms involved in the product has positive regularity: while the
  commutator guarantees us powers of $n$, it does not guarantee regularization
  on small scales. For this we need the estimate of $\xi^n$ in spaces of
  positive regularity. Since $\xi^n$ is constant on boxes this is not possible
  in the $L^\infty$ scale of spaces, so we have to introduce an additional
  integrability parameter. For this we assume that $\kappa$ is small enough so
  that $\frac 1 r = \frac 1 p + 2 \kappa \leq 1$ and $-1 + 2 \kappa \leq -
  \frac{7 \kappa}{2}$. Then:
  \begin{align*}
   \|\xi^n \reso \mQ_n &C^\Pi_n(\psi^\prime, \mQ_n
   X_{n})\|_{\mC^{-1+2\kappa}_{p}}  \leq \|\xi^n \reso \mQ_n
   C^\Pi_n(\psi^\prime, \mQ_n X_{n})\|_{\mC^{-\frac{7 \kappa}{2}}_{p}} \\
   & \lesssim \|\xi^n \reso \mQ_n C^\Pi_n(\psi^\prime, \mQ_n X_{n})\|_{\mC^{\frac{\kappa}{2}}_{r}}
    \lesssim \|\xi^n\|_{\mC^{\kappa}_{\frac{1}{2 \kappa}}} \|\mQ_n
   C^\Pi_n(\psi^\prime, \mQ_n X_{n})\|_{\mC^{-\frac{\kappa}{2}}_p}
  \end{align*}
  where in the second step we used Besov embedding and in the last step we used
  the resonant product estimate with arbitrary integrability parameters from
  Lemma~\ref{lem:paraproduct-estimates}. Overall:
\begin{align*}
	\|\xi^n \reso \mQ_n C^\Pi_n(\psi^\prime, \mQ_n
	X_{n})\|_{\mC^{-1+2\kappa}_{p}} & \lesssim n^{1 + \kappa} \opnorm{
	\boldsymbol\xi_n } n^{-(1 - 2 \kappa)}
	\|\psi^\prime\|_{\mC^{1-\kappa}_p} \| \mQ_n X_{n}\|_{\mC^{-\frac{\kappa}{2}}_p} \\
	& \lesssim \|\psi^\prime\|_{\mC^{1-\kappa}_p}  \opnorm{ \boldsymbol\xi_n }^2.
\end{align*}  
As for the second term in \eqref{eqn:proof-semidiscrete-AM-support-commutators}, by Lemma~\ref{lem:commutator_resonan} 
\begin{align*}
  \|C^{\reso}(\xi^n, \psi^\prime, \Pi_n^2 X_{n})\|_{\mC^{-1+ 2\kappa}_{p}}
  & \leq \|C^{\reso}(\xi^n, \psi^\prime, \Pi_n^2 X_{n})\|_{\mC^{-\kappa}_{p}} \\
  & \lesssim \| \xi^{n} \|_{\mC^{-1- \frac \kappa 2}} \| \psi^{\prime}
  \|_{\mC^{1- \kappa}_{p}} \| \Pi^{2}_{n} X_{n}
  \|_{\mC^{1- \frac \kappa 2}} \lesssim \| \psi^{\prime} \|_{\mC^{1 -
  \kappa}_{p}} \opnorm{ \boldsymbol\xi_{n}}^{2}_{n, \kappa}.
\end{align*}
Here we estimated, via Lemma~\ref{lem:deterministic-estimates-on-X-n-lambda}:
\begin{align*}
 \| \Pi_{n}^{2} X_{n , \lambda} \|_{\mC^{1 - \frac{\kappa}{2}}} & \leqslant \|
\Pi_{n}^{2} \mP_{n} X_{n, \lambda}\|_{\mC^{1 - \frac{\kappa}{2}}} + \|
\Pi_{n}^{2} \mQ_{n} X_{n, \lambda} \|_{\mC^{1 - \frac{\kappa}{2}}} \\
&   \lesssim \opnorm{\bx_{n}}_{n, \kappa} + n \| X_{n, \lambda} \|_{\mC^{-
\frac{\kappa}{2}}}  \lesssim  \opnorm{\bx_{n}}_{n, \kappa}.
\end{align*}
Similarly for the last term in
\eqref{eqn:proof-semidiscrete-AM-support-commutators}. Here we recall that in
the norm \( \opnorm{ \bx_{n}}_{n, \kappa} \) the term \( Y_{n, \lambda} =
\xi^{n} \reso \Pi_{n}^{2} X_{n, \lambda} - c_{n} \) is allowed to mildly
explode for \( \lambda \to \infty. \) We obtain:
\begin{equation}\label{eqn:proof-resolven-estimate-resonant-product-with-blowup}
\begin{aligned}
  \| \psi^{\prime} (\xi^{n} \reso \Pi^{ 2}_{n} X_{n , \lambda} - c_{n})
  \|_{\mC^{-1+2 \kappa}_{p}} & \lesssim \| \psi^{\prime} \|_{\mC^{1-
  \kappa}_{p}} \| \xi^{n} \reso \Pi^{2}_{n} X_{n} -
  c_{n} \|_{\mC^{-1 + 2 \kappa}} \\
& \lesssim \lambda^{\frac{\kappa}{4}} \| \psi^{\prime} \|_{\mC^{1- \kappa}_{p}} \opnorm{ \boldsymbol\xi_{n} }_{n, \kappa}.
\end{aligned}
\end{equation}

\textit{Step 4: Estimates for \(M^{\sharp, 2}_{\varphi}\).} Here we apply the
commutator estimate for $C_{n, \lambda}(\Pi^2_n \psi, \xi^n)$ from
Lemma~\ref{lem:commutator_resolvent}. We start by
estimating the large scales: 
\begin{align*}
  \| \mP_{n} M^{\sharp, 2}_{\varphi}(\psi) \|_{\mC^{1 + \kappa}_{p}} & = \|
  \Pi_{n}^{2} \mP_{n} C_{n , \lambda}(\Pi_{n}^{2} \psi, \xi^{n})
  \|_{\mC^{1+ \kappa}_{p}} \lesssim \| \mP_{n} C_{n , \lambda}(\Pi_{n}^{2} \psi, \xi^{n})
  \|_{\mC^{1+ \kappa}_{p}} \\
& \lesssim \lambda^{- \frac{\kappa}{2}} \| \Pi_{n}^{2} \psi \|_{\mC^{1 - \kappa}_{p}} \| \xi^{n}
\|_{\mC^{-1 - \frac{\kappa}{2}}} \lesssim \lambda^{- \frac{\kappa}{2}}  \| \psi \|_{\mD^{\lambda}_{n}}(1 + \opnorm{\xi_{n}}_{n,
\kappa})^{2},
\end{align*}
where we used that, provided \( \kappa \) is sufficiently small, \( (1 - \kappa) + (-1 - \kappa /2) + 2(1 - \kappa/2) > 1 + \kappa  \)
together with the estimate \eqref{eqn:for-proof-resolvent-Pi2n-psi-estimate}
for \( \Pi_{n}^{2} \psi \).
On small scales we find:
\begin{align*}
\lambda^{\frac{\kappa}{2}} n^{2 - \kappa} \| \mQ_{n} M^{ \sharp,
2}_{\varphi}(\psi) \|_{\mC^{-1 + 2\kappa}_{p}} & \lesssim \lambda^{\frac{\kappa}{2}} n^{2 - \kappa}  \| \mQ_{n} C_{n, \lambda}(\Pi_{n}^{2} \psi, \xi^{n})
\|_{\mC^{-1 + 2\kappa}_{p}} \\
& \lesssim  n^{-1} \big( \lambda^{\frac{\kappa}{2}
} n^{1 +2(1 - \kappa/2)} \| \mQ_{n} C_{n, \lambda}(\Pi_{n}^{2} \psi, \xi^{n})
\|_{\mC^{-1 + 2\kappa}_{p}} \big) \\
& \lesssim \| \Pi_{n}^{2} \psi
\|_{\mC^{1 - \kappa}_{p}} ( n^{-1} \| \xi^{n} \|_{\mC^{-1 + 3 \kappa}}) \lesssim \| \psi \|_{\mD^{\lambda}_{n}}
(1 + \opnorm{\bx_{n}}_{n, \kappa})^{2},
\end{align*}
where we once again used the estimates on \( \Pi_{n}^{2} \psi \)
from~\eqref{eqn:for-proof-resolvent-Pi2n-psi-estimate}.

\textit{Step 5: Collecting the estimates.}
The estimates of step 2 guarantee that there exists an increasing map
$ \mf{c} \colon [0, \infty) \to [1, \infty) $ 
such that
\begin{equation}\label{eqn:proof-resolvent-final-1}
\| M^{\prime}_{\varphi}(\psi) \|_{\mC^{1- \kappa}_{p}} \leqslant \mf{c}(
\opnorm{\bx_{n}}_{n, \kappa} ) \big( 
\lambda^{- \frac{\kappa}{2}} \| \psi^{\prime} \|_{\mC^{1 - \kappa}_{p}} + 
\| \mP_{n} \psi^{\sharp} \|_{\mC^{1+ \kappa}_{p}} + \| \mQ_{n} \psi^{\sharp}
\|_{\mC^{-1 + 2 \kappa}_{p}} \big).
\end{equation}
In addition, estimates of steps \( 3 \) and \( 4 \) guarantee that (up to
choosing a larger \( \mf{c} \)):
\begin{equation}\label{eqn:proof-resolvent-final-2}
\| \mP_{n} M^{\sharp}_{\varphi}(\psi) \|_{\mC^{1 + \kappa}_{p}} + n^{2 -
\kappa}\| \mQ_{n} M^{\sharp}_{\varphi}(\psi) \|_{\mC^{-1 + 2
\kappa}_{p}} \leqslant \lambda^{- \frac{\kappa}{4}}  \mf{c}(\opnorm{ \bx_{n}}_{n,
\kappa}) \big( \| \varphi \|_{\mC^{-1 + 2 \kappa}_{p}} +  \| \psi
\|_{\mD^{\lambda}_{n}} \big).
\end{equation}
Observe that the factor \( \lambda^{-\frac{\kappa}{4}} \), instead of \(
\lambda^{- \frac{\kappa}{2}}\), is not a typo: it follows from
\eqref{eqn:proof-resolven-estimate-resonant-product-with-blowup}, where we pay
a factor \( \lambda^{\frac{\kappa}{4}} \) to control the product \(
\xi^{n} \reso \Pi_{n}^{2} X_{n, \lambda} - c_{n} \).
Combined with the linearity of the map \( \mM_{\varphi} \) we find that:
\begin{align*}
\| \mM_{\varphi} (\psi)\|_{\mD_n} & \leq \mf{c} (\opnorm{ \bx_{n}}_{n,
\kappa})\Big[
\|\varphi\|_{\mC^{-1+2\kappa}_{p}} + \| \psi\|_{\mD_{n}}  \Big]\\
\big\| \big[ \mM_{\varphi} (\psi) - \mM_{\varphi} ( \tilde{\psi}) \big]^2
\big\|_{\mD_n} & \leq \mf{c}^{2}(\opnorm{ \bx_{n}}_{n,
\kappa})\Big[ \lambda^{-\frac \kappa 4} \| \psi
-\tilde{\psi}\|_{\mD_{n}} \Big].
\end{align*}
Note that we take the second power of the map in the last estimate, because in
\eqref{eqn:proof-resolvent-final-1} we do not have a small
factor \( \lambda^{- \frac{\kappa}{4}} \)  in front of the rest term with
$\psi^\sharp$. \\
In particular, we finally can conclude that there
exists a $\bar{\lambda}=\bar{\lambda}(\sup_n \opnorm{\bx_{n}}_{n, \kappa})$
(so it is independent of \( n \)) such that for $\lambda > \bar{\lambda}$ the map $\mM_{\varphi}$ admits a unique
fixed point, which we  denote by $\mH_{n, \lambda}^{-1} \varphi.$ Moreover,
by the Banach fixed point theorem
\begin{equation}\label{eqn:bound-resolvent-paracontrolled-norm}
\| \mH_{n, \lambda}^{-1} \varphi \|_{\mD_{n}} \lesssim \|
\mM_{\varphi}^2(0)\|_{\mD_{n}} \lesssim \mf{c}^{2}(\opnorm{\bx_{n}}_{n, \kappa}) \|
\varphi \|_{\mC^{-1+2\kappa}_{p}},
\end{equation}
implying that $\mH^{-1}_{n, \lambda} \in B(\mC^{-1+2\kappa}_{p},
\mD_{n}^{\lambda})$,
with the norm bounded uniformly in $n$. Similar, but less involved
calculations lead to a construction of the resolvent $\mH_{\lambda}^{-1} = (\mH
- \lambda)^{-1}$ in the continuum for $\lambda \geq \bar{\lambda}$ (in the continuous case no division of scales is required). The resolvent is then a
bounded operator $\mH_\lambda^{-1} \in B(\mC^{-1+2\kappa}_{p},
\mD^{\lambda})$, where
the latter is the Banach space defined by the norm (for $\psi = \psi^\prime
\para (-\Delta + \lambda)^{-1} \xi + \psi^\sharp$):
\begin{align*}
\| \psi \|_{\mD^{\lambda}} = \|\psi^\prime\|_{\mC^{1-\kappa}_{p}} + \|\psi^\sharp\|_{\mC^{1+\kappa}_{p}}.
\end{align*}
By linearity and computations on the line of those in the previous steps one
can then show that:
\begin{equation}\label{eqn:convergence_operators_in_paracontrolled_norm}
\begin{aligned}
	\lim_{n \to \infty} \sup_{\|\varphi\|_{\mC^{-1+2\kappa}_{p}}\leq 1}\Big\|
	( \mH_{n, \lambda}^{-1} \varphi)^\prime - ( \mH_\lambda^{-1} \varphi
	)^\prime \Big\|_{\mC^{1-\kappa}_{p}} + \Big\| \mP_{n}( \mH_{n, \lambda}^{-1}
	\varphi)^\sharp - ( \mH_\lambda^{-1} \varphi )^\sharp
	\Big\|_{\mC^{1+\kappa}_{p}} =0.
\end{aligned}
\end{equation}
Since $\mH_{n, \lambda}^{-1} \in B(\mC^{-1+2\kappa}_{p}, \mD_n^{\lambda})$, to
prove convergence of the resolvents in \( B(L^{2}, L^{2}) \) it would
be sufficient to show, in the particular case \(p=2\), that
$\mD_{n}^{\lambda} \hookrightarrow L^p$, in the sense that
\(\| \psi \|_{L^{p}} \lesssim \| \psi \|_{\mD_{n}^{\lambda}}\). Unfortunately, this is
not the case, because a priori \(\mQ_{n} \psi^{\sharp} \in \mC_{p}^{-1 + 2 \kappa}\).
So we need a better control on the regularity of $\psi^\sharp$, which we will
obtain by using that $\varphi \in L^2$.

\textit{Step 6: \( L^{2} \) estimates.} Let us fix \( p=2 \). We
want to improve our previous bound by showing that if \( \varphi \in
L^{2} \), then for every \( \psi \in \mD^{\lambda}_{n} \):
\begin{equation}\label{eqn:proof-resolvent-L2-small-scale}
\| \mQ_{n} M^{\sharp}_{\varphi}(\psi) \|_{L^{2}} \lesssim n^{- \kappa}
\mf{c}(\opnorm{\bx_{n}}_{n, \kappa}) (\| \varphi \|_{L^{2}} +\| \psi
\|_{\mD^{\lambda}_{n}}).
\end{equation}
Let us start with estimating by Plancherel (using the same notation as in
Section~\ref{sec:schauder_estimates}):
\begin{align*}
\| (- \mA_{n} + \lambda)^{-1} \mQ_{n} \varphi \|_{L^{2}}^{2} & \simeq \sum_{k \in \ZZ^{d}}
\bigg\vert \frac{ (1 - \daleth)(n^{-1} k)}{\lambda + n^{2}(1 -\hat{\chi}^{4}(n^{-1}
k))} \bigg\vert^{2} |\hat{\varphi}(k)|^{2} \\
& \lesssim \frac{1}{n^{4}} \sum_{k \in \ZZ^{d}} | (1 - \daleth)(n^{-1} k) \hat{\varphi}(k)
|^{2} \lesssim \frac{1}{n^{4}}  \| \mQ_{n} \varphi \|_{L^{2}}^{2},
\end{align*}
where we used that \( \hat{\chi}^{4}(n^{-1} k) < 1 \forall k \neq 0\), together
with the support properties of \( (1 - \daleth)(n^{-1} k) \). Hence we conclude
that
\begin{align*}
\| \mQ_{n} M^{\sharp}_{\varphi}(\psi) \|_{L^{2}} \lesssim n^{-2}\|
\mQ_{n}\varphi \|_{L^{2}} + \| \mQ_{n} \widetilde{M}^{\sharp, 1}_{\varphi}(\psi)  \|_{L^{2}} +
\| \mQ_{n} M^{\sharp, 2}_{\varphi}(\psi)  \|_{L^{2}},
\end{align*}
with \( M^{\sharp,2}_{\varphi}(\psi)  \) as in step \( 2 \) and 
\begin{align*}
\widetilde{M}^{\sharp, 1}_{\varphi} (\psi) = (- \mA_{n}+ \lambda)^{-1}
\Pi_{n}^{2} \Big\{ \xi^{n} \reso \Pi_{n}^2 \psi^{\sharp}  +   \xi^{n} \reso
    [\Pi_{n}^2 (  \psi^{\prime} \para X_{n, \lambda})] - c_{n} \psi^{\prime}  +
   \xi^{n} \para \Pi_{n}^2 \psi  \Big\}. 
\end{align*}
The smoothing effect of \( \Pi_{n}^{2} \) and the elliptic Schauder estimates guarantee that
\begin{align*}
n^{\frac{\kappa}{2}} \| \mQ_{n} &\widetilde{M}^{\sharp, 1}_{\varphi}(\psi)
\|_{\mC^{1 + \frac{\kappa}{2}}_{2}} \\
& = n^{\frac{\kappa}{2}}\big\| \Pi_{n}^{2}\mQ_{n}(- \mA_{n} + \lambda)^{-1} \big(\xi^{n} \reso \Pi_{n}^2 \psi^{\sharp}  +   \xi^{n} \reso
    [\Pi_{n}^2 (  \psi^{\prime} \para X_{n, \lambda})] - c_{n} \psi^{\prime}  +
   \xi^{n} \para \Pi_{n}^2 \psi \big) \big\|_{\mC^{1 + \frac{\kappa}{2}}_{2}} \\
& \lesssim n^{2 - \kappa}\big\| \mQ_{n}(- \mA_{n} + \lambda)^{-1} \big(\xi^{n} \reso \Pi_{n}^2 \psi^{\sharp}  +   \xi^{n} \reso
    [\Pi_{n}^2 (  \psi^{\prime} \para X_{n, \lambda})] - c_{n} \psi^{\prime}  +
   \xi^{n} \para \Pi_{n}^2 \psi \big) \big\|_{\mC^{-1 +2
\kappa}_{2}} \\
& \lesssim \big\|\xi^{n} \reso \Pi_{n}^2 \psi^{\sharp}  +   \xi^{n} \reso
    [\Pi_{n}^2 (  \psi^{\prime} \para X_{n, \lambda})] - c_{n} \psi^{\prime}  +
   \xi^{n} \para \Pi_{n}^2 \psi \big\|_{\mC^{-1 +2 \kappa}_{2}}.
\end{align*}
Now we can follow verbatim the estimates of step \( 3 \) to obtain, up to
slightly increasing \( \mf{c} \):
\begin{equation}\label{eqn:proof-proposition-bound-M-tilde-sharp-1}
n^{\kappa}\| \mQ_{n} \widetilde{M}^{\sharp, 1}_{\varphi}(\psi)
\|_{\mC^{1 + \frac{\kappa}{2}}_{2}} \leqslant
\mf{c}(\opnorm{\bx_{n}}_{n, \kappa}) \| \psi \|_{\mD_{n}^{\lambda}}.
\end{equation}
Similarly for \( M^{\sharp, 2}_{\varphi}(\psi) \), where we find:
\begin{equation}\label{eqn:proof-proposition-bound-M-tilde-sharp-2}
\begin{aligned}
n^{\frac{\kappa}{2}}\| \mQ_{n} M^{\sharp, 2}_{\varphi}(\psi)
\|_{\mC^{1 + \frac{\kappa}{2}}_{2}} &= n^{\frac{\kappa}{2}} \| \mQ_{n}
\Pi_{n}^{2} C_{n, \lambda}( \Pi_{n}^{2} \psi, \xi^{n}) \|_{\mC^{1 +
\frac{\kappa}{2}}_{2}} \\
 & \lesssim n^{2 - \kappa} \| \mQ_{n} C_{n, \lambda}( \Pi_{n}^{2} \psi, \xi^{n}) \|_{\mC^{-1 + 2
\kappa}_{2}} \\
& \leqslant \mf{c}(\opnorm{\bx_{n}}_{n, \kappa}) \| \psi
\|_{\mD_{n}^{\lambda}},
\end{aligned}
\end{equation}
where in the last step we followed verbatim the calculations in step \( 4
\). In particular, we have concluded the proof of~\eqref{eqn:proof-resolvent-L2-small-scale}.
The bound~\eqref{eqn:proof-resolvent-L2-small-scale} allows us in particular to
conclude that
\[ \lim_{n \to \infty} \sup_{ \| \varphi \|_{L^{2}} \leqslant 1} \| \mQ_{n} \mH_{n, \lambda}^{-1}
\varphi \|_{L^{2}} = 0. \]
Together with~\eqref{eqn:convergence_operators_in_paracontrolled_norm} we
conclude that
  \begin{align*}
    \lim_{n \to \infty} \sup_{\| \varphi \|_{L^{2}} \leqslant 1} \| &\mH_{n ,
\lambda}^{-1} \varphi  - \mH_{\lambda}^{-1} \varphi \|_{L^{2}} \\
 & \leq   \lim_{n \to \infty} \sup_{\| \varphi \|_{L^{2}} \leqslant 1 } \Big\{ \| (\mH_{n , \lambda}^{-1} \varphi )^{\prime} \para
    X_{n, \lambda} -  (\mH_{\lambda}^{-1} \varphi)^{\prime} \para (- \Delta + \lambda)^{-1} \xi  \|_{L^{2}} \\
     & \qquad \qquad \quad  + \| \mP_{n} (\mH_{n , \lambda}^{-1} \varphi)^{\sharp} - (\mH_{\lambda}^{-1} \varphi )^{\sharp}\|_{L^{2}} + \|
    \mQ_{n} (\mH_{n , \lambda}^{-1} \varphi )^{\sharp} \|_{L^{2}} \Big\} = 0,
  \end{align*} 
thus proving the convergence of the resolvents.

\end{proof}

Having established convergence in resolvent sense of the operator \(
\mH_{n} \) we complete the proof of
Theorem~\ref{thm:semidiscrete-pam-approximation} by showing that the
eigenfunctions of the operators converge in an appropriate sense.

\begin{proof}[Proof of Theorem~\ref{thm:semidiscrete-pam-approximation}]

As usual, let us fix \(\omega \in \Omega\), the latter satisfying
Assumption~\ref{assu:on-the-probability-space-PAM} and to lighten the
notation we avoid writing explicitly the dependence on \(\omega\) in what
follows. Also, as in the previous proof we restrict to discussing the case \(
d=2 \), which is more complicated.

To complete the proof of the theorem we collect all the previous results.
Proposition~\ref{prop:semidiscrete-resolvent-convergence} guarantees that \(
\mH_{n} \) converges to \( \mH \) in the resolvent sense, as a sequence of operators on \(
L^{2}(\TT^{d}) \). In particular,
Corollary~\ref{cor:convergence-eigenfunctions} guarantees that, for any
eigenvalue \( \lambda \) of \( \mH \) with multiplicity \( m(\lambda) \in \NN
\) and associated orthogonal eigenfunctions \( \{ e_{j} \}_{j = 1}^{m(\lambda)}
\), there exists a sequence \( \{ e_{j}^{n} \}_{j =1}^{m(\lambda)} \subseteq
L^{2}(\TT^{d}) \), for \( n \geqslant n(\lambda) \) with \( n(\lambda) \) sufficiently large, such that:
\[ e^{n}_{j} \to e_{j}, \qquad \mH_{n} e^{n}_{ j} \to \mH e_{j}, \qquad \text{
in } L^{2}(\TT^{d}). \] 
Moreover any \( e^{n}_{j} \) can be represented as
\[ e^{n}_{j} = \sum_{i = 1}^{m (\lambda)} \alpha_{i j}^{n}
\overline{e}^{n}_{i}, \qquad \sum_{i = 1}^{m(\lambda)} (\alpha_{i
j}^{n})^{2} = 1 , \]
where \( \overline{e}_{i}^{n} \) are eigenfunctions for \( \mH_{n} \) with
eigenvalue \( \lambda^{n}_{i} \) such that \( \lim_{n \to \infty}
\lambda^{n}_{i} = \lambda. \) To conclude the proof we will show the
following additional convergences, for any \( \kappa \in (0, \frac{1}{2}
) \) sufficiently small:
\begin{equation*}
\begin{aligned}
\Pi_{n} e^{n}_{j} \to e_{j}, \qquad \Pi_{n} \mH_{n} e^{n}_{j} \to \lambda e_{j}
\qquad \text{ in } \mC^{\kappa}(\TT^{d}),
\end{aligned}
\end{equation*}
for \( \kappa > 0 \) sufficiently small.
In the previous discussion we already have explained the convergences above in \(
L^{2}(\TT^{d}) \). By compact embedding \( \mC^{\kappa}(\TT^{d}) \subseteq
\mC^{\kappa^{\prime}} (\TT^{d})  \) for
\( \kappa^{\prime} < \kappa \), and since \( \kappa \) is arbitrary, it thus
suffices to prove the bounds:
\begin{equation*}
\begin{aligned}
\sup_{n \geqslant n(\lambda)} \Big\{ \| \Pi_{n} e^{n}_{j} \|_{\mC^{\kappa}} +
\| \Pi_{n} \mH_{n} e^{n}_{j} \|_{\mC^{\kappa}} \Big\} < \infty.
\end{aligned}
\end{equation*}
By our previous considerations, observing that
$\mH_{n} e^{n}_{j} = \sum_{i =1 }^{m(\lambda)} \lambda^{n}_{i} \alpha_{ij}
\overline{e}^{n}_{i},$
we can further reduce the problem to proving that
\begin{equation}\label{eqn:proof-theorem-bound-eigenfunctions}
\begin{aligned}
\sup_{n \geqslant n(\lambda)} \| \Pi_{n} \overline{e}^{n}_{i}
\|_{\mC^{\kappa}}  < \infty, \qquad \forall i = 1, \ldots, m(\lambda).
\end{aligned}
\end{equation}
Now we fix \( i \) and make use of the fact that \( \overline{e}^{n}_{i} \) is an eigenfunction
of \( \mH_{n} \) with eigenvalue \( \lambda_{i}^{n} \to \lambda\). To lighten
the notation, since \( i \) is fixed, let us write
$ e^{n} = \overline{e}^{n}_{i}.$
We find that for \( \mu>1 \) sufficiently large such that
Proposition~\ref{prop:semidiscrete-resolvent-convergence} applies (with \(
\lambda \) replaced by \( \mu \), and following the notations introduced by the
proposition and its proof) and defining \( v^{n} = (\mu - \lambda^{n}_{i})
e^{n} \):
\begin{align*}
e^{n} & =  \mH_{n,
\mu}^{-1} v^{n} = \Pi_{n}^{2} \big\{ (e^{n})^{\prime} \para X_{n, \mu}
\big\} + ( e^{n})^{\sharp}
\end{align*}
The bound~\eqref{eqn:bound-resolvent-paracontrolled-norm} now guarantees
that
\begin{align*}
\| e^{n} \|_{\mD^{\mu}_{n}} = \| (e^{n})^{\prime}  \|_{\mC^{1 - \kappa}_{2}} + \| \mP_{n}
( e^{n} )^{\sharp}   \|_{\mC^{1 + \kappa}_{2}} + n^{2 - \kappa}
\| \mQ_{n} ( e^{n})^{\sharp} \|_{\mC^{-1 + 2 \kappa}_{2}} \lesssim \|
e^{n} \|_{L^{2}} \lesssim 1.
\end{align*}
This bound is sufficient for large scales, but small scales need more care.
Here we observe that
\begin{align*}
(e^{n})^{\sharp} = (- \mA_{n} + \mu)^{-1} v^{n}  + \widetilde{M}^{\sharp,
1}_{v^{n}}(e^{n}) + M^{\sharp, 2}_{v^{n}}(e^{n}),
\end{align*}
where \( \widetilde{M}^{\sharp, 1}, M^{\sharp, 2}\) have been introduced in
Step 6 of Proposition~\ref{prop:semidiscrete-resolvent-convergence} and
satisfy, following~\eqref{eqn:proof-proposition-bound-M-tilde-sharp-1} and
\eqref{eqn:proof-proposition-bound-M-tilde-sharp-2}:
\begin{align*}
\| \mQ_{n} \widetilde{M}^{\sharp, 1}_{v^{n}} (e^{n})  \|_{\mC^{1 +
\frac{\kappa}{2}}_{2}} + \| \mQ_{n} M^{\sharp, 2}_{v^{n}} (e^{n})
\|_{\mC^{1 + \frac{\kappa}{2}}_{2}} \lesssim \| e^{n} \|_{\mD^{\mu}_{n}}
\lesssim 1.
\end{align*}
Now we are in position to conclude our estimate. By Besov embedding, since we
are considering the case \( d=2 \) (note that in \( d=1 \) we loose less
regularity, so the estimates simplify) we have
\begin{align*}
\| \varphi \|_{\mC^{\alpha - 1}} \lesssim \| \varphi \|_{\mC^{\alpha}_{2}},
\qquad \forall \varphi \in \mC^{\alpha}_{2}.
\end{align*}
In particular we find that
\begin{align*}
\sup_{n \geqslant n(\lambda)} \| (e^{n})^{\prime} \|_{\mC^{- \kappa}} \lesssim
\sup_{n \geqslant n(\lambda)} \| e^{n} \|_{\mD^{\mu}_{n}}  < \infty,
\end{align*}
so that (note that the term \( \Pi_{n}^{3} \) appears because we want to
estimate the norm of \( \Pi_{n} e^{n} \): for this estimate the presence of the
additional \( \Pi_{n} \) does not matter):
\begin{align*}
\sup_{n \geqslant n(\lambda)} \Big\| \Pi_{n}^{3} \big\{ (e^{n})^{\prime} \para X_{n, \mu}
\big\} \Big\|_{\mC^{1 - 3 \kappa}} & \lesssim \sup_{n \geqslant n(\lambda)}  n^{2 - \kappa} \| (e^{n})^{\prime}
\para X_{n, \mu} \|_{\mC^{-1 -2\kappa}} \\
& \lesssim \sup_{n \geqslant n(\lambda)}  \| (e^{n})^{\prime}
\|_{\mC^{- \kappa}} \| X_{n, \mu} \|_{\mC^{-1- \kappa}} \\
& \lesssim \sup_{n \geqslant n(\lambda)} \| e^{n} \|_{\mD^{\mu}_{n}}
\opnorm{ \bx_{n}}_{n, \kappa} < \infty. 
\end{align*}
Next we control the rest term:
\begin{align*}
\| \Pi_{n} (e^{n})^{\sharp} \|_{\mC^{\frac{\kappa}{2}}} & \lesssim \| \Pi_{n}
(e^{n})^{\sharp} \|_{\mC^{1 + \frac{\kappa}{2}}_{2}} \\
& \lesssim \| \mP_{n} \Pi_{n} ( e^{n})^{\sharp} \|_{\mC^{1+ \frac{\kappa}{2}
}_{2}} + \| \mQ_{n} \Pi_{n} (- \mA_{n} + \mu)^{-1} v^{n} \|_{\mC^{1 +
\frac{\kappa}{2}}_{2}} + \|
\mQ_{n} \Pi_{n} ( \widetilde{M}^{\sharp, 1}_{v^{n}}(e^{n}) + M^{\sharp,
2}_{v^{n}} (e^{n})  ) \|_{\mC^{1 + \frac{\kappa}{2}}_{2}}\\
& \lesssim\| \mQ_{n} \Pi_{n} (- \mA_{n} + \mu)^{-1} v^{n} \|_{\mC^{1 +
\frac{\kappa}{2}}_{2}} + \| \mP_{n} ( e^{n})^{\sharp} \|_{\mC^{1+
\frac{\kappa}{2}}_{2}} +  \| \mQ_{n} ( \widetilde{M}^{\sharp, 1}_{v^{n}}(e^{n}) + M^{\sharp,
2}_{v^{n}} (e^{n})  ) \|_{\mC^{1 + \frac{\kappa}{2}}_{2}} \\
& \lesssim \| \mQ_{n} \Pi_{n} (- \mA_{n} + \mu)^{-1} v^{n} \|_{\mC^{1 +
\frac{\kappa}{2} }_{2}} + \| e^{n} \|_{\mD^{\mu}_{n}},
\end{align*}
where in the last step we used all the previous estimates. Observe that so far
we did non use the smoothing effect of the additional term \( \Pi_{n} \). We
use this effect in the following last step, where we estimate the only
remaining term:
\begin{align*}
\| \mQ_{n} \Pi_{n} (- \mA_{n} + \mu)^{-1} v^{n}   \|_{\mC^{1 +
\frac{\kappa}{2}}_{2}} \lesssim \big\|  (1 + | \cdot|^{2})^{\frac{d + 1}{4}}
\mF_{\TT^{d}} \big(\mQ_{n} \Pi_{n} (- \mA_{n} + \mu)^{-1} e^{n} \big)
\big\|_{L^{2}(\ZZ^{d})}.
\end{align*}
Here we used that for \( \kappa \) sufficiently small and since \( d=2 \): \( 1
+ \frac{\kappa}{2}  \leqslant \frac{d +1}{2} \). Then we used one of the many definitions of fractional Sobolev spaces, via the
norm (for \( \alpha>0 \)):
\[ \| \varphi \|_{H^{\alpha}} = \| (1 - \Delta)^{\frac{\alpha}{2}} \varphi
\|_{L^{2}(\TT^{d})} \simeq \| (1 + | \cdot |^{2} )^{\frac{\alpha}{2}}
\mF_{\TT^{d}} \varphi \|_{L^{2}}, \] 
together with the embedding (see for example \cite[Section 2.3.5]{Triebe2010}):
\[ \| \varphi \|_{\mC^{\alpha}_{2}} \lesssim \| \varphi \|_{H^{\alpha}}, \qquad
\forall \varphi \in H^{\alpha}. \]
Hence we conclude with the following estimate (here we follow the notations of
Section~\ref{sec:schauder_estimates}):
\begin{align*}
\big\|  (1 + | \cdot|^{2})^{\frac{d + 1}{4}} & \mF_{\TT^{d}} \big(\mQ_{n} \Pi_{n}
(- \mA_{n} + \mu)^{-1} e^{n} \big) \big\|_{L^{2}(\ZZ^{d})}^{2} \\
& = \sum_{k \in \ZZ^{d}} \bigg\vert \frac{ (1 + |k|^{2})^{\frac{d + 1}{4}} }{\mu +
n^{2}(1 - \hat{\chi}^{4}(n^{-1} k)))} \bigg\vert^{2} | \hat{\chi}(n^{-1} k)
(1 - \daleth)(n^{-1} k)) |^{2} | \hat{e}^{n} (k)|^{2} \\
& \lesssim  \left( \frac{1}{\mu + n^{2}} \right)^{2} \sum_{k \in \ZZ^{d}} (1 + | k |)^{d + 1} | \hat{\chi}(n^{-1} k)
(1 - \daleth)(n^{-1} k)) |^{2} | \hat{e}^{n} (k)|^{2} \\
& \lesssim \left( \frac{1}{\mu + n^{2}} \right)^{2} \sum_{k \in \ZZ^{d}}
\frac{(1 + | k |)^{d + 1}}{(1 + |n^{-1} k|)^{d +1}} | (1 - \daleth)(n^{-1} k))
|^{2} | \hat{e}^{n} (k)|^{2}\\
& \lesssim \left( \frac{1}{\mu + n^{2}} \right)^{2} n^{d +1} \sum_{k \in
\ZZ^{d}\setminus \{ 0 \}}
\frac{(1 + | k |)^{d + 1}}{ | k|^{d +1}}  | \hat{e}^{n} (k)|^{2} \\
& \lesssim n^{(d + 1) - 4}\| e^{n} \|_{L^{2}}^{2} \lesssim \| e^{n}
\|_{L^{2}}^{2} \lesssim 1.
\end{align*}
Here we used the fact that \( \hat{\chi}(k) < 1 \) for \( k \neq 0 \) together
with the support properties of \( \daleth \) to bound
\begin{align*}
\frac{1}{\mu + n^{2}(1 - \hat{\chi}^{4}(n^{-1} k))} \lesssim
\frac{1}{n^{2}}
\end{align*}
uniformly over \( n \) and \( k \) such that \( (1 - \daleth)(n^{-1} k) \neq 0
\). We also applied the bound
\begin{align*}
| \hat{\chi}(k) | \lesssim \frac{1}{(1 + | k |)^{\frac{d+1}{2}}}
\end{align*}
from Lemma~\ref{lem:taylor_expansion_and_decay_ft_characteristic_function}.
This concludes the proof of the theorem, since we have proven~\eqref{eqn:proof-theorem-bound-eigenfunctions} with \( \kappa \)
replaced by \( \frac{\kappa}{2} \) (but this does not matter since \( \kappa
\) is arbitrarily small).

Before we conclude, let us observe that in the last bound we used that \( d=2
\) to bound \( 1 + \frac{\kappa}{2} \leqslant \frac{d+1}{2}\). If \( d=1
\) this fails, but we actually need less, since by Besov embedding 
$\| \Pi_{n} e^{n}  \|_{\mC^{\kappa}(\TT^{1})} \lesssim \| \Pi_{n} e^{n}
\|_{\mC^{\frac{1}{2} + \kappa}_{2}(\TT^{1})}.$
In particular, following all the previous steps we can bound, for \( \kappa
\) sufficiently small such that \( \frac{1}{2} + \kappa \leqslant
\frac{d +1}{2} =1\):
\begin{align*}
\| \mQ_{n} \Pi_{n} (- \mA_{n} + \mu) v^{n} \|_{\mC^{\kappa}} & \lesssim \| \mQ_{n} \Pi_{n} (- \mA_{n} + \mu)^{-1} e^{n}   \|_{\mC^{\frac{1}{2} +
\frac{\kappa}{2}}_{2}} \\
& \lesssim \big\|  (1 + | \cdot|^{2})^{\frac{d + 1}{4}}
\mF_{\TT^{d}} \big(\mQ_{n} \Pi_{n} (- \mA_{n} + \mu)^{-1} e^{n} \big)
\big\|_{L^{2}(\ZZ^{d})},
\end{align*}
and from here we can follow, for example, the same calculations as above.
\end{proof}

\begin{remark}\label{rem:on-bound-for-convergence-of-eigenfunctions}
We observe that in the last bound for \( (e^{n})^{\sharp}\) we used \( \Pi_{n} \) to gain \(
\frac{d+1}{2} \) regularity. In dimension \( d=2 \) this is crucially larger
than \( 1 \). This statement is in apparent contradiction with
Corollary~\ref{cor:regularity_gain_convolution_with_characteristic_functions},
where we show a possible regularity gain of at most \( 1 \). While
the latter corollary works for any integrability parameter \( p \) and extends
to other characteristics functions (than just those of balls), the improvement
we see in the proof depends on the choice \( p=2 \) and our exact computations
for the decay of the Fourier transform \( \hat{\chi}\).
\end{remark}

\subsection{Commutator estimates}\label{subsec:pam:paraproducts_and_commutator_estimates}

This section is devoted to products of distributions and commutator estimates, starting with the decomposition in paraproducts (through the symbol $\para$) and resonant products ($\reso$). For \(\varphi, \psi \in
\mS^{\prime}(\TT^{d})\) set
\[ 
  S_{i}\varphi := \sum_{j = -1}^{i -1} \Delta_{j} \varphi, \qquad \varphi
  \para \psi := \sum_{i \geq -1} S_{i -1} \varphi \Delta_{i}
  \psi, \qquad \varphi \reso \psi := \sum_{|i {-} j | \leq 1} \Delta_{j} \varphi
  \Delta_{i} \psi,
\]
where the latter sum might not be well defined. Then, an \textit{a priori} ill-posed
product of \(\varphi\) and \(\psi\) can be written as
\( \varphi \cdot \psi = \varphi \para \psi + \varphi \reso \psi + \varphi \rpara
\psi.\) 
The following estimates are classical, see e.g.
\cite[Lemmata 2.82 and 2.85]{BahouriCheminDanchin2011FourierAndNonLinPDEs} and guarantee that the product is actually well-defined if the regularities $\alpha$ and $\beta$ of $\varphi$ and $\psi$ satisfy $\alpha+\beta>0$.

\begin{lemma}\label{lem:paraproduct-estimates}
   Fix \(\alpha, \beta \in \RR\) and \(p, q, r \in [1, \infty]\) such that
   \(\frac{1}{r}  = \frac{1}{p} + \frac{1}{q} \). Then, for all \(\varphi, \psi \in
  \mS^{\prime}(\TT^{d})\) the following estimates are satisfied:
  \begin{align*}
    \| \varphi \para \psi \|_{\mC^{\alpha}_r} \lesssim \| \varphi
    \|_{L^{p}} \| \psi \|_{\mC^{\alpha}_{q}}, & \qquad \| \varphi \para \psi \|_{\mC^{\alpha {+} \beta}_{r}}\lesssim \| \varphi
    \|_{\mC^{\beta}_{p}} \| \psi \|_{\mC^{\alpha}_{q}}, \ \ \text{if} \ \ \beta  <
    0, \\
    \| \varphi \reso \psi \|_{\mC^{\alpha {+} \beta}_{r}} & \lesssim \| \varphi
    \|_{\mC^{\beta}_{p}} \| \psi \|_{\mC^{\alpha}_{q}} \ \ \text{if} \ \ \alpha
    {+} \beta > 0.
  \end{align*}
\end{lemma}
The rest of this subsection  deals with the following commutators.
\begin{definition}\label{def:all_commutators_at_once}
  For distributions \(\varphi, \psi, \sigma \in \mS^{\prime}(\TT^{d})\) we define the (a-priori ill-posed) 
  commutators
  \begin{align*}
    C^{\reso}(\varphi, \psi, \sigma) & := \varphi \reso ( \psi \para \sigma) -
    \psi ( \varphi \reso \sigma), \\
    C^{\Pi}_{n} ( \varphi, \psi) & := \Pi_{n}^2( \varphi \para \psi) - \varphi
    \para \Pi_{n}^2 \psi,\\
    C_{n, \lambda}(\varphi, \psi)  &:= (- \mA_{n} + \lambda)^{-1} (\varphi \para
    \psi) - \varphi \para (- \mA_{n} + \lambda)^{-1} \psi.
  \end{align*}
\end{definition}

The first commutator estimate is crucial, but by now well-known.

\begin{lemma}[\cite{GubinelliPerkowski2015LectureNotes}, Lemma 14]\label{lem:commutator_resonan}
  For \(\varphi, \psi, \sigma \in \mS^{\prime}(\TT^{d})\), \(\alpha, \beta,
  \gamma \in \RR\) with \(\alpha+ \beta + \gamma>0\) and \(p \in [1, \infty]\):
  \[ \| C^{\reso} (\varphi, \psi , \sigma) \|_{\mC^{\alpha +
  \gamma}_{p}} \lesssim \| \varphi \|_{\mC^{\alpha}} \| \psi
\|_{\mC^{\beta}_p} \| \sigma \|_{\mC^{\gamma}}.  \] 
\end{lemma}

We pass to the second commutator. Recall the operators $\mP_n, \mQ_n$ as in Definition \ref{def:cut_off_operators}.

\begin{lemma}\label{lem:commutator_with_averaging}
  For \(\varphi, \psi \in \mS^{\prime}(\TT^{d})\) and \(\alpha \in \RR, \beta
  >0, p \in [1, \infty]\) it holds for every \(\delta \in [0, \beta \wedge 1)\):
  \begin{align*}
    \| \mP_{n}C^{\Pi}_{n}(\varphi, \psi) \|_{\mC^{\alpha + \delta}_{p}} \lesssim \| \varphi
    \|_{\mC^{\beta}_{p}} \| \psi \|_{\mC^{\alpha}}, \qquad \| \mQ_{n} C^{\Pi}_{n} (\varphi, \psi) \|_{\mC^{\alpha}_{p}} \lesssim
    n^{-\delta} \| \varphi \|_{\mC^{\beta}_{p}} \| \psi \|_{\mC^{\alpha}}.
  \end{align*}
\end{lemma}

\begin{proof}

Note that for any $i\geq 0$ there exists an annulus \(\mA\) (that is a set of the form $\{ k \in \RR^d \ | \ r \leq |k| \leq R\}$ for some $0<r<R$) such that the Fourier transform of 
  \[ \Pi _{n}^2 [S_{i {-} 1} \varphi \Delta_{i} \psi] - S_{i {-} 1} \varphi
  \Pi_{n}^2 \Delta_{i} \varphi  \] 
  is supported in \(2^{i}\mA\). 
  It is therefore sufficient to show that 
  \begin{align}
  \label{eqn:commutator_bound}
    \big\| \Pi _{n}^2 [S_{i {-} 1}\varphi
    \Delta_{i} \psi] - S_{i {-} 1} \varphi \Pi_{n}^2 \Delta_{i} \varphi \big\|_{L^p} \lesssim n^{-\delta} \| \varphi \|_{\mathcal{C}^\beta_p} \|\Delta_i \psi\|_{L^\infty},
  \end{align}
  since this implies the required bound by estimating $n^{-\delta} \lesssim 2^{-\delta i}$ for $i$ such that $\mP_n \Delta_i \neq 0$.
  To obtain \eqref{eqn:commutator_bound},  recall the  Sobolev-Slobodeckij
  characterization of fractional spaces of
  Proposition~\ref{prop:sobolev_slobodeckij_equivalence}, which implies that
  for \(\delta \in [0, \beta \wedge 1)\)
  \begin{multline*}
    \| \Pi _{n}^2 [S_{i {-} 1} \varphi \Delta_{i} \psi] - S_{i {-} 1} \varphi
    \Pi_{n}^2 \Delta_{i} \varphi \|_{L^p} \leq \bigg( \int_{\TT^{d}}
    \Big\vert \mint{-}_{B_{n}(x)} [S_{i {-} 1} \varphi (y) {-} S_{i {-} 1}
    \varphi (x)] \Delta_{i} \psi(y) \ud y \Big\vert^p \ud x \bigg)^{ 1/p} \\
     \lesssim n^{-\delta} \bigg( \int_{\TT^{d}}
    \Big\vert \mint{-}_{B_{n}(x)} \frac{[S_{i {-} 1} \varphi (y) {-} S_{i {-} 1}
    \varphi (x)]}{|y-x|^{\delta}} \Delta_{i} \psi(y) \ud y \Big\vert^p \ud x \bigg)^{ 1/p} \\
	 \lesssim n^{-\delta} \bigg( \int_{\TT^{d}}
    \int_{B_{n}(x)} \frac{|S_{i {-} 1} \varphi (y) {-} S_{i {-} 1}
    \varphi (x)|^p}{|y-x|^{d+\delta p}} \ud y  \ud x \bigg)^{ 1/p} \|\Delta_{i} \psi\|_{\infty} 
    \lesssim n^{-\delta}
    \| S_{i {-} 1} \varphi \|_{\mC^{\beta}_p} 2^{{-} \alpha i} \| \psi\|_{\mC^{\alpha}},
  \end{multline*}
  where the first inequality follows by Jensen and we have used the embedding $B^{\beta}_{p, \infty}\subset B^\delta_{p,p}$. Now the result follows since
  $ \| S_{i {-} 1} \varphi \|_{\mC^{\beta}_p} \lesssim \| \varphi
  \|_{\mC^{\beta}_p}.$
\end{proof}

\begin{lemma}\label{lem:commutator_resolvent}
  For \(\alpha \in (0,1), \beta \in \RR, \lambda \geq 1\)
  and \(p \in [1, \infty]\) it holds that: 
  \[ \| \mP_{n} C_{n, \lambda}( \varphi, \psi) \|_{\mC^{\alpha + \beta +
  2}_{p}} \lesssim \| \varphi \|_{\mC^{\alpha}_{p}} \| \psi \|_{\mC^{\beta}}, \qquad
\forall \varphi \in \mC^{\alpha}_{p}, \psi \in \mC^{\beta}.\]
In addition there exists a \(k \in \NN\) such that for \(n \geq k\)
\[  n^{3} \|\mQ_n C_{n,\lambda}(\varphi, \psi)\|_{\mC^{\alpha +\beta-1}_p}
\lesssim \| \varphi \|_{\mC^{\alpha}_{p}} \| \mQ_{n-k} \psi \|_{\mC^{\beta}},
\qquad \forall \varphi \in \mC^{\alpha}_{p}, \psi \in \mC^{\beta}.\] 
\end{lemma}
\begin{proof}
  By the elliptic Schauder
  estimates in Proposition~\ref{prop:elliptic_schauder_estimates}, it is
  sufficient to prove that
  \begin{align*}
    \| (-\mA_n +\lambda) \mP_{n}C_{n, \lambda}(\varphi, \psi)\|_{\mC^{ \alpha +
  \beta}_p} & \lesssim \| \varphi\|_{\mC^\alpha_p} \| \psi \|_{\mC^\beta}, \\
  n \|(- \mA_{n} + \lambda)\mQ_n C_{n,\lambda}(\varphi, \psi)\|_{\mC^{\alpha +\beta-1}_p}
  & \lesssim \| \varphi \|_{\mC^{\alpha}_{p}} \| \mQ_{n-k} \psi \|_{\mC^{\beta}}.
\end{align*}
In turn to obtain this bound,
since the quantities below are supported in an annulus $2^i \mA$, it suffices
to estimate for a given sequence \(i(n)\) such that \(2^{i(n)} \simeq
n\):
\begin{equation}\label{eqn:proof_commutator_estimate_intermediate_large_scale}
  \begin{aligned}
\|S_{i-1} \varphi \Delta_i \psi - (-\mA_n +\lambda)[ S_{i-1}\varphi (-\mA_n
+\lambda)^{-1} \Delta_i \psi ]\|_{L^p} \lesssim 2^{-i (\alpha + \beta)} \|
\varphi\|_{\mC^\alpha_p} \|\psi\|_{\mC^\beta},
\end{aligned}
\end{equation}
if \(i \leq i(n)\), and similarly
\begin{equation}\label{eqn:proof_commutator_estimate_intermediate_small_scale}
  \begin{aligned}
\|S_{i-1} \varphi \Delta_i \psi - (-\mA_n +\lambda)[ S_{i-1}\varphi (-\mA_n
+\lambda)^{-1} \Delta_i \psi ]\|_{L^p} \lesssim n^{-1}2^{-i (\alpha+  \beta-1)} \|
\varphi\|_{\mC^\alpha_p} \| \mQ_{n-k}\psi\|_{\mC^\beta},
\end{aligned}
\end{equation}
if \(i >i(n)\). Moreover, we can choose \(k\) such that
\[ \mQ_{n-k} \Delta_{i} = \Delta_{i}, \forall i \geq i(n), n \in \NN, \] 
so that we may replace \(\psi\) by \(\mQ_{n-k} \psi\) on small scales (hence we
will no longer discuss the appearance of \(\mQ_{n-k}\)).
To obtain these estimates, let \(B_n(\varphi, \psi)\) be defined as
  \[ 
    B_n (\varphi, \psi) (x)= n^2 \mint{-}_{B_{n}(x)} 
    \mint{-}_{B_{n}(y)} \mint{-}_{B_{n}(z)} \mint{-}_{B_{n}(r)}  (\varphi(s)
    {-} \varphi(x))(\psi(s) {-} \psi(x) ) \ud s \ud r \ud z \ud y.
  \]
Then $\mA_{n}$ acting on a product can be decomposed as 
  \[ \mA_n (\varphi \cdot \psi) = \mA_n (\varphi) \cdot \psi {+} \varphi \cdot
  \mA_n (\psi) {+} B_n( \varphi, \psi),\]
  Hence proving
  Equations~\eqref{eqn:proof_commutator_estimate_intermediate_large_scale} and \eqref{eqn:proof_commutator_estimate_intermediate_small_scale}
  reduces to finding a bound for
  \begin{align*}
  \| (-\mA_n + \lambda)[ S_{i-1} \varphi ] (-\mA_n +\lambda)^{-1} [\Delta_i \psi] \|_{L^p} + \| B_n ( S_{i-1} \varphi , (-\mA_n + \lambda )^{-1} \Delta_i \psi ) \|_{L^p}.
  \end{align*}
Starting with the first term, one has:
\begin{align*}
	\|(-\mA_n + \lambda)[ S_{i-1} \varphi ] (-\mA_n +\lambda)^{-1}
	[\Delta_i \psi] \|_{L^p} & \lesssim \|(-\mA_n + \lambda) [S_{i-1}
	\varphi ]\|_{L^p} \|  (-\mA_n +\lambda)^{-1} [\Delta_i \psi]
	\|_{L^\infty}.
\end{align*}
If $i \leq i(n)$,  since $\alpha <2$, one can estimate via Proposition~\ref{prop:crucial_for_schauder}:
\begin{align*}
\|(-\mA_n + \lambda) [S_{i-1} \varphi ]\|_{L^p} \leq \sum_{j =-1}^{i-1} \|(-\mA_n + \lambda) [\Delta_j \varphi ]\|_{L^p} \lesssim  \sum_{j =-1}^{i-1} 2^{j(2-\alpha)}\| \varphi \|_{\mC^\alpha_p}\lesssim 2^{i(2-\alpha)}\| \varphi \|_{\mC^\alpha_p}.
\end{align*}
If $i>i(n)$, following the previous calculations and using that $\alpha>0$:
\begin{align*}
\|(-\mA_n + \lambda) [S_{i-1} \varphi ]\|_{L^p} & \leq \sum_{j =-1}^{i(n)-1} \|(-\mA_n + \lambda) [\Delta_j \varphi ]\|_{L^p} + \sum_{j = i(n) }^{i-1} \|(-\mA_n + \lambda) [\Delta_j \varphi ]\|_{L^p} \\
& \lesssim n^{(2-\alpha)} \|\varphi\|_{\mC^\alpha_p}.
\end{align*}
By Proposition~\ref{prop:schauder_estimates} moreover
\begin{align*}
  \|(-\mA_n +\lambda)^{-1} \Delta_i \psi \|_{L^{\infty}} \lesssim \Big( 2^{-2i}1_{ \{i \leq
i(n)\}} + n^{-2} 1_{\{i > i(n)\}}\Big) 2^{-\beta i} \|\psi\|_{\mC^\beta}.
\end{align*}
Together with the previous bounds we have proven that for \(i \leq i(n)\):
\begin{align*}
  \|(-\mA_n + \lambda)[ S_{i-1} \varphi ] (-\mA_n +\lambda)^{-1} [\Delta_i
  \psi] \|_{L^p} \lesssim 2^{-i(\alpha + \beta)} \| \varphi
  \|_{\mC^{\alpha}_{p}} \| \psi \|_{\mC^{\beta}},
\end{align*}
and similarly (using that \(\alpha<1\)) for \(i > i(n)\):
\begin{align*}
  \|(-\mA_n + \lambda)[ S_{i-1} \varphi ] (-\mA_n +\lambda)^{-1} [\Delta_i
  \psi] \|_{L^p} & \lesssim n^{- \alpha} 2^{-i\beta} \| \varphi
  \|_{\mC^{\alpha}_{p}} \| \mQ_{n-k} \psi \|_{\mC^{\beta}} \\
  & \lesssim n^{-1}2^{- i(\beta + \alpha - 1)}\| \varphi
  \|_{\mC^{\alpha}_{p}} \| \mQ_{n-k} \psi \|_{\mC^{\beta}}, 
\end{align*}
which are bounds of the required order for
\eqref{eqn:proof_commutator_estimate_intermediate_large_scale} and
\eqref{eqn:proof_commutator_estimate_intermediate_small_scale}.
Finally, we have to bound the term containing $B_n$. If $i\leq i(n)$, using
\(\alpha < 1\) we find
  \begin{align*} 
    \| B(S_{i-1} \varphi, (-\mA_{n} + \lambda)^{-1} \Delta_{i} \psi) \|_{L^{p}} & \lesssim \| \nabla S_{i-1}\varphi
    \|_{L^{p}} \| \nabla (-\mA_n +\lambda)^{-1}\Delta_{i} \psi \|_{L^{\infty}}
    \\
    & \lesssim 2^{-i(\alpha - 1)} \| \varphi\|_{\mC^{\alpha}_{p}} 2^{-i(1+\beta)}\| \psi \|_{\mC^{\beta}} 
    \lesssim  2^{-i( \alpha + \beta)} \| \varphi\|_{\mC^\alpha_p} \| \psi \|_{\mC^{\beta}},
  \end{align*}
  whereas if $i> i(n)$
  \begin{align*}
  \| B_n( S_{i-1} \varphi, (-\mA_n +\lambda)^{-1}\Delta_{i} \psi) \|_{L^{p}}
    & \lesssim n \| \nabla S_{i-1} \varphi \|_{L^p} \| (-\mA_n +\lambda)^{-1}\Delta_{i} \psi \|_{L^{\infty}}\\
    & \lesssim n^{-1} 2^{- i(\alpha - 1)} 2^{-\beta
i}\|\varphi\|_{\mC^{\alpha}_{p}} \| \psi\|_{\mC^\beta}.
\end{align*}
These bounds are again of the correct order for
\eqref{eqn:proof_commutator_estimate_intermediate_large_scale},
\eqref{eqn:proof_commutator_estimate_intermediate_small_scale} and hence the
proof is complete.
\end{proof}

\appendix

\section{ }
\subsection{Construction of the process}\label{app:construction}

In this section we provide a rigorous construction of the spatial $\Lambda$-Fleming-Viot process (SLFV) in a random environment.
We work under the following assumptions.
\begin{assumption}\label{assu:rigorous-construction-slfv}
  Let $(\Omega, \mF, \PP)$ be a probability space. Fix $n \in \NN$ and $\mf{u}
  \in (0,1), d=1,2$ and let $w^{0} \colon \TT^d \rightarrow [0, 1]$ and $s_n
  \colon \Omega \times \TT^d \rightarrow (-1, 1)$ be two measurable
  functions.
\end{assumption}

The natural state space of the spatial SLFV process is:
$$M = \{ w \colon \TT^d \rightarrow [0, 1], \ \ w \  \text{measurable}\},$$
which is a metric space when endowed with the distance $d_M(u,w) = \sup_{x \in \TT^d}|u(x)- w(x)|$. Then under the assumption above, for $x \in \TT^d, \mf{p} \in \{\mf{a}, \mf{A}\}$ and any function $w: \TT^{d} \to [0,1]$ define the operator
\(\Theta^{\mf{p}}_{x}\colon M \rightarrow M\) by
\begin{align*} 
  \Theta^{\mf{p}}_{x} w (y) & = w(y) 1_{ \{B^c_{ n}(x)\}}(y) {+}
  (\mf{u}
  1_{ \{ \mf{p} = \mf{a}\}}  {+} (1 {-} \mf{u}) w(y)) 1_{ \{  B_{n}(x)\}}(y) \\
  & = w(y) {+} \mf{u}(1_{ \{\mf{p}=\mf{a}\}} {-} w(y)) 1_{\{{B_{n}(x)}\}} (y) .
\end{align*}
In the discussion below, let $\mathcal{B}(E)$ be the Borel sigma-algebra
associated to some metric space $E$. We say that a probability measure
$\PP^\omega$ on $(E, \mB(E))$ indexed by $\omega \in \Omega$ is a Markov
kernel, if for any $A \in \mB(E)$ the map $\omega \mapsto \PP^\omega(A)$ is
measurable. Then one can build the semidirect product measure $\PP \ltimes
\PP^\omega$ on $\Omega \times E$ (with the product sigma-algebra),
characterized, for $A\in \mF, B \in \mB(E)$, by:
$$\PP \ltimes \PP^\omega (A \times B ) = \int_A \PP^\omega(B) \PP(\ud
\omega).$$
In the definition below we write:
\[ s_{+} (x) = \max \{ s(x), 0\}, \qquad s_{-}(x) = \max \{ -s(x), 0\}. \] 
\begin{lemma}\label{lem:appendix-construction-markov-process}
    Under Assumption~\ref{assu:rigorous-construction-slfv}, fix $\omega \in \Omega$. There exists a unique Markov jump process $t \mapsto w(t)$ in $\DD( [0, \infty) ; M)$ started in $w(0) = w^0$, associated to the generator
    \[\mL(n, s_n(\omega), \mf{u}) \colon C_b(M ; \RR) \rightarrow C_b(M ; \RR),\]
    defined by
    \[ \mL ( f ) (w) = \int_{M} (f(w^\prime) - f(w))\mu (w, \ud w^\prime), \qquad f \in C_b(M; \RR),\]
    where the transition function $\mu \colon M \times \mB(M) \to \RR$ (depending on $s_n(\omega), \mf{u}, n$) is defined by:
    \begin{align*}
        \mu (w, \ud w^\prime) = 0 \ \ \text{unless there exist} \ \ x\in \TT^d, \mf{p} \in \{\mf{a}, \mf{A}\} \ \ \text{such that} \ \ w^\prime = \Theta^\mf{p}_x w. 
    \end{align*}
    And if $w^\prime = \Theta^{\mf{p}}_x w$ for some $x \in \TT^d, \mf{p}\in \{\mf{a}, \mf{A}\}$:
    \begin{align*}
        \mu  (w, \ud w^\prime) = & \bigg\{ (1 {-} |s_n(\omega, x)|) \Big[\Pi_n^{3}w  
      1_{\{\mf{p} = \mf{a} \}} {+} (1 {-}
    \Pi_n^{3} w )1_{\{\mf{p} = \mf{A}\}}\Big]
    (x)\\
    & + (s_n)_{-}(\omega, x) \bigg[ \big(\Pi_n^{3} w \big)^{2} 1_{\{\mf{p}=
      \mf{a}\}} {+} \big(1 {-} \big(\Pi_n^{3} w \big)^{2}
\big) 1_{\{\mf{p} = \mf{A}\}} \bigg] (x) \\
    & +  (s_n)_{+}(\omega, x)
  \bigg[ \Pi_n^{3} w( 2 {-} \Pi_n^{3} w ) 1_{\{\mf{p} = \mf{a} \}} {+} (1 {-}
\Pi_n^{3} w )^2 1_{\{\mf{p} = \mf{A}\}}  \bigg] (x) \bigg\} \ud x.
    \end{align*}
     The law $\PP^\omega$ of $w$ in $\DD([0, \infty); M)$ is a Markov kernel and induces the semidirect product measure $\PP \ltimes \PP^\omega$ on $\Omega \times \DD([0, \infty); M)$.
\end{lemma}

\begin{proof}
    Note that $\mu$ defined as above is a Markov kernel on $M \times \mB(M)$
(to be precise, here we have to observe that for fixed $w$ the set $\{
\Theta^{\mf{p}}_x w, \ \ x \in \TT^d, \mf{p} \in \{\mf{a}, \mf{A}\}\}$ is
closed and hence measurable in $M$). Hence, the Markov process is constructed
following \cite[Section 4.2]{EthierKurtz1986}. In addition, for $f \in C_b(M;
\RR)$ measurable and bounded the map $\omega \mapsto \int_M f(w^\prime)
\mu_\omega (w, \ud w^\prime)$ is measurable (we made explicit the dependence of
$\mu$ on $\omega$). This implies, e.g.\ by \cite[Equation 4.2.8]{EthierKurtz1986}, that the map $\omega \mapsto \PP^\omega(A)$ is measurable, for $A \in \mB(\DD([0, \infty); M))$. So the proof is complete.
\end{proof}

\begin{lemma}\label{lem:appendix-martingale-problem-proof}
Under Assumption~\ref{assu:rigorous-construction-slfv} fix $\omega \in \Omega$
and let $w$ be the Markov process as in the previous result. For any
$\varphi
\in L^{\infty}(\TT^d)$ the process $t \mapsto \langle w(t), \varphi \rangle$ satisfies the martingale problem of Lemma~\ref{lem:martingale_problem_model_description}.
\end{lemma}
\begin{proof}
  In the discussion below we omit the dependence of $s_n(\omega)$ on $n$ and
$\omega$, since such dependence is not relevant here. We will apply the
generator to functions of the form $F_\varphi (w) =F( \langle w , \varphi
\rangle)$, with $F \in C(\RR ; \RR), \varphi \in L^{\infty}( \TT^d)$. For simplicity we divide the operator \(\mL  =\mL(n, s, \mf{u})\) in three parts:
  \begin{align*}
    \mL (F_{\varphi})(w) & := \mL^{\mathrm{neu}} (F_{\varphi})(w) +
    \mL^{\mathrm{sel}} (F_{\varphi})(w) \\
    & := \mL^{\mathrm{neu}} (F_{\varphi})(w) +  \mL^{\mathrm{sel}}_{<}
    (F_{\varphi})(w) +  \mL^{\mathrm{sel}}_{>} (F_{\varphi})(w) 
  \end{align*}
  (the first is the neutral part, the second two are the selective parts of the operator), where
  \begin{align*}
    \mL^{\mathrm{neu}}(F_{\varphi}) (w) & =\int\limits_{\TT^d} (1 {-} |s(x)|)
    \Big[\Pi_n^{3} w 
      [F_{\varphi}(\Theta_{x}^{\mf{a}} w) {-} F_{\varphi}(w)] {+} (1 {-}
    \Pi_n^{3} w )[F_{\varphi}( \Theta_{x}^{\mf{A}} w) {-} F_{\varphi}( w)] \Big]
    (x) \ud x\\
    \mL^{\mathrm{sel}}_{<}(F_{\varphi})(w) & = \int\limits_{\TT^d} s_{-}(x)
    \bigg[ \big(\Pi_n^{3} w \big)^{2} [
  F_{\varphi}(\Theta_{x}^{\mf{a}} w) {-} F_{\varphi}(w)] {+} \big(1 {-}
    \big(\Pi_n^{3} w \big)^{2}
\big) [F_{\varphi}(\Theta_{x}^{\mf{A}} w ) {-} F_{\varphi}(w)] \bigg] (x)\ud x \\
    \mL^{\mathrm{sel}}_{>} (F_{\varphi}) (w)& = \int\limits_{\TT^d} s_{+}(x)
    \bigg[ \Pi_n^{3} w ( 2 {-} \Pi_n^{3} w ) [ F_{\varphi}(\Theta_{x}^{\mf{a}} w)
      {-} F_{\varphi}(w)] {+} (1 {-} \Pi_n^{3} w )^2 [F_{\varphi}(\Theta_{x}^{\mf{A}} w )
    {-} F_{\varphi}(w)] \bigg] (x)\ud x 
  \end{align*}
Now, in the special case of \(F = \mathrm{Id}_{\varphi}\),
the neutral part of the generator takes the form
\begin{align*}
    \mL^{\mathrm{neu}} (\mathrm{Id}_{\varphi})(w) & = \mf{u} n^{-d} \int_{\TT^d} (1 {-}
    |s(x)|)[ (\Pi_n^{3} w) (\Pi_n \varphi) - \Pi_n (w\varphi) ](x) \ud x,
\end{align*}
 Analogously, the selective part can be written as
  \begin{align*}
    \mL^{\mathrm{sel}} (\mathrm{Id}_\varphi)(w) = \mf{u} n^{-d} \int_{\TT^d}
    s(x) [ \Pi_n (w\varphi) {-}\big( \Pi_n^{3} w \big)^2 
    \Pi_n \varphi](x)+ 2s_{+}(x) [ \Pi_n^{3} w
    \Pi_n \varphi {-} \Pi_n (w\varphi) ](x) \ud x .
  \end{align*}
  Adding those two we conclude that 
  \begin{align*}
    \mL(\mathrm{Id}_\varphi)(w) & = \mf{u} n^{-d} \int_{\TT^{d}} [
      (\Pi_n^{3} w) 
    (\Pi_n \varphi)  - \Pi_n (w\varphi)  ](x) + s(x)[ (\Pi_n^{3} w )(\Pi_n \varphi) -
    (\Pi_n^{3} w )^{2} \Pi_n \varphi  ](x) \ud x.
  \end{align*}
  This justifies the drift in the required decomposition. 
 To obtain the predictable
  quadratic variation of the martingale make use of Dynkin's formula, that is
  \[ \langle M^{n} ( \varphi) \rangle_{t} = \int_{0}^{t} \mL(
  \mathrm{Id}_{\varphi}^{2}) - 2 \big( 
  \mathrm{Id}_{\varphi} \mL( \mathrm{Id}_{\varphi}) \big) (w_{r}) \ud r. \]
  Once again, it is natural to treat the terms involving $\mathcal{L}^{\mathrm{neu}}$ and
  $\mathcal{L}^{\mathrm{sel}}$ separately. For the neutral term:
  \begin{align*} 
    \big(\mL^{\mathrm{neu}} (\mathrm{Id}_{\varphi}^{2}) & - 2 F_{\varphi}
    \mL^{\mathrm{neu}} (\mathrm{Id}_{\varphi}) \big) (w) \\
    & =  \mf{u}^{2} n^{-2d} \int_{\TT^d} (1 {-}
    | s(x)|) \Big[ \Pi_{n}^{3}w   \big( \Pi_{n}\varphi {-}
      \Pi_{n}(w \varphi)\big)^2 + \big( 1 {-} \Pi_{n}^{3}w \big) \big( \Pi_{n} (w\varphi )
    \big)^2 \Big](x) \ud x,
  \end{align*}
which can be written as
\begin{align*}  
  \mf{u}^{2} n^{-2d} \int_{\TT^d} (1 {-} | s(x)|) \Big[  \Pi_{n}^{3}w \big[
    \big(\Pi_{n}\varphi \big)^2 - 2 \Pi_{n}\varphi(x) \Pi_{n}(w
  \varphi) \big] + \big[ \Pi_{n}(w \varphi ) \big]^2 \Big] (x) \ud x.
\end{align*}
  Analogous calculations for $\mathcal{L}^{\mathrm{sel}}_{<}$ lead to
\begin{align*}
  \big( \mL^{\mathrm{sel}}_{<}   (\mathrm{Id}_{\varphi}^{2}) & -  2
  \mathrm{Id}_{\varphi} \mL^{\mathrm{sel}}_{<}  \mathrm{Id}_{\varphi} \big)(w)= \\
  & = \mf{u}^{2} n^{-2d} \int_{\TT^d} s_{-}(x) \Big[
    (\Pi_{n}^{3}w )^{2} \big[(\Pi_{n}\varphi)^2 - 2 \Pi_{n}\varphi
    \Pi_{n}(w\varphi) \big] + \big[ \Pi_{n}(w \varphi ) \big]^2 \Big] (x) \ud x.
\end{align*}
  Whereas for $\mathcal{L}^{\mathrm{sel}}_{>}$ they lead to 
\begin{align*}
  \big(\mL^{\mathrm{sel}}_{>} (\mathrm{Id}_{\varphi}^{2}) &- 2
  \mathrm{Id}_{\varphi}
  \mL^{\mathrm{sel}}_{>} \mathrm{Id}_{\varphi}\big)(w) \\
   & = \mf{u}^{2} n^{-2d}\int_{\TT^d} s_{+}(x) \Big[
     (\Pi_{n}^{3}w) (2 {-} \Pi_{n}^{3} w) \big(\Pi_{n}\varphi {-}  
    \big(\Pi_{n}( w \varphi )\big)^2 + \big(1 {-}
    \Pi_{n}^{3}w \big)^2 \big(
    \Pi_{n}w \big)^2 \Big](x) \ud x \\
  & = \mf{u}^{2} n^{-2d}\int_{\TT^d} s_{+}(x) \Big[
    (\Pi_{n}^{3}w) (2 {-} \Pi_{n}^{3}w) \big[\big(\Pi_{n}\varphi \big)^2 {-} 2 \Pi_{n}\varphi
    \Pi_{n}(w \varphi ) \big] {+}
  \big[ \Pi_{n}(w \varphi\big) \big]^2 (x) \Big] \ud x. 
  \end{align*}
  Summing neutral and selective terms one obtains
\begin{align*}
  \mf{u}^{2} n^{-2d} \langle & \Pi_{n}^{3} w, (1 {-} | s|)\Big[\big( \Pi_{n}\varphi\big)^2 {-} 2
  \Pi_{n}\varphi \Pi_{n}(w\varphi )\Big] \rangle {+}
  \langle \big(\Pi_{n}(w \varphi) \big)^2, (1 {-}
  | s|) \rangle \\ 
  & + \mf{u} n^{-2 d}\langle  (\Pi_{n}^{3} w )^2, s_{-} 
  \Big[ (\Pi_{n} \varphi )^2 - 2(\Pi_{n}\varphi )
  \big(\Pi_{n} (w \varphi) \big) \Big] + \langle \big(
  \Pi_{n}(w \varphi)\big)^2, s_{-} \rangle \\
  & + \mf{u}^{2} n^{-2d} \langle  \Pi_{n}^{3} w, s_{+} \Big[
    (2 {-} \Pi_{n}^{3}w) \Big(
       (\Pi_{n}\varphi )^2 {-}2 (\Pi_{n}\varphi )
     \big( \Pi_{n}(w\varphi) \big) \Big) \Big] \rangle 
    {+} \langle  \big( \Pi_{n}(w \varphi) \big)^2,
    s_{+} \rangle,
\end{align*}
which can be written in the form from  the statement of the Lemma.

\end{proof}

\subsection{Stochastic bounds}
\label{app:stochastic-bouns}

This appendix is devoted to the control of the noise for approximations of the
Anderson Hamiltonian. In particular, we prove Proposition~\ref{prop:stochastic_estimates}.

\begin{proof}[Proof of Proposition~\ref{prop:stochastic_estimates}]
  First we will prove the bounds for
  \(\xi^{n}, X_{n, \lambda}\) and \(\xi^{n} \diamond
  \Pi_{n}^{2} X_{n, \lambda}\). Eventually we address the convergence of these terms.
  Although only in the first case the
  dimension is allowed to be both \(d=1\) and \(d=2\), we will keep
  \(d\) as a parameter throughout the proof, for the sake of clarity. For convenience, let us indicate sums on $\ZZ^d$ with integrals (for $m \in \NN$):
  \[\int_{(\ZZ^d)^m} f(k_1, \ldots, k_m) \ud k_1 \cdots \ud k_m = \sum_{k_1, \ldots, k_m \in \ZZ^d } f(k_1, \ldots, k_m).\]
\textit{Step 1: Bounds on \( \xi^{n} \).} 
First, observe that by Assumption~\ref{assu:probability-space}:
\[ | \xi^{n}(x)  |\leq 2 n^{\frac{d}{2} }.\] 
This explains both the \(L^{\infty}\) bounds on \(\xi^{n}\) and the bound in
\( \mC^{- \frac{\kappa}{2} }\) (i.e. for \(\zeta=1\)). If we show that
\[ \sup_{n \in \NN} \EE \big[ \| \xi^{n} \|_{\mC^{- \frac{d}{2}
- \frac \kappa 2}} \big] < \infty, \]
the bound for arbitrary \(\zeta\) follows, since by
interpolation, from the definition of Besov spaces, for any
$\zeta \in [0,1]$ and $\alpha, \beta\in \RR$:
\begin{align*}
  \| \varphi \|_{\mC^{ \zeta \alpha + (1- \zeta)\beta}} \leq \|\varphi
  \|_{\mC^{\alpha}}^{\zeta}\|\varphi \|_{\mC^\beta}^{1- \zeta}.
\end{align*} 
Hence let us consider the case \(\zeta = 0\). By Besov embedding, the required
inequality follows if one can show that for any \(p \in [2, \infty)\):
  \begin{align*}
    \sup_{n \in \NN} \EE \| \xi^{n} \|_{B^{{-}
    \frac d2 {-} \frac \kappa 4}_{p,p}}^{p} < \infty.
  \end{align*}   
 Here in view of Assumption~\ref{ass:selection_coefficient}, and by the discrete
 Burkholder-Davis-Gundy inequality as well as Jensen's inequality one finds
that:
  \begin{align*}
    \int_{\TT^ d}  \EE [| \Delta_{j} n^{ \frac{d}{2} } s_n|^p
    (x)] \ud x & \lesssim \int_{\TT^d}  \bigg( \sum_{z \in \ZZ^{d}_{n}}  n^{-d} |
    \Delta_{j} \chi_{Q_{n}}|^2 (z {+} x) \bigg)^{p/2 } \ud x 
    \\
     & \leq \int_{\TT^d} \bigg( \int_{\TT^d} \ud z \ |K_j(x {+} z) |^2
     \bigg)^{p/2 } \ud x \lesssim \| K_j \|_{L^2}^{p} \lesssim 2^{j
     \frac{dp}{2}},
  \end{align*}
  which is a bound of the required order.
    
  Now, let us pass to the bound in $\mC^{\kappa}_{\frac{1}{2 \kappa}}$. In fact we will prove that for any $p \in [1, \infty), \ \zeta\in [0, \frac 1 p)$ we have a bound on $\|\xi^n\|_{\mC^\zeta_p}$. We use the Sobolev-Slobodeckij norm of Proposition~\ref{prop:sobolev_slobodeckij_equivalence} for the Besov space $B^{\zeta}_{p,p}$ (which embeds in $\mC^\zeta_p$, so finding a bound in the latter space is sufficient). Let us start by computing:
  \begin{align*}
  \| \xi^n \|_{B^\zeta_{p,p}} & \simeq \| \xi^n \|_{L^p} + \bigg( \int_{\TT^d} \int_{\TT^d} \frac{|\xi^n(x) - \xi^n(y)|^p}{|x-y|^{d+\zeta p}} \ud x \ud y\bigg)^{\frac 1 p} \\
  & \lesssim n^{\frac d 2} + \bigg( \sum_{z \in \znd \cap \TT^d} \int_{Q_n(z)} \int_{\TT^d} \frac{|\xi^n(x) - \xi^n(y)|^p}{|x-y|^{d+\zeta p}} \ud x \ud y\bigg)^{\frac 1 p} \\
  & \lesssim n^{\frac d 2} + \bigg( \sum_{z \in \znd \cap \TT^d} \int_{Q_n(z)} \int_{\TT^d \setminus Q_n(z)} \frac{|\xi^n(x) - \xi^n(y)|^p}{|x-y|^{d+\zeta p}} \ud x \ud y\bigg)^{\frac 1 p} \\
  & \lesssim n^{\frac d 2} + n^{\frac{d}{2}} n^{\frac d p} \bigg( \int_{Q_n(0)} \int_{\TT^d \setminus Q_n(0)} \frac{1}{|x-y|^{d+\zeta p}} \ud x \ud y\bigg)^{\frac 1 p},
  \end{align*}
  where we have used in the last step that $\|\xi^n\|_\infty \lesssim n^{\frac d 2}$.
  Now we can follow the same calculations as in the proof of Lemma~\ref{lem:besov_regularity_characteristic_function}  to obtain
  \begin{align*}
  \bigg( \int_{Q_n(0)} \int_{\TT^d \setminus Q_n(0)} \frac{1}{|x-y|^{d+\zeta p}} \ud x \ud y\bigg)^{\frac 1 p} \lesssim n^{\zeta - \frac d p}.
  \end{align*}
  Hence, overall for any $p \in [1, \infty)$ and $\zeta \in [0, 1/p)$:
  \begin{align*}
  \| \xi^n \|_{\mC^\zeta_p} \lesssim n^{\frac d 2 + \zeta}.
  \end{align*}

  \textit{Step 2: Bounds for \( X_{n, \lambda} \).} As for \(X_{n, \lambda}\), 
  we need to bound \(n\| \mQ_{n} X_{n, \lambda} \|_{L^{\infty}}\). Here:
  \begin{equation} 
  \label{eqn:stochastic_est_small_parts_bound_I}
  \begin{aligned}
    \| \mQ_{n} X_{n, \lambda} \|_{L^{\infty}(\TT^{d}) }  & = \| \mF^{{-} 1}_{\TT^{d}}
    [ (1 {-} \daleth ) (n^{-1} \cdot) (-\vt_{n}+\lambda)^{{-} 1} (\cdot) \widehat{\xi}_{n}( \cdot)]
    \|_{L^{\infty}(\TT^{d})} \\
     & \leq \| \mF_{\TT^{d}}^{-1} [ (1 {-} \daleth)
    (n^{-1} \cdot) (-\vt_{n} +\lambda)^{{-} 1} (\cdot) ] \|_{L^{1}(\TT^{d})} \| \xi^{n}
    \|_{L^{\infty}(\TT^{d})} \\
    & \lesssim n^{-2}\| \mF_{\RR^{d}}^{-1} [ (1 {-} \daleth) (n^{-1} \cdot) (
    -\widehat{\chi}^{2} + 1 + n^{-2}\lambda)^{-1} ( n^{-1} \cdot)]\|_{L^{1}(\RR^{d})} \|
    \xi^{n} \|_{L^{\infty}(\TT^{d})} 
  \end{aligned}
  \end{equation}
  where we applied the Poisson summation formula of
  Lemma~\ref{lem:poisson_summation_formula}.
Note that 
  \begin{align*}
  \| \mF_{\RR^{d}}^{-1} &[ (1 {-} \daleth)  (n^{-1} \cdot) ( -\widehat{\chi}^{2} + 1
  + n^{-2} \lambda)^{-1} ( n^{-1} \cdot)]\|_{L^{1}(\RR^{d})} \\ 
   & \leq \bigg\| \mF_{\RR^{d}}^{-1} \bigg[  \frac{1 - \daleth (n^{-1} \cdot)}{1 +
     n^{-2}\lambda} + (1 {-} \daleth) (n^{-1} \cdot) \Big[ \frac{1}{ -\widehat{\chi}^2 + 1
   + n^{-2} \lambda } - \frac{1}{1 + n^{-2} \lambda} \Big]( n^{-1} \cdot)\bigg]
   \bigg\|_{L^{1}(\RR^{d})} \\
   & \leq \bigg\| \mF_{\RR^{d}}^{-1} \bigg[  \frac{1 - \daleth (n^{-1} \cdot)}{1 +
   n^{-2} \lambda} \bigg] \bigg\|_{L^{1}(\RR^{d})} +  \bigg\| \mF_{\RR^{d}}^{-1} \bigg[ (1 {-}
     \daleth) (n^{-1} \cdot) \Big[ \frac{1}{ -\widehat{\chi}^2 + 1 +
   n^{-2} \lambda } - \frac{1}{1 + n^{-2} \lambda} \Big]( n^{-1} \cdot)\bigg]
   \bigg\|_{L^{1}(\RR^{d})} 
  \end{align*}
  The first summand is bounded in $L^{1}(\RR^{d})$ uniformly over
  $n$ and \( \lambda \) (with some abuse of notation for the Dirac \(\delta\) function). As for
  the second summand observe that, for some \(c>0\):
  \begin{align*}
  \bigg\| \mF_{\RR^{d}}^{-1} & \bigg[ (1 {-} \daleth) (n^{-1} \cdot) \Big[ \frac{1}{
    -\widehat{\chi}^2 + 1 + n^{-2}\lambda } - \frac{1}{1 + n^{-2}\lambda} \Big]( n^{-1}
\cdot)\bigg]\bigg\|_{L^{1}(\RR^{d})} \\
& \leq \sup_{x\in \RR^{d}}(1+|x|^{2})^{\frac{d+1}{2} }  \bigg\vert \int_{\RR^{d}} 
e^{2\pi \iota \langle x,k \rangle} (1-\daleth(k))\bigg[ \frac{1}{ -\widehat{\chi}^2(k) + 1
+ n^{-2} \lambda } - \frac{1}{1 + n^{-2}\lambda} \bigg] \mathrm{d} k\bigg\vert \\ 
& \lesssim \sum_{ 0 \leq | \alpha | \leq 2d} \int_{\RR^{d}} \bigg\vert
  \bigg[ \partial^{\alpha}\bigg(\frac{1}{ -\widehat{\chi}^2(k) + 1 + n^{-2} \lambda } -
  \frac{1}{1 + n^{-2} \lambda} \bigg) \bigg] 1_{\{|k| \geq c\}}\bigg\vert \ud k,
  \end{align*}
  where with the sum we indicate all partial derivatives up to order
  \(2d\). Now this term can be bounded by
  Lemma~\ref{lem:taylor_expansion_and_decay_ft_characteristic_function}. Let us
  show this for \(\alpha= 0\) (the other cases are similar), where by a Taylor
  expansion:
  \begin{align*}
     \int_{\RR^{d}} \bigg\vert \frac{1}{
     -\widehat{\chi}^2(k) + 1 + n^{-2} \lambda } - \frac{1}{1 + n^{-2}
     \lambda} \bigg\vert 1_{\{|k| \geq c\}}\ud k & \lesssim_{c} \bigg(\frac{1}{1 +
     n^{-2} \lambda}\bigg)^{2} \int_{\RR^{d}} \hat{\chi}^{2}(k)
     1_{\{|k| \geq c\}}\ud k \\
     & \lesssim \int_{\RR^{d}} \frac{1}{1 + |k|^{d +1}}  \ud k < \infty.
  \end{align*}
Combining the last two observations with \eqref{eqn:stochastic_est_small_parts_bound_I} leads to
\begin{align*}
  \sup_{\lambda \geqslant 1} \| \mQ_{n} X_{n, \lambda} \|_{L^{\infty}(\TT^{d}) } \lesssim n^{-2} \| \xi^{n}
  \|_{L^{\infty}(\TT^{d})} \lesssim n^{-2 + \frac{d}{2} },
\end{align*}  
which is of the required order.

\textit{Step 3: Bounds on \( \xi^{n} \reso \Pi_{n}^{2} X_{n , \lambda} \).} We now consider the bound on \(\xi^{n} \reso
\Pi_{n}^{2} X_{n, \lambda}\), starting with \( \lambda =1 \): at the end of
this step we explain how to obtain a bound uniformly over \( \lambda \) at the
cost of a small explosion in \( \lambda \). In this computation it is important to note
that \(d=2\). 

Define $\psi_0(k_1, k_2) $  and $\widehat{\xi}_n(k)$ as
$$\psi_0(k_1, k_2) := \sum_{|i-j|\leq 1} \vr_i(k_1)\vr_j(k_2), \qquad  \widehat{\xi}_n(k):=\mF_{\TT^d} \xi^n (k).$$
Then
\begin{align*}
\EE \big[ \widehat{\xi}_n(k_1) \widehat{\xi}_n(k_2)\big]  & =
\int_{(\TT^2)^2} e^{-2 \pi \iota (k_1 \cdot x_1+k_2 \cdot x_2)}
\chi_{Q_n(x_1)} (x_2) \ud x_1 \ud x_2 \\
 &= \int_{\TT^2} e^{-2 \pi \iota (k_1 +k_2) \cdot x_1} \widehat{\chi}_Q (n^{-1} k_2) \ud x_1
 =  \widehat{\chi}_Q(n^{-1} k_1)1_{\{ k_1 + k_2 =0\}}.
\end{align*}
Hence to compute the renormalisation constant observe that 
\begin{align*}
c_{n} = \EE \big[ \xi^{n} \reso \Pi_n^2 X_{n, 1} (x) \big] & = \int_{(\ZZ^2)^2}
e^{2 \pi \iota (k_1+k_2) \cdot x} \psi_0(k_1, k_2)  \frac{\widehat{\chi}^2(
n^{-1} k_2)}{-\vt_n (k_2) + 1}  \EE \big[ \widehat{\xi}_n (k_1) \widehat{\xi}_n (k_2) \big] \ \ud k_1 \ud k_2 \\
& = \int_{\ZZ^2}  \frac{\widehat{\chi}^2( n^{-1} k) \widehat{\chi}_Q(n^{-1}
k)}{-\vt_n (k) + 1} \ \ud k.\\
\end{align*}
A similar calculation shows that actually
$
c_n = \EE \big[ \xi^n \Pi_{n}^{2} X_{n, 1}\big]
$
and the asymptotic $c_{n} \simeq \log{n}$ follows from a manipulation of the
sum.

We turn our attention to a bound for \(\| \xi^{n} \diamond
\Pi_{n}^{2}X_{n, 1} \|_{\mC^{- \frac{\kappa}{2} }}\).
As before, for $p \geq 2$, consider
\begin{equation}
\label{eqn:stochastic_estimate_besov_norm}
\begin{aligned}
  \EE \| \xi^{n} \reso X_{n, 1} {-} c_{n} \|_{B^{\alpha}_{p,p}}^{p}
  & = \sum_{j \geq {-} 1} 2^{\alpha j p} \EE \| \Delta_{j} ( \xi^{n} \reso
  X_{n, 1} {-} c_{n} 1_{j = {-} 1}) \|_{L^{p}(\TT^{d})}^{p} \\
  & =  \sum_{j \geq {-} 1} 2^{\alpha j p} \int_{\TT^{d}} \EE |\Delta_{j}
  (\xi^{n} \reso X_{n, 1} {-} c_{n} 1_{j = {-} 1} )|^{p} (x) \ud x.
\end{aligned}
\end{equation}
It is now convenient to introduce the notation:
\begin{align*}
\mathcal{K}^{n}_m(x)= \mF_{\TT^2}^{-1} \bigg( \vr_m(\cdot) \frac{ \widehat{\chi}^2(n^{-1} \cdot )}{-\vt_n( \cdot) + 1}\bigg)(x).
\end{align*}
Then the integrand in \eqref{eqn:stochastic_estimate_besov_norm} can be written as
\begin{equation}
\label{eqn:stochastic_estimate_besov_bound_II}
\begin{aligned}
 \EE \big[ & |\Delta_j (\xi^{n} \reso \Pi_n^2 X_{n, 1})(x) - c_n
 1_{\{j=-1\}}|^p\big]  \\
 & = \EE \big[ |\Delta_j (\xi^{n} \reso \Pi_n^2
 X_{n, 1})(x)  - \EE\Delta_j (\xi^{n} \reso \Pi_n^2 X_{n,
 1})(x) |^p\big] \\
 & = \EE \bigg\vert \int_{\TT^2} K_j(x-y) \sum_{|l-m|\leq 1} \bigg(
 \int_{(\TT^2)^2} K_l(y-z_1) \mathcal{K}^{n}_m(y-z_2) \xi^n(z_1) \diamond
 \xi^n(z_2) \ud z_1 \ud z_2 \bigg)\ud y\bigg\vert^p,
\end{aligned}
\end{equation}
where, conveniently:
\begin{align*}
\xi^n(z_1) \diamond \xi^n (z_2) = \xi^n (z_1) \xi^n(z_2) - \EE \big[ \xi^n (z_1) \xi^n(z_2) \big].
\end{align*}
Now we can write \eqref{eqn:stochastic_estimate_besov_bound_II} as a discrete
stochastic integral and apply \cite[Lemma 5.1]{MartinPerkowski2019Bravais} to obtain
\begin{align*}
 \EE & \bigg\vert\!\int_{\TT^2}\! K_j(x{-}y) \sum_{|l-m|\leq 1}\!\sum_{x_1, x_2 \in
 \ZZ^2_n \cap \TT^2} \bigg( \int\limits_{Q_n(x_1) \times Q_n(x_2)}\! K_l(y-z_1)
 \mathcal{K}_m^{n}(y{-}z_2) \! \ud z_1 \ud z_2 \bigg) \xi^n(x_1) \diamond \xi^n(x_2) \ud y\bigg\vert^p \\
  & \lesssim \bigg[ \sum_{x_1, x_2 \in \ZZ^2_n \cap \TT^2} n^{-2d}
  \bigg\vert \int_{\TT^2} K_j(x{-}y) \sum_{|l-m|\leq 1}  \bigg(
\mint{-}\limits_{Q_n(x_1) \times Q_n(x_2)} K_l(y-z_1)
\mathcal{K}_m^{n}(y-z_2)\!  \ud z_1\! \ud z_2 \bigg) \!\ud y \!\bigg\vert^2
\bigg]^{p/2}\\
  &=\bigg[ \sum_{x_1, x_2 \in \ZZ^2_n \cap \TT^2} n^{-2d} \bigg\vert
  \mint{-}\limits_{Q_n(x_1) \times Q_n(x_2)} \int_{\TT^2} K_j(x-y)
\sum_{|l-m|\leq 1}  K_l(y-z_1) \mathcal{K}_m^{n}(y-z_2)  \ud y \ud z_1 \ud z_2
\bigg\vert^2 \bigg]^{p/2}\\
  &\leq \bigg[ \int_{(\TT^2)^2}\bigg\vert \int_{\TT^2} K_j(x-y) \sum_{|l-m|\leq
  1}  K_l(y-z_1) \mathcal{K}^{n}_m(y-z_2)  \ud y \bigg\vert^2 \ud z_1 \ud z_2
\bigg]^{p/2},
\end{align*}
where the last step is an application of Jensen's inequality. Now, via  Parseval's Theorem, the latter is bounded by
\begin{multline*}
\bigg[ \int_{(\ZZ^2)^2} \bigg\vert \int_{\TT^2} K_j(x-y) \sum_{|l-m|\leq 1}  e^{2 \pi \iota k_1 \cdot y} \vr_l (k_1) e^{2 \pi \iota k_2 \cdot y}\vr_m (k_2) \frac{\widehat{\chi}^2(n^{-1} k_2 )}{-\vt_n( k_2) + 1} \ud y \bigg\vert^2 \ud k_1 \ud k_2 \bigg]^{p/2}\\
 = \bigg[ \int_{(\ZZ^2)^2} \bigg\vert e^{2 \pi \iota (k_1 +k_2) \cdot x } \vr_j (k_1+k_2) \psi_0 (k_1, k_2) \frac{\widehat{\chi}^2(n^{-1} k_2 )}{-\vt_n( k_2) + 1}  \bigg\vert^2 \ud k_1 \ud k_2 \bigg]^{p/2}.
\end{multline*}
By Lemma~\ref{lem:taylor_expansion_and_decay_ft_characteristic_function}:
\begin{align*}
\frac{\widehat{\chi}^2(n^{-1} k)}{-\vt_n (k) + 1} \lesssim \frac{\widehat{\chi}^2(n^{-1} k)}{|k|^2 + 1} 1_{\{|k \lesssim n|\}} + \frac{|k|^{-3}}{1} 1_{\{|k|\gtrsim n\}}  \lesssim \frac{1}{1 + |k|^2}.
\end{align*}
Finally, taking into account the supports of the functions,
\begin{align*}
\bigg[ \int_{(\ZZ^2)^2} \bigg\vert \vr_j (k_1+k_2) \psi_0 (k_1, k_2) \frac{1}{1+|k_2|^2}  \bigg\vert^2 \ud k_1 \ud k_2 \bigg]^{p/2} \lesssim \Big[ 2^{j2d}2^{-4j} \Big]^{p/2}\leq 1,
\end{align*}
which provides a bound of the required order. This concludes the proof of the
required bound in the case \( \lambda =1 \). For general \( \lambda \geqslant 1
\) we observe that
\begin{align*}
\xi^{n} \reso \Pi_{n}^{2} X_{n, \lambda} - c_{n} = \xi^{n} \reso
\Pi_{n}^{2} (X_{n , \lambda} - X_{n, 1}) + \xi^{n} \reso \Pi_{n}^{2} X_{n,
\lambda} - c_{n}.
\end{align*}

To complete the proof of our result it now suffices to show that
\begin{align*}
\EE \sup_{\lambda \geqslant 1} \lambda^{-\frac{\kappa}{4}} \| \xi^{n} \reso
\Pi_{n}^{2} ( X_{n , \lambda} - X_{n, 1}) \|_{\mC^{ - \frac{\kappa}{2} }} < \infty.
\end{align*}
For this purpose we observe that by a resolvent identity:
\begin{align*}
X_{n, \lambda} - X_{n, 1} & = \Big[ (- \mA_{n} + \lambda)^{-1} - (-
\mA_{n} + 1)^{-1}  \Big]\xi^{n} \\
& = (1 - \lambda) ( - \mA_{n} + \lambda)^{-1} (- \mA_{n} + 1)^{-1}
\xi^{n}.
\end{align*}
Now we can apply the elliptic Schauder estimates of
Proposition~\ref{prop:elliptic_schauder_estimates} to obtain:
\begin{align*}
\lambda^{-\frac{\kappa}{4} } \| \Pi_{n}^{2} \mP_{n} (X_{n, \lambda} - X_{n, 1})
\|_{\mC^{1 + \frac{\kappa}{4}}} & \lesssim \lambda^{1 - \frac{\kappa}{4}} \|
 ( - \mA_{n} + \lambda)^{-1} X_{n, 1}\|_{\mC^{1 + \frac{\kappa}{4}}} \\
& \lesssim \| X_{n, 1} \|_{\mC^{1 - \frac{\kappa}{4}}}\\
& \lesssim \| \xi^{n} \|_{\mC^{-1 - \frac{\kappa}{4}}}.
\end{align*}
And on small scales, using the regularizing properties of \(
\Pi_{n}^{2} \) from
Corollary~\ref{cor:regularity_gain_convolution_with_characteristic_functions}:
\begin{align*}
\lambda^{- \frac{\kappa}{4}} \| \Pi_{n}^{2}\mQ_{n}  (X_{n, \lambda} - X_{n, 1})
\|_{\mC^{1 + \frac{\kappa}{4}}} & \lesssim \lambda^{- \frac{\kappa}{4}}
n^{2 - \frac{\kappa}{4}} \| \mQ_{n} (X_{n, \lambda} - X_{n, 1})\|_{\mC^{-1 +
\frac{\kappa}{2}}} \\
& \lesssim \lambda^{1 - \frac{\kappa}{4}} n^{2 - \frac{\kappa}{4}} \|
\mQ_{n} (- \mA_{n} + \lambda)^{-1} X_{n, 1} \|_{\mC^{-1 + \frac{\kappa}{2}
}}\\
& \lesssim n^{2 - \frac{3 \kappa}{4}} \| \mQ_{n - k_{0}} X_{n, 1}
\|_{\mC^{-1 + \frac{\kappa}{2}}} \\
& \lesssim n^{- \frac{3 \kappa}{4}} \| \xi^{n} \|_{\mC^{-1 +
\frac{\kappa}{2}}}.
\end{align*}
Here we have chosen a deterministic \( k_{0} \in \NN \) (uniformly over \( n
\)) such that \( \mQ_{n} \mQ_{n - k_{0}} = \mQ_{n} \).
Hence overall we obtain:
\begin{align*}
\EE \sup_{\lambda \geqslant 1} \lambda^{- \frac{\kappa}{4}} \| \xi^{n} \reso \Pi_{n}^{2} (X_{n, \lambda} -
X_{n, 1}) \|_{\mC^{- \frac{\kappa}{2} }} & \lesssim \EE \sup_{\lambda \geqslant 1}
\lambda^{- \frac{\kappa}{4}} \| \xi^{n} \reso \Pi_{n}^{2} (X_{n, \lambda} -
X_{n, 1}) \|_{\mC^{\frac{\kappa}{8}}} \\
&\lesssim \EE \| \xi^{n} \|_{\mC^{ -1 -
\frac{\kappa}{8}}} \sup_{\lambda \geqslant 1} \lambda^{- \frac{\kappa}{4}
}\| \Pi_{n}^{2}(X_{n, \lambda} - X_{n, 1}) \|_{\mC^{1 + \frac{\kappa}{4}}}\\
 & \lesssim \EE  \left[ \| \xi^{n} \|_{\mC^{-1 - \frac{\kappa}{8}}} \big( \|
\xi^{n} \|_{\mC^{-1 - \frac{\kappa}{4}}} + n^{- \frac{\kappa}{2} -
\frac{\kappa}{4}} \| \xi^{n} \|_{\mC^{- 1 + \frac{\kappa}{2}}} \big)\right] \\
& \lesssim 1,
\end{align*}
where the last average is bounded by the same arguments presented in Step \( 1
\) (up to changing \( \kappa \)).
With this we have concluded the
proof of the regularity bound. We are left with a discussion of the
convergence.

\textit{Step 4.}
What we established so far implies tightness of the following sequences of
random variables in their respective spaces:
$$\xi^n \in \mC^{-\frac d 2-\kappa}, \qquad \mP_{n} X_{n, \lambda} \in
\mC^{1- \kappa}, \qquad \xi^n \diamond \Pi_{n}^{2} X_{n, \lambda} \in \mC^{-\kappa}.$$
The next step is to show that the limiting points of \(\xi^{n}\) and
\(\xi^{n} \diamond \Pi_{n}^{2} X_{n, \lambda}\) are unique in
distribution. In particular,  in view of
Proposition~\ref{prop:elliptic_schauder_estimates}, this
would imply weak convergence also of $ \mP_{n} X_{n, \lambda}$. In the last
step we will address the almost sure convergence and the almost sure uniform
bound.

Convergence of $\xi^{n}$ to space time white noise $\xi$ is an instance of
central limit theorem (notice the normalization of variance in
Assumption~\ref{ass:selection_coefficient}). We therefore focus our attention
on the more involved Wick product $\xi^n \diamond X_{n, \lambda}$. Now, the
deterministic bounds at the end of Step \( 3 \) show that the convergences 
\[ \xi^{n} \to \xi \quad \text{ in } \quad \mC^{-1- \kappa}, \qquad
 \xi^{n} \diamond \Pi_{n}^{2} X_{n,1} \to \xi \diamond X_{1} \quad \text{ in } \quad \mC^{ - \kappa}\]
for any \( \kappa > 0 \) imply also the convergence of \( \xi^{n} \diamond
\Pi_{n}^{2} X_{n, \lambda} \) for general \( \lambda \geqslant 1 \). Hence we
can restrict to discussing the case \( \lambda = 1 \).
For fixed $\varphi \in \mS(\TT^2)$
\begin{align*}
\langle  \varphi & , \xi^n \diamond X_{n, 1}\rangle \\
 & = \int_{\TT^2} \varphi (y) \sum_{|l-m|\leq 1}  \sum_{x_1, x_2 \in \ZZ^2_n
 \cap \TT^2} \bigg( \int\limits_{Q_n(x_1) \times Q_n(x_2)} K_l(y-z_1)
 \mathcal{K}_m^{n}(y{-}z_2)  \ud z_1 \ud z_2 \bigg) \xi^n(x_1) \diamond
 \xi^n(x_2) \ud y \\
 & =   \sum_{x_1, x_2 \in \ZZ^2_n} \big\langle \varphi ( \cdot ) ,
 \sum_{|l-m|\leq 1} \Pi_n K_l( \cdot - x_1) \Pi_n \mathcal{K}_m^{n}(
 \cdot {-}x_2) \big\rangle \xi^n(x_1) \diamond \xi^n(x_2). 
\end{align*}
Consider a map $L_{n} : \big(\ZZ^{2}_{n}\big)^2 \to \RR$ defined by
\begin{align*}
L_n (x_1, x_2) := \langle \varphi ( \cdot ) , \sum_{|l-m|\leq 1} \Pi_n^Q K_l( \cdot - x_1) \Pi_n^Q \mathcal{K}_m( \cdot {-}x_2) \rangle 1_{\{(x_1, x_2)\in \TT^2 \times \TT^2\}}.
\end{align*}
This definition naturally extends to $n=\infty$, where $L$ maps $(\RR^2)^2$ to $\RR$.
Our goal is to show that
\begin{equation}\label{eqn:proof_sotchastic_estimates_covergence_second_chaos}
\sum_{(x_1, x_2)\in (\ZZ^2_n)^2} L_n(x_1, x_2) \xi^n(x_1)\diamond \xi^n(x_2) \to \int_{(\RR^2)^2} L(x_1, x_2) \xi( \ud x_1)\diamond \xi(\ud x_2),
\end{equation}
where convergence holds in distribution and  the limit is interpreted as an iterated stochastic integral in the second Wiener-It\^o chaos.
It is sufficient to verify the assumptions of \cite[Lemma 5.4]{MartinPerkowski2019Bravais}.
That is, we have to show that there exists a $g\in L^2((\RR^2)^2)$ such that:
\begin{align*}
\sup_{n \in \NN} |1_{(n\TT^2)^2} \mF_{(\ZZ^2_{n})^2} L_n | \leq
g, \qquad \lim_{n \to \infty} \| 1_{(n\TT^2)^2} \mF_{(\ZZ^2_n)^2} L_n -
\mF_{(\RR^2)^2} L\|_{L^{2}( (\RR^{2})^{2})} = 0
\end{align*}
For this purpose we calculate
\begin{align*}
& 1_{(n\TT^2)^2}  \mF_{(\ZZ^2_n)^2} L_n (k_1, k_2) \\
& = 1_{(n\TT^2)^2}(k_1, k_2)  \int\limits_{(\ZZ^2_n\cap \TT^2)^2} e^{2 \pi \iota (k_1 \cdot x_1 + k_2 \cdot x_2)}\langle \varphi ( \cdot ) , \sum_{|l-m|\leq 1} \Pi_n^Q K_l( \cdot - x_1) \Pi_n^Q \mathcal{K}_m( \cdot {-}x_2) \rangle  \ud x_1 \ud x_2\\
& = 1_{(n\TT^2)^2}(k_1, k_2)  \int\limits_{(\TT^2)^2} e^{2 \pi \iota (k_1 \cdot x_1 + k_2 \cdot x_2)}\langle \varphi ( \cdot ) , \sum_{|l-m|\leq 1} K_l( \cdot - x_1) \mathcal{K}_m( \cdot {-}x_2) \rangle  \ud x_1 \ud x_2 \\
& = 1_{(n\TT^2)^2}(k_1, k_2)  \int_{\TT^2} \varphi (y) e^{2 \pi \iota (k_1 + k_2) \cdot y} \sum_{|l-m|\leq 1} \vr_l( - k_1) \vr_m(-k_2) \frac{ \widehat{\chi}^2(-n^{-1} k_2)}{-\vt_n( -k_2) + 1} \ud y\\
& = 1_{(n\TT^2)^2}(k_1, k_2) ( \mF_{\TT^2} \varphi) (k_1 + k_2)
\sum_{|l-m|\leq 1} \vr_l( k_1) \vr_m(k_2) \frac{ \widehat{\chi}^2(n^{-1}
k_2)}{-\vt_n( k_2) + 1},
\end{align*}
so that the required assumptions are naturally satisfied. 
Since $\varphi$ is smooth, the latter term is bounded in $L^2$, uniformly over
$n$. In particular
\eqref{eqn:proof_sotchastic_estimates_covergence_second_chaos} follows. Hence
the distribution of any limit point of $\langle \varphi, \xi^n \diamond
\Pi_{n}^{2} X_{n, 1}\rangle$ is uniquely characterized and since $\varphi$ is arbitrary
this implies convergence in distribution of $\xi^n \diamond
\Pi_{n}^{2} X_{n, 1}$.

\textit{Step 5.} Above we have proven that \( \xi^{n} \) and \( \xi^{n} \diamond
\Pi_{n}^{2} X_{n, \lambda}\) converge in distribution in \( \mC^{-1-\kappa}  \)
and \( \mC^{- \kappa} \) respectively. Now let us prove almost sure convergence
up to changing probability space (we discuss only the case of \(\xi^{n} \),
since the other term can be treated similarly). We would like to apply
Skorohod's representation theorem, which requires the underlying space to be
separable. Unfortunately the space \( \mC^{-1- \kappa} \) is not separable, but we can embed
\[ \mC^{-1-\kappa} \subseteq B^{-1- \kappa}_{p(\kappa),p(\kappa)} \subseteq
\mC^{-1- 2 \kappa}  \] 
for some \( p(\kappa) \in (1, \infty) \) sufficiently large. Now the space \(
B^{-1 - \kappa}_{p(\kappa), p(\kappa)} \) is separable, so we can apply
Skorohod's representation theorem to obtain almost sure convergence in \(
\mC^{-1-2 \kappa} \). Since \( \kappa \) is arbitrary this is sufficient for
the required result. \\
The last statement we have to prove is that in this new probability space
(that we call \( ( \overline{\Omega}, \overline{\mF}, \overline{\PP}) \)) we
have a uniform bound for almost all \( \omega \in \overline{\Omega}\):
\[ \sup_{n \in \NN} \opnorm{ \bx_{n}(\omega)}_{n, \kappa} < \infty. \] 
Recall that
\begin{align*}
  \opnorm{ \boldsymbol{\xi}_{n}(\omega) }_{n, \kappa}: =  & \sup_{\zeta \in [0,1]
}\left\{  n^{-\zeta}   \|\xi^n\|_{\mC^{- (1-\zeta) -\frac \kappa 2}}\right\}   +
n^{-1}\|\xi^n\|_{L^\infty} + n^{-1- \kappa}\| \xi^n\|_{\mC^{\kappa}_{\frac{1}{2
\kappa}}} \\
 &+ \sup_{\lambda \geqslant 1} \left\{ n \|\mQ_n X_{n, \lambda}\|_{L^\infty} +
\lambda^{-\frac{\kappa}{4}}\| Y_{n, \lambda}
\|_{\mC^{-\frac \kappa 2}} \right\}. 
\end{align*}
Now following Steps \( 1 \) and \( 2 \) we see that the bounds on \( \| \xi^{n} (\omega)
\|_{\infty},  \| \xi^{n} (\omega) \|_{\mC^{\kappa}_{\frac{1}{2 \kappa} }} \) and
\( \| X_{n, \lambda}(\omega) \|_{\infty} \) depend only on the deterministic bound \(
| \xi^{n} (\omega, x) | \leqslant 2 n \) (in \( d=2 \)), so we are left with
proving:
\begin{align*}
  \sup_{n \in \NN} &  \left\{ \sup_{\zeta \in [0,1] } n^{-\zeta}
    \|\xi^n(\omega)\|_{\mC^{- (1-\zeta) - \frac{\kappa}{2} }} + \sup_{\lambda \geqslant 1}
    \lambda^{-\frac{\kappa}{4}}\| \xi^{n} \diamond \Pi_{n}^{2} X_{n, \lambda}(\omega)
\|_{\mC^{-\frac \kappa 2}} \right\} \\
& \lesssim  \sup_{n \in \NN}  \Big\{ \sup_{\zeta \in [0,1] }(n^{-1} \|
  \xi^{n} (\omega) \|_{\infty} )^{\zeta} \| \xi^{n}(\omega) \|_{\mC^{-1 -
\frac{\kappa}{2(1- \zeta)} }}^{1 - \zeta}  + \| \xi^{n} \diamond
\Pi_{n}^{2} X_{n ,1} (\omega) \|_{\mC^{- \frac{\kappa}{2} }}\\
& \qquad \qquad \qquad  + \sup_{\lambda \geqslant 1} \lambda^{- \frac{\kappa}{4}} \| \xi^{n} \reso \Pi_{n}^{2} (X_{n, \lambda} -
X_{n, 1}) (\omega) \|_{\mC^{- \frac{\kappa}{2} }} \Big\} \\
& \lesssim \sup_{n \in \NN} \Big\{ 1 + \| \xi^{n}(\omega) \|_{\mC^{-1 -
\frac{\kappa}{2} }} + \| \xi^{n} \diamond
\Pi_{n}^{2} X_{n ,1} (\omega) \|_{\mC^{- \frac{\kappa}{2} }}\\
& \qquad \qquad \qquad  + \| \xi^{n}(\omega) \|_{\mC^{-1 - \frac{\kappa}{8}}} \big( \|
  \xi^{n}(\omega) \|_{\mC^{-1 - \frac{\kappa}{4}}} + n^{- \frac{\kappa}{2} -
\frac{\kappa}{4}} \| \xi^{n}(\omega) \|_{\mC^{- 1 + \frac{\kappa}{2}}} \big)
\Big\},
\end{align*}
where we used interpolation for the first term, as in Step \( 1 \), and the
same bounds as in Step \( 3 \) for the last term. In particular now the
uniform bound is a consequence of the convergence of \( \xi^{n} \) and \(
\xi^{n} \diamond \Pi_{n}^{2} X_{n, 1} \) in the correct spaces.
\end{proof}

\subsection{Proof of
Lemma~\ref{lem:taylor_expansion_and_decay_ft_characteristic_function}}\label{app:prf-of-lemma}

\begin{proof}[Proof of
  Lemma~\ref{lem:taylor_expansion_and_decay_ft_characteristic_function}]
  Let us start with the term involving the gradient. We have that for \(i = 1, \ldots, d\):
  \[ (D \widehat{\chi})_i (0) = {-} 2 \pi \iota \mint{-}_{B_{1}(0)} x_i
  e^{{-} 2 \pi \iota \langle k, x \rangle} \ud x \Big\vert_{k = 0} = 0.\]
  For the term involving the Hessian, we observe that an analogous computation for
  \(i \neq j\) shows that    
\( (D^2 \widehat{\chi})_{i, j} (0) = 0\) . If \(i =j\) we find that
  \[ 
    (D^2 \widehat{\chi})_{i, i} (0) = {-} (2 \pi)^2 \mint{-}_{B_{1}(0)} x_i^2
    e^{{-} 2 \pi \iota \langle k, x \rangle} \ud x \Big\vert_{k = 0} = \colon
\frac{- (2 \pi)^2}{4} \nu_0,
  \]
  with the value of $\nu_{0}$ as in the statement. 
  The two-sided inequality follows by a Taylor approximation.  

We are left with a bound on the decay of $\hat{\chi}$:
  \begin{equation*}
     \Big\vert  
     \frac{\ud^n}{\ud x_{i_1} \dots \ud x_{i_n}} \hat{\chi}_B(k) \Big\vert \lesssim (1 {+} |k|)^{{-} \frac{d {+} 1}{2}}. 
  \end{equation*}
  For this purpose let \(J_{\nu}( \cdot)\) be the Bessel function of the first kind with  parameter \(\nu\), that is
\begin{align*}
 J_{\nu}(k) = \sum_{m=0}^\infty \frac{(-1)^m}{m! \Gamma(m+\nu+1)} {\left(\frac{k}{2}\right)}^{2m+\nu}.
\end{align*}
The Fourier transform of $\chi_B$ can be written, for some $c, C>0$, as
\begin{align}\label{eqn:FT_of_ch_funct}
\hat{\chi}_B (k) = c(d) \int_0^{ \pi} 
  \sin^d {(t)} e^{{-} 2 \pi \iota |k| \cos{(t)/4} } \ \ud t = C|k|^{{-} d/2}
  J_{d/2}(\pi |k| /2).
\end{align}
In the last step we used one of the alternative representations of Bessel
functions, see e.g. \cite[Section 6.15, Equation (5)]{Watson1995Bessel}
(the author uses the notation \( K_{n} \) for the real part of \( J_{\nu} \),
but in our case the Bessel function is real valued). Since $J_{\frac{1}{2}}(k)  = \sqrt{\frac{2}{\pi k}} \sin k$, the bound for $d = 1$ is immediate.
For $d=2$, we make use of an asymptotic bound for Bessel functions:
 \[ \sup_{\varrho \geq 1} \varrho^{1/2} |J_{\nu}(\varrho)| < {+} \infty.\]
 We provide a proof of this bound in the next Lemma. 
The bound for the derivatives then
 follows from \eqref{eqn:FT_of_ch_funct}, the asymptotic result for Bessel functions, and the following pair of identities
\begin{align*}
 \partial_{x} J_{n}(x) = \hh ( J_{n {-} 1}(x) {+} J_{n {+} 1}(x)), \qquad
  \forall n \in \ZZ, 
  \\
  J_{ {-} n} ( \cdot) = ({-} 1)^{n}J_{n}( \cdot)  \qquad \forall n \in \NN_0.
\end{align*}
%

\end{proof}
The following result is well-known (see e.g. \cite{Watson1995Bessel}, where many deeper results are presented). For completeness we provide a proof that satisfies all our purposes.
\begin{lemma}
Fix $\nu \in \RR$. Then
 \[ \sup_{\varrho \geq 1} \varrho^{1/2} |J_{\nu}(\varrho)| < {+} \infty,\] 
\end{lemma}
 \begin{proof}
   Through \eqref{eqn:FT_of_ch_funct} and by changing variables $x=cos(t)$ we rewrite the Bessel function as
   \begin{align*}
     \int_{{-} 1}^{ 1}  (1 {-} x^2)^{\frac{d {-} 1}{2}
     } e^{ \iota \varrho x }\ud x = 2 \mathrm{Re} \bigg( \int_0^1 (1 {-} x^2)^{\frac{d {-} 1}{2} } e^{\iota \varrho x} \bigg) \ud x.
   \end{align*}
   A change variables  \(x = 1 {-} u^2.\) yields
   \begin{align*}
     e^{i \varrho } \int_0^1  \big( u^2(2 {-}
     u^2)\big)^{\frac{d {-} 1}{2} } e^{{-} \iota \varrho u^2} u \ud u =
     \frac{e^{i \varrho}}{\varrho^{\frac{d {+} 1}{2} }} \int_0^{\sqrt{\varrho}
     }  \big( w^2 (2 {-} \frac{w^2}{\varrho} )\big)^{\frac{d {-} 1}{2}
     } e^{{-} \iota w^2} w \ud w.
 \end{align*}
 Observe that in order to obtain the desired bound it is now sufficient to show that the integral terms is bounded uniformly in $\rho$.
 After another change of variable  \(w = e^{{-} \iota \frac{\pi}{4}} z\) we obtain
 \begin{multline*}
   \int_0^{e^{\frac{\iota\pi}{4} } \sqrt{\varrho} } \big( {-} \iota z^2
   (2 {+} \iota z^2 / \varrho) \big)^{\frac{d {-} 1}{2} } e^{{-} z^2} z \ud z  \\ 
   =
   \int_0^{\sqrt{\varrho} } \big( {-} \iota z^2
   (2 {+} \iota z^2 / \varrho) \big)^{\frac{d {-} 1}{2} } e^{{-} z^2} z  \ud z +
      \int_0^{\pi/4}  \big( {-} \iota \varrho e^{2 \iota \varphi}
   (2 {+} \iota e^{2 \iota \varphi} ) \big)^{\frac{d {-} 1}{2} } e^{{-} \varrho
   e^{2 \iota \varphi} } \varrho e^{2 \iota \varphi} \ud \varphi. 
 \end{multline*}
 The first integral can be trivially bounded uniformly over \(\varrho\) while
the second one is tends to $0$ as $\rho$ tends to infinity since the
exponential term dominates all the others.
 \end{proof}

\bibliographystyle{plain}

\end{document}